\documentclass[11pt,a4paper]{article}

\usepackage[british]{babel}

\usepackage[T1]{fontenc}

\usepackage[utf8]{inputenc}

\usepackage{csquotes}
\usepackage{dirtytalk}

\usepackage[backend=biber, maxbibnames=10, maxcitenames=4, maxalphanames=4, bibencoding=utf8, style=alphabetic, isbn=false, url=false, doi=true, giveninits=true]{biblatex}

\renewbibmacro{in:}{}
\renewbibmacro*{doi+eprint+url}{%
	\printfield{doi}%
	\newunit\newblock%
	\iftoggle{bbx:eprint}{%
		\usebibmacro{eprint}%
	}{}%
	\newunit\newblock%
	\iffieldundef{doi}{%
		\usebibmacro{url+urldate}}%
	{}%
}
\AtEveryBibitem{\clearfield{month}}
\AtEveryBibitem{\clearfield{note}}

\usepackage{setspace}

\usepackage{lmodern}

\usepackage[indentafter]{titlesec}
\titleformat{name=\section}{}{\thesection.}{0.8em}{\centering\scshape}
\titleformat{name=\subsection}[runin]{}{\thesubsection.}{0.5em}{\bfseries}[.]
\titleformat{name=\subsubsection}[runin]{}{\thesubsubsection.}{0.5em}{\itshape}[.]
\titleformat{name=\paragraph,numberless}[runin]{}{}{0em}{}[.]
\titlespacing{\paragraph}{0em}{0em}{0.5em}
\titleformat{name=\subparagraph,numberless}[runin]{}{}{0em}{}[.]
\titlespacing{\subparagraph}{0em}{0em}{0.5em}

\usepackage{xurl}

\usepackage[dvipsnames]{xcolor}

\usepackage{mathtools}
\usepackage{nicematrix}

\DeclarePairedDelimiter\abs{\lvert}{\rvert}%
\DeclarePairedDelimiter\norm{\lVert}{\rVert}%

\makeatletter
\let\oldabs\abs
\def\abs{\@ifstar{\oldabs}{\oldabs*}}
\let\oldnorm\norm
\def\norm{\@ifstar{\oldnorm}{\oldnorm*}}
\makeatother

\usepackage{amssymb}

\usepackage{amsthm}

\usepackage{mathrsfs}

\usepackage{dsfont}

\usepackage{enumitem}
\setlist[enumerate]{noitemsep, partopsep=0pt, topsep=0pt, parsep=0pt, itemsep=0pt}
\setlist[itemize]{noitemsep, partopsep=0pt, topsep=0pt, parsep=0pt, itemsep=0pt}

\usepackage{interval}

\intervalconfig{soft open fences}

\usepackage[stretch=10]{microtype}

\usepackage[affil-it, noblocks]{authblk}

\usepackage[draft=false]{hyperref}
\hypersetup{
	colorlinks = true, 
	urlcolor   = blue, 
	linkcolor  = blue, 
	citecolor  = ForestGreen 
}
\newsavebox{\linkunderlinebox}
\newcommand{\linkunderline}[1]{%
	\begingroup
	\sbox{\linkunderlinebox}{#1}%
	\leavevmode
	\rlap{\raisebox{-1.6pt}[0pt][0pt]{\rule{\wd\linkunderlinebox}{0.35pt}}}%
	\usebox{\linkunderlinebox}%
\endgroup
}
\let\templatehref\href
\renewcommand{\href}[2]{\templatehref{#1}{\linkunderline{#2}}}
\ExplSyntaxOn
\NewDocumentCommand{\bibbreakablehref}{mm}{%
	\tl_map_inline:nn {#2} {%
		\templatehref{#1}{\linkunderline{##1}}%
		\penalty\UrlBreakPenalty
	}%
}
\ExplSyntaxOff
\DeclareFieldFormat{doi}{%
	\mkbibacro{DOI}\addcolon\space
	\bibbreakablehref{https://doi.org/#1}{#1}}
\DeclareFieldFormat{url}{%
	\mkbibacro{URL}\addcolon\space
	\bibbreakablehref{#1}{#1}}
\renewbibmacro*{cite}{%
	\printtext[bibhyperref]{%
		\linkunderline{%
			\textbf{%
				\printfield{labelprefix}%
				\printfield{labelalpha}%
				\printfield{extraalpha}%
				\ifbool{bbx:subentry}
				{\printfield{entrysetcount}}
				{}}}}}

\usepackage{aliascnt}
\usepackage[nameinlink,noabbrev,capitalize]{cleveref}
\crefname{equation}{}{}

\usepackage[plain]{fullpage}

\newlist{theoenum}{enumerate}{1} 
\setlist[theoenum]{label=\normalfont(\roman*), ref=\theproposition~\normalfont(\roman*), noitemsep, partopsep=0pt, topsep=0pt, parsep=0pt, itemsep=0pt}
\crefalias{theoenumi}{theorem}

\usepackage{tocloft}

\usepackage{titlefoot}

\theoremstyle{plain}
\newtheorem{theorem}{Theorem}[section]
\newtheorem*{theorem*}{Theorem}

\newaliascnt{proposition}{theorem}
\newtheorem{proposition}[proposition]{Proposition}
\aliascntresetthe{proposition}
\crefname{proposition}{Proposition}{Propositions}
\Crefname{proposition}{Proposition}{Propositions}

\newaliascnt{corollary}{theorem}
\newtheorem{corollary}[corollary]{Corollary}
\aliascntresetthe{corollary}
\crefname{corollary}{Corollary}{Corollaries}
\Crefname{corollary}{Corollary}{Corollaries}

\newaliascnt{conjecture}{theorem}

\aliascntresetthe{conjecture}
\crefname{conjecture}{Conjecture}{Conjectures}
\Crefname{conjecture}{Conjecture}{Conjectures}

\newaliascnt{lemma}{theorem}
\newtheorem{lemma}[lemma]{Lemma}
\aliascntresetthe{lemma}
\crefname{lemma}{Lemma}{Lemmas}
\Crefname{lemma}{Lemma}{Lemmas}
\newtheorem*{lemma*}{Lemma}

\theoremstyle{definition}
\newaliascnt{definition}{theorem}
\newtheorem{definition}[definition]{Definition}
\aliascntresetthe{definition}
\crefname{definition}{Definition}{Definitions}
\Crefname{definition}{Definition}{Definitions}

\theoremstyle{remark}
\newaliascnt{example}{theorem}

\aliascntresetthe{example}
\crefname{example}{Example}{Examples}
\Crefname{example}{Example}{Examples}

\newaliascnt{remark}{theorem}
\newtheorem{remark}[remark]{Remark}
\aliascntresetthe{remark}
\crefname{remark}{Remark}{Remarks}
\Crefname{remark}{Remark}{Remarks}
\newtheorem*{remark*}{Remark}

\crefname{theorem}{Theorem}{Theorems}
\Crefname{theorem}{Theorem}{Theorems}

\newaliascnt{appsec}{section}
\aliascntresetthe{appsec}
\crefname{appsec}{Appendix}{Appendices}
\Crefname{appsec}{Appendix}{Appendices}

\newcommand*\diff{\mathop{}\!\mathrm{d}}
\newcommand{\R}{\mathbb{R}}
\newcommand{\heis}{\mathbb{H}}

\newcommand{\uptau}{\tau}

\renewcommand{\exp}{\mathrm{exp}}
\renewcommand{\epsilon}{\varepsilon}
\renewcommand{\subset}{\subseteq}

\DeclareMathOperator{\T}{\mathrm{T}}

\DeclareMathOperator{\sign}{sgn}
\DeclareMathOperator{\arcsinh}{arcsinh}
\DeclareMathOperator{\arccosh}{arccosh}

\let\originalleft\left
\let\originalright\right
\renewcommand{\left}{\mathopen{}\mathclose\bgroup\originalleft}
\renewcommand{\right}{\aftergroup\egroup\originalright}

\title{
  {\large\bfseries\MakeUppercase{Constant mean curvature surfaces in the sub-Lorentzian Heisenberg group}}
}

\author{
  Samu\"el Borza
  \protect\footnotemark[1]
  \protect\footnotemark[2]
  \and
  Andrea Pinamonti
  \protect\footnotemark[3]
  \protect\footnotemark[5]
  \and
  Omar Zoghlami
  \protect\footnotemark[1]
  \protect\footnotemark[4]
}

\date{}

\makeatletter
\def\@maketitle{
  \newpage
  {\centering
    \@title\par
    \vskip 1.5em
      {\large
        \lineskip .5em
        \begin{tabular}[t]{c}
          \@author
        \end{tabular}\par}
    \ifx\@date\@empty\else
      \vskip 1em
        {\large \@date}
    \fi
    \par}
  \vskip 1.5em
}

\let\template@maketitle\maketitle
\renewcommand{\maketitle}{
  \begingroup
  \renewcommand{\thefootnote}{\fnsymbol{footnote}}
  \long\def\@makefntext##1{
    \parindent 0pt
    \noindent\makebox[0pt][r]{\@makefnmark\,}##1
  }
  \template@maketitle
  \let\orig@makefnmark\@makefnmark   
  \insert\footins{
    \reset@font\footnotesize
    \interlinepenalty\interfootnotelinepenalty
    \splittopskip\footnotesep
    \splitmaxdepth\dp\strutbox
    \floatingpenalty\@MM
    \hsize\columnwidth
    \@parboxrestore
    \parindent 0pt
    \noindent Date: \today. \par
  }
  \let\@makefnmark\orig@makefnmark
  \footnotetext[1]{Faculty of Mathematics, University of Vienna, Oskar-Morgenstern-Platz 1, 1090 Vienna, Austria}
  \footnotetext[3]{Department of Mathematics, University of Trento, Via Sommarive 14, 38123 Povo (Trento), Italy}
  \def\@makefnmark{}                 
  \footnotetext[2]{\textit{E-mails}: 
    \orig@makefnmark\href{mailto:samuel.borza@univie.ac.at}{\nolinkurl{samuel.borza@univie.ac.at}};\,
    {\let\@makefnmark\orig@makefnmark \footnotemark[5]}
    \href{andrea.pinamonti@unitn.it}{\nolinkurl{andrea.pinamonti@unitn.it}};\,
    {\let\@makefnmark\orig@makefnmark \footnotemark[4]}
    \href{mailto:omar.zoghlami@univie.ac.at}{\nolinkurl{omar.zoghlami@univie.ac.at}}}
  \endgroup
}
\makeatother

\begin{document}

\maketitle							

\providecommand{\keywords}[1]
{
	\par\noindent\textbf{\textit{Keywords---}} #1\par
}

\providecommand{\msc}[1]
{
	\noindent\textbf{\textit{MSC (2020)---}} #1\par
}

\begin{abstract}
	We study constant horizontal mean curvature surfaces in the sub-Lorentzian
	Heisenberg group. We derive the first-variation formula for horizontal area
	under volume-preserving radial variations and show that smooth isoperimetric
	candidates have constant horizontal mean curvature away from the characteristic
	set. We then give a complete classification of smooth boost-symmetric constant mean
	curvature surfaces: their characteristic sets, causal behaviour, and
	ambient sub-Lorentzian isometry classes. From this classification, we single out a family of smooth, acausal, boost-symmetric surfaces with nonzero constant mean curvature. Written as a two-sheeted graph over the exterior of a future hyperbola, this family is a natural sub-Lorentzian analogue of the Pansu bubbles and leads us to conjecture that it gives the isoperimetric maximisers in the sub-Lorentzian Heisenberg group.
\end{abstract}

\keywords{Sub-Lorentzian Heisenberg group, constant mean curvature, isoperimetric problem}

\msc{53C50, 53C17, 53A10, 49Q05}

{\renewcommand{\contentsname}{\large Contents}
  \small
  \tableofcontents
}
\section{Introduction}

The sub-Lorentzian Heisenberg group is the Lorentzian counterpart to the sub-Riemannian Heisenberg group. Motion is constrained to the usual horizontal directions, but the horizontal distribution is equipped with a Lorentzian metric instead of a positive definite one. Horizontal vectors may be timelike, spacelike, or null, and admissible curves are horizontal curves with causal velocity.
Compared with sub-Riemannian geometry, sub-Lorentzian geometry is still less developed. It nevertheless fits naturally within the growing framework of synthetic Lorentzian geometry \cite{kunzingersaemann2018,cavallettimondino2024}. The causal and geodesic structures of the sub-Lorentzian Heisenberg group were studied in \cite{Grochowski2006,sachkovsachkova23}, optimal transport in \cite{borza2025}, and synthetic curvature-dimension conditions in \cite{hausdorffSLHeis}. In this work, we investigate surfaces of constant mean curvature in the sub-Lorentzian Heisenberg group.

Constant mean curvature equations are fundamental in geometry and analysis. They arise in isoperimetric problems, where a natural strategy is to search for optimizer candidates among constant mean curvature surfaces, since optimizers have this property. In the sub-Riemannian Heisenberg group, the longstanding Pansu conjecture predicts that isoperimetric regions are bounded by special rotationally symmetric surfaces of constant horizontal mean curvature \cite{PansuConjecture,Capogna2007}, while metric balls are known not to be isoperimetric \cite{BallNotIsop}. There has been substantial work on this conjecture. It has been resolved within the class of \(C^{1,1}\) sets \cite{RitoreRosales}, generalized to the sub-Finsler Heisenberg group in \cite{HeisIsopFinsler}, and studied in the axially symmetric \cite{MontiAxial} and convex \cite{MontiConvex} classes, to name only a few. To date, the only fully resolved sub-Riemannian isoperimetric problem is that of the \(\alpha\)-Grushin plane \cite{MontiGrushin}. Nearly all of these results use the constant mean curvature equation.

The variational and analytic background is closely related to the foundational work of Garofalo--Nhieu and Franchi--Serapioni--Serra Cassano on \(X\)-perimeter, isoperimetric inequalities, and minimal surfaces in Carnot--Carath\'eodory spaces \cite{MR1871966,GarofaloNhieu1996}, as well as to the sub-Riemannian surface calculus of \cite{DanielliGarofaloNhieu2007}. Constant mean curvature equations are also tied to Bernstein-type rigidity and minimal surface theory, both extensively studied in the Heisenberg group: stability results for minimal surfaces \cite{DanielliGarofaloNhieuPauls2009,DanielliGarofaloNhieuPauls2010}, \(p\)-mean curvature and the associated degenerate PDEs \cite{MR2165405}, and Codazzi-type equations and \(C^1\) regularity results \cite{MR2983199}, and Euler--Lagrange equations for minimal and isoperimetric surfaces in more general sub-Riemannian manifolds \cite{HladkyPauls2008}. It is worth mentioning several recent and relevant works on prescribed and constant mean curvature equations: the prescribed mean curvature equation for \(t\)-graphs in the sub-Finsler Heisenberg group \cite{GiovannardiPinamontiPozueloVerzellesi2024}, existence and uniqueness for \(t\)-graphs of prescribed mean curvature in Riemannian Heisenberg groups, with sub-Riemannian solutions obtained via an approximation argument \cite{PozueloVerzellesi2026}, regularity and Bernstein-type questions for the mean curvature equation in the sub-Finsler Heisenberg group \cite{GiovannardiRitore2021,GiovannardiRitore2024}, horizontally totally geodesic hypersurfaces in higher-dimensional Heisenberg groups \cite{PinamontiVerzellesi2025}, and curvature estimates for stable minimal hypersurfaces \cite{GiovannardiPinamontiVerzellesi2024}.

The present work is also a step toward a contribution to Lorentzian isoperimetric theory. Bahn and Ehrlich solved a first isoperimetric problem in Minkowski space \cite{BahnEhrlich}, while Cavalletti and Mondino extended L\'evy--Gromov's isoperimetric inequality to Lorentzian length spaces satisfying a timelike curvature-dimension condition \(\mathsf{TCD}(K,N)\) \cite{MondinoIsopLor}, see also \cite{LangePeteranderl2025} for recent quantitative versions of both results and \cite{Lambert2021} for an isoperimetric inequality in Robertson--Walker spacetimes. No isoperimetric result is known in sub-Lorentzian geometry.

Constant mean curvature hypersurfaces play an important role in general relativity. Spacelike Cauchy hypersurfaces provide initial data for Einstein's equations, but this initial data must satisfy the nonlinear Einstein constraint equations. The presence of a Cauchy hypersurface with constant mean curvature, or even better, a spacetime region foliated by constant mean curvature hypersurfaces, greatly simplifies the analysis. The question of whether CMC hypersurfaces exist under general physically motivated assumptions is one of the central problems in this area of research, see \cite{Marsden1980,Gerhardt1983,Bartnik1988,Ecker1991,Ling2024,Ling2025}.

With this background in mind, we investigate constant mean curvature hypersurfaces in the sub-Lorentzian Heisenberg group. We review the needed basics of sub-Lorentzian geometry in \cref{sec:sub_lorentzian_geometry}, and fully classify the sub-Lorentzian isometries of the Heisenberg group in \cref{thm:sub_lorentzian_isometries}. The isometry group is identified as \(\mathbb H \rtimes O(1,1)\) and splits into four components: time-preserving boosts, time-inverting boosts, time-preserving rotations, and time-inverting rotations. We then prove that surfaces area-stationary under volume-preserving radial variations have constant mean curvature, see \cite{RitoreRosales} for the parallel sub-Riemannian result. Our proof is intrinsically sub-Lorentzian and does not use Lorentzian approximants. We also show in \cref{thm:pseudo-sphere-not-CMC} that pseudo-spheres are not
solutions to Lorentzian isoperimetric problems in the sub-Lorentzian
Heisenberg group, in parallel with the classical result that
Carnot--Carath\'eodory balls are not isoperimetric in the sub-Riemannian
Heisenberg group \cite{BallNotIsop}.

In \cref{section:boost-symmetric-surfaces}, we study boost-symmetric hypersurfaces of constant mean curvature. This symmetry assumption is the Lorentzian analogue of rotational symmetry in the sub-Riemannian Heisenberg group, see \cite{RitoreRosales2006}. As shown in \cref{subsec:boost-symmetric-surfaces}, such a surface is always obtained by horizontally lifting a constant-curvature $k$ curve in the Minkowski plane and applying the boost action. These profile curves are straight lines for $k = 0$ and hyperbolas for $k \neq 0$. The remainder of \cref{section:boost-symmetric-surfaces} classifies these
surfaces up to ambient sub-Lorentzian isometry and studies their causal
properties. We summarise the results of that section here.

In the case \(k>0\), the profiles are determined by \(k>0\), where \(1/k\) is the radius of the hyperbola, by \(\varepsilon\in\{\pm1\}\), which selects the branch, and by \((c,d)\in\mathbb R^2\), which translates the hyperbola in the Minkowski plane. The additive constant in the lift is removed by a time-preserving boost, and the case \(\varepsilon=-1\) is carried to the case \(\varepsilon=1\) by
a time-inverting boost. Finally, after a Heisenberg dilation and a reparametrisation, one may set \(k=1\), reducing the corresponding surfaces to a family which we denote by \(S_{(c,d)}\). The classification of the normalised nonzero-curvature family \(S_{(c,d)}\)
is then organised according to the causal character of
\((c,d)\) in the Minkowski plane: the cases
\(c^2-d^2<0\), \(c^2-d^2>0\), and \(c^2-d^2=0\), which we call surfaces of
spacelike, timelike, and null parameters, respectively.

Consider two such surfaces \(S_{(c,d)}\) and \(S_{(c',d')}\), and write
\(\Delta=c^2-d^2\) and \(\Delta'=(c')^2-(d')^2\). In the nonzero null case,
we write \((c,d)=(c,\omega c)\) and
\((c',d')=(c',\omega'c')\), with \(\omega,\omega'\in\{-1,1\}\).
The isometry classification is summarised in \cref{tab:nonzero_cmc_isometry_types}: each entry gives the necessary and sufficient condition for an isometry of the
specified type to exist between \(S_{(c,d)}\) and \(S_{(c',d')}\).

\begin{table}[ht]
    \centering
    \begin{tabular}{|c|c|c|}
        \hline
        Parameter type
         &
        Time-preserving boost
         &
        Time-preserving rotation
        \\
        \hline
        \(\Delta<0\)
         &
        \(\Delta=\Delta'\), \(dd'>0\)
         &
        \(\Delta=\Delta'\), \(dd'<0\)
        \\
        \hline
        \(\Delta>0\)
         &
        \(\Delta=\Delta'\), \(cc'>0\)
         &
        \(\Delta=\Delta'\), \(cc'>0\)
        \\
        \hline
        \(\Delta=0\), \((c,d)\neq(0,0)\)
         &
        \(cc'>0\), \(\omega=\omega'\)
         &
        \(cc'>0\), \(\omega=-\omega'\)
        \\
        \hline
        \((c,d)=(0,0)\)
         &
        \((c',d')=(0,0)\)
         &
        \((c',d')=(0,0)\)
        \\
        \hline
    \end{tabular}
    \caption{Time-preserving isometry types for the normalised nonzero-curvature
        family \(S_{(c,d)}\). Time-inverting isometries do not occur in this family.}
    \label{tab:nonzero_cmc_isometry_types}
\end{table}
The hypersurfaces \(S_{(c,d)}\) are smooth precisely when
\(c^2-d^2\neq1\) or \(c\geq0\). In the case
\(c^2-d^2=1\) with \(c<0\), there is  a genuine singularity at the origin. Spacelike-parameter surfaces are
not achronal. Timelike-parameter surfaces are acausal for \(c>0\), not achronal for \(c<0\) in the regular case, and acausal in the singular case. Null-parameter
surfaces are acausal for \(c\geq0\) and not achronal for \(c<0\).

In the case \(k=0\), the profiles are determined by \(\varepsilon\in\{\pm1\}\), which selects the sign of the \(y\)-component of the direction of the line, by \((\alpha,\alpha')\in\mathbb R^2\), which translates the line in the Minkowski plane, and by \(\beta\in\mathbb R\), which determines its direction. The additive constant in the lift is removed by a time-preserving boost, and the parameter \(\alpha'\) is removed by reparametrising the line, reducing the corresponding surfaces to a family which we denote by \(S_{(\varepsilon,\alpha,\beta)}\).

For the maximal family, the sign \(\varepsilon\) does not affect the isometry
class. For two maximal surfaces \(S_{(\varepsilon,\alpha,\beta)}\) and
\(S_{(\varepsilon',\alpha',\beta')}\), the isometry classification is summarised in
\cref{tab:maximal_isometry_types}.

\begin{table}[ht]
    \centering
    \begin{tabular}{|c|c|}
        \hline
        Time-preserving boost
         &
        \(\abs{\alpha}\cosh\beta=\abs{\alpha'}\cosh\beta'\), \(\alpha\alpha'>0\)
        \\
        \hline
        Time-preserving rotation
         &
        \(\abs{\alpha}\cosh\beta=\abs{\alpha'}\cosh\beta'\), \(\alpha\alpha'>0\)
        \\
        \hline
        Time-inverting boost
         &
        \(\abs{\alpha}\cosh\beta=\abs{\alpha'}\cosh\beta'\), \(\alpha\alpha'<0\)
        \\
        \hline
        Time-inverting rotation
         &
        \(\abs{\alpha}\cosh\beta=\abs{\alpha'}\cosh\beta'\), \(\alpha\alpha'<0\)
        \\
        \hline
    \end{tabular}
    \caption{Isometry types for regular maximal surfaces
        \(S_{(\varepsilon,\alpha,\beta)}\) and
        \(S_{(\varepsilon',\alpha',\beta')}\). The case \(\alpha=0\) is singular and gives a single image, independent of
        \(\varepsilon\) and \(\beta\). This singular maximal surface is not isometric to
        any regular maximal surface.
    }
    \label{tab:maximal_isometry_types}
\end{table}
The maximal hypersurfaces \(S_{(\varepsilon,\alpha,\beta)}\) are smooth
precisely when \(\alpha\neq0\). When \(\alpha=0\), the surface is independent of
\(\varepsilon\) and \(\beta\) and has a genuine singularity at the origin. Every regular maximal surface \(S_{(\varepsilon,\alpha,\beta)}\) with \(\alpha\neq0\) is not achronal, whereas the singular maximal surface corresponding to \(\alpha=0\) is acausal. Interestingly, all the maximal surfaces are ruled by straight horizontal lines, which is an important property to study the Bernstein problem in the Heisenberg group \cite{Young2022}.

The classifications in
\cref{tab:nonzero_cmc_isometry_types,tab:maximal_isometry_types} are exhaustive:
no further ambient sub-Lorentzian isometries occur among the surfaces
\(S_{(c,d)}\) and \(S_{(\varepsilon,\alpha,\beta)}\). In particular, surfaces
lying in different rows of \cref{tab:nonzero_cmc_isometry_types} are not
isometric, and surfaces with different curvature parameters \(k\) are not
isometric. The proof combines the explicit form of the ambient isometry group, geometric invariants, and the complete description of the characteristic set of each surface.

We end this introduction by conjecturing that sub-Lorentzian isoperimetric
maximisers should belong to the timelike-parameter family \(S_{(c,d)}\) with
\(c^2-d^2>0\) and \(c>0\). This is a natural candidate family singled out by
the boost-symmetric classification: the other nonzero-CMC families are ruled
out by the causal constraints, since spacelike-parameter surfaces are not
achronal, regular timelike-parameter surfaces with \(c<0\) are not achronal,
null-parameter surfaces with \(c<0\) are not achronal and those with \(c>0\) can be seen as limits of timelike ones. Up to a
time-preserving boost, the conjectured isoperimetric maximiser is \(S_{(C,0)}\)
with \(C>0\), and it can be written as the set of points \((x,y,z)\) such that
\(x>0\), \(x^2-y^2\geq(C+1)^2\), and
\[
    \abs{z}
    =
    \frac12\left(
    \arccosh\left(\frac{x^2-y^2-C^2-1}{2C}\right)
    +
    \sqrt{\left(\frac{x^2-y^2-C^2-1}{2}\right)^2-C^2}
    \right).
\]
This expression is strongly reminiscent of the Pansu bubble in the
sub-Riemannian Heisenberg isoperimetric problem, see \cite{PansuConjecture} and \cite[Section~8.1, p.~152]{Capogna2007}. This candidate surface is illustrated in \cref{fig:timelikesurface}.

\begin{figure}[h]
    \centering
    \includegraphics[scale = 0.4]{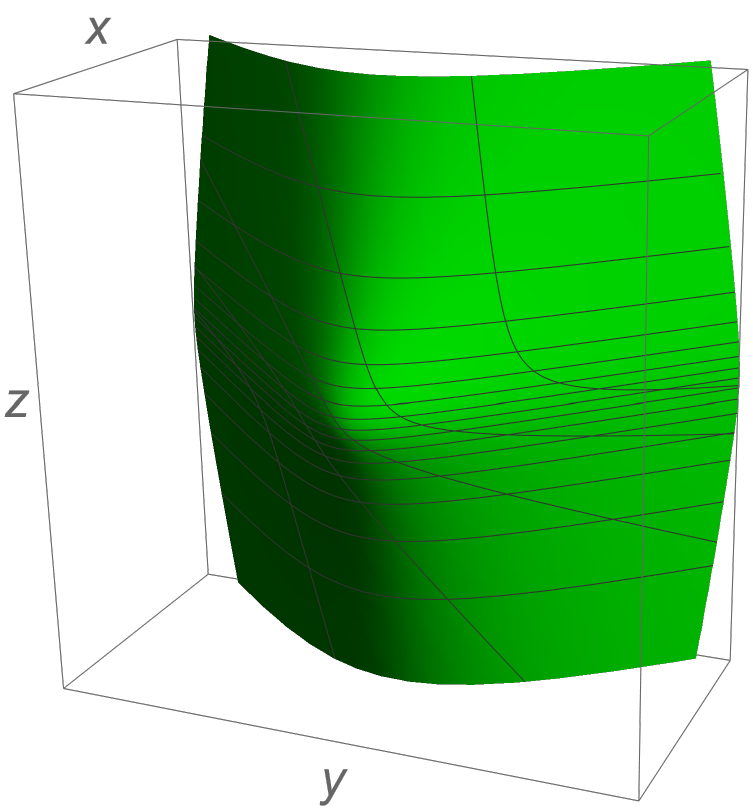}
    \caption{A timelike-parameter surface \(S_{(C,0)}\) with \(C>0\).}
    \label{fig:timelikesurface}
\end{figure}

\section*{Acknowledgements}

This research was funded by the Austrian Science Fund (FWF) [Grant DOI 10.55776/EFP6]. For open access
purposes, the authors have applied a CC BY public copyright license to any author
accepted manuscript version arising from this submission. S.B. is supported by SUBLOR, a research project funded by the European Union under the Horizon Europe programme's Marie Skłodowska-Curie Actions Postdoctoral Fellowships, grant agreement No. 101282277. A.P. is supported by the University of Trento. A. P. received funding through INdAM-GNAMPA 2026 Project \emph{Variational, Geometric, and Analytic Perspectives on Regularity}, CUP E53C25002010001.
The authors are thankful to Chiara Rigoni for early conversations on this work
and to Davide Carazzato for helpful discussions related to
\Cref{sec:volume_preserving_variations}.

\section{Sub-Lorentzian and horizontal geometry in the Heisenberg group}

\subsection{The sub-Lorentzian Heisenberg group}

\label{sec:sub_lorentzian_geometry}

The Heisenberg group $\heis$ has been studied extensively when equipped with its natural sub-Riemannian structure; see, for instance, \cite{Capogna2007,Agrachev2020}. For the purposes of this work, we review the sub-Lorentzian structure of the Heisenberg group, in line with works such as \cite{sachkovsachkova23, Huang2012, Grochowski2006, Grochowski2004, borza2025, hausdorffSLHeis}.

The Heisenberg group $\heis$ is the Lie group that can be identified with the Euclidean space $\R^3$ equipped with the non-Abelian group law
\[
    (x,y,z) \ast (x',y',z') \coloneqq \left( x+x',y + y', z + z' +\frac{1}{2}(x y'- x' y) \right),
\]
for all $(x, y, z), (x', y', z') \in \R^3$. The identity element is the origin $e = (0, 0, 0)$ of $\R^3$, while $(x, y, z)^{-1} = (-x, -y, -z)$. The left translation by $p \in \heis$ is the automorphism $L_p : \heis \to \heis$ given by $L_p(q) = p \ast q$, and a vector field $V$ on $\heis$ is said to be \emph{left-invariant} if $\diff_q L_p [V(q)] = V(p \ast q)$ for all $p, q \in \heis$. A left-invariant vector field is uniquely determined by its value at the group identity element, and those that coincide with $\partial_x, \partial_y, \partial_z$ at $e$, respectively, are
\[
    X = \partial_x - \frac{y}{2} \partial_z,\quad Y = \partial_y + \frac{x}{2} \partial_z, \quad \text{and} \quad Z = \partial_z.
\]

The sub-bundle $\Delta$ of the tangent bundle $\T(\heis)$ of $\heis$, defined by
\[
    \Delta \coloneqq \operatorname{span} \left\{ X,Y \right\} \subseteq \T(\heis),
\]
is referred to as the {\em horizontal distribution} of $\heis$ and, at any point $p \in \heis$, vectors $v \in \Delta_p$ are called {\em horizontal vectors}.
Importantly, the distribution $\Delta$ is {\em not} involutive; indeed, we have $[X,Y] = \partial_z = Z \notin \Delta$. A vector field $V$ on $\heis$ is called \emph{horizontal} if it is a smooth section of $\Delta$. A curve $\gamma : I \to \heis$ is said to be \emph{horizontal} if $\gamma$ is absolutely continuous and there exist $u, v \in L^{\infty}(I, \R)$ such that
\begin{equation}
    \label{eq:horizontalcurveinH}
    \dot \gamma (t) = u(t) X(\gamma(t)) + v(t) Y(\gamma(t)), \qquad \text{for almost every } t \in I.
\end{equation}
For any $p \in \heis$ and any horizontal curve $\gamma$, the translated curve $p \ast \gamma \coloneqq L_p \circ \gamma$ is also horizontal. Note that a curve \(\gamma(t)=(x(t),y(t),z(t))\) is horizontal if and only if
\begin{equation}
    \label{eq:horizontallift}
    \dot z(t)=\frac12\bigl(x(t)\dot y(t)-y(t)\dot x(t)\bigr)
    , \qquad \text{for almost every } t \in I.
\end{equation}

The sub-Lorentzian structure on $\heis$ is induced by considering the Lorentzian metric $g$ on $\Delta$ uniquely determined by the conditions
\[
    g(X,X)=-1, \quad g(X,Y) = 0, \quad g(Y,Y) = 1.
\]
The structure $(\heis, \Delta, g)$, which is preserved under left translations, is what we call the {\em sub-Lorentzian Heisenberg group}. For any $p \in \heis$, we say that a horizontal vector $v \in \Delta_p$ is
\[
    \begin{cases}
        \text{causal}   \\
        \text{timelike} \\
        \text{null}     \\
        \text{spacelike}
    \end{cases}
    \quad \text{if} \quad g_p(v, v) \;
    \begin{cases}
        \leq 0 \text{ and } v \neq 0 \\
        < 0                          \\
        = 0 \text{ and } v \neq 0    \\
        > 0 \text{ or } v = 0
    \end{cases}.
\]
A causal vector is said to be \emph{future-directed} if $g_p(v, X_p) < 0$. Accordingly, we say that a horizontal curve $\gamma$ is {\em causal} (resp. {\em timelike}, {\em null}, {\em future-directed}) if the horizontal vector $\dot{\gamma}(t)$ is {\em causal} (resp. {\em timelike}, {\em null}, {\em future-directed}) for almost every $t$. The \emph{horizontal gradient} of $f$ is defined by the condition
\[
    g\bigl(\operatorname{grad}_{\heis} f,V\bigr)=V(f)
    \qquad
    \text{for every } V\in\Gamma(\Delta).
\]

For any two points $p,q \in \heis$, we say that $p$ {\em causally precedes} $q$, denoted by $p \leq q$, if $p=q$ or there exists a future-directed causal curve $\gamma$ joining $p$ to $q$. Analogously, we say that $p$ {\em chronologically precedes} $q$, denoted by $p \ll q$, if there exists a future-directed timelike curve joining $p$ to $q$. The {\em causal} and {\em chronological futures} of $A \subseteq \heis$ are defined, respectively, by
\begin{equation}
    \label{eq:causalchronologicalfuturedefinition}
    J^+(A) \coloneqq \{ y \in \heis \mid \exists x \in A,\ x \leq y \}, \quad \text{and} \quad I^+(A) \coloneqq \{ y \in \heis \mid \exists x \in A,\ x \ll y \}.
\end{equation}
Similarly, one can define the causal and chronological pasts $J^-(A)$ and $I^-(A)$. The {\em causal} and {\em chronological diamonds} of two sets $A, B \subseteq \heis$ are given by
\begin{equation}
    \label{eq:causalchronologicaldiamonddefinition}
    J(A, B) \coloneqq J^+(A) \cap J^-(B), \quad \text{and} \quad I(A, B) \coloneqq I^+(A) \cap I^-(B).
\end{equation}
By left-invariance, we have $J^+(p) = p \ast J^+(e)$ and $I^+(p) = p \ast I^+(e)$, and similarly for the pasts. Moreover, it is known that
\begin{equation}\label{eq:future_origin_heisenberg}
    J^+(e) = \{x\geq 0\} \cap \{-x^2+y^2+4\abs{z} \leq 0\}, \quad
    I^+(e) = \{x>0\} \cap \{-x^2+y^2+4\abs{z} < 0\}.
\end{equation}

The {\em sub-Lorentzian length} of a future-directed causal curve $\gamma : I \to \heis$ is given by
\[
    L(\gamma) \coloneqq \int_I \sqrt{-g(\dot{\gamma}(t),\dot{\gamma}(t))} \, \diff t.
\]
The {\em time-separation function} $\uptau$ between two points $p,q \in \heis$ is then defined by
\[
    \uptau(p,q) \coloneqq \sup \{ L(\gamma) \mid \gamma \ \text{is future-directed causal and joins $p$ to $q$}\} \quad \text{if } p \leq q,
\]
and $\uptau(p,q) \coloneqq 0$ if $p$ does not causally precede $q$. It is easy to show that the time-separation function $\uptau$ satisfies the {\em reverse triangle inequality}, that is,
\[
    \uptau(q_1, q_3) \geq \uptau(q_1, q_2) + \uptau(q_2, q_3) \qquad \text{if } q_1 \leq q_2 \leq q_3 \text{ in } \heis.
\]

A future-directed causal curve $\gamma$ from $p$ to $q$ is a {\em maximizing geodesic} if $L(\gamma) = \uptau(p,q)$. The \emph{exponential map} from $e$ of the sub-Lorentzian Heisenberg group is given by
\begin{equation}
    \label{eq:exponentialmap}
    \exp_{e}(u, v, w) \coloneqq
    \begin{pmatrix}
        \dfrac{v(\cosh(w)-1) +u\sinh(w)}{w} \\[1em]
        \dfrac{v\sinh(w) +u(\cosh(w)-1)}{w} \\[1em]
        \dfrac{u^2-v^2}{2}\cdot\dfrac{\sinh(w)-w}{w^2}
    \end{pmatrix}
\end{equation}
for $u > \abs{v}$ and $w \neq 0$. It is extended smoothly to $\exp_e(u, v, w) = (u, v, 0)$ when $w = 0$. The exponential map describes all constant speed timelike geodesics $\gamma(t)$ from the origin, i.e. $\gamma(t) =  \exp_e(tu, tv, tw)$ for some $u, v, w$, see \cite{sachkovsachkova23,borza2025,hausdorffSLHeis}. By left translation, all other timelike geodesics can be recovered and the exponential map $\exp_p$ from any $p$ generating timelike geodesics starting from $p$ can a be considered.
The map $\tau(p, \cdot) \coloneq q \mapsto \tau(p, q)$ is analytic on $\{p \ll q\}$, see \cite{sachkovsachkova23}, and satisfies the eikonal equation
\begin{equation}
    \label{eq:eikonaleq1}
    g\bigl(\operatorname{grad}_{\heis}\tau(p, \cdot),
    \operatorname{grad}_{\heis}\tau(p, \cdot)\bigr)=-1,
\end{equation}
and we also have
\begin{equation}
    \label{eq:eikonaleq2}
    \frac{d}{dt}\exp_p(t\lambda)
    =
    -\tau\bigl(p, \exp_p(t \lambda)\bigr)\,
    \operatorname{grad}_{\heis}\tau\bigl(p, \exp_p(t\lambda)\bigr).
\end{equation}

The following theorem describes the structure of smooth isometries of $\heis$, that is, diffeomorphisms $F$ for which, for every $p \in \heis$, it holds $d_pF(\Delta_p) = \Delta_{F(p)}$ and $F^*g = g$.
\begin{theorem}\label{thm:sub_lorentzian_isometries}
    A map $F \colon \heis \to \heis$ is a smooth isometry if and only if
    \[
        F(x,y,z) \coloneqq
        q \ast
        \begin{pNiceArray}{w{c}{0.5cm}c|c}[margin]
            \Block{2-2}<\Large>{\Lambda} & & 0\\
            & & 0 \\
            \hline
            0 & 0 & \det(\Lambda)
        \end{pNiceArray}
        \begin{pmatrix}
            x \\
            y \\
            z
        \end{pmatrix}
    \]
    for some $q \in \heis$ and some $\Lambda \in O(1,1)$.
\end{theorem}
\begin{proof}
    The \say{if} part of the statement is a straightforward computation. To prove the converse implication, by replacing $F$ with $F(0)^{-1} \ast F$ it is enough to assume that $F(0)=0$ and prove that, in this case, $F$ is of the above form with $q=0$.
    We consider the $1$-forms $X^\flat, Y^\flat$ on $\Delta$, dual to $X$ and $Y$ via the metric $g$, that is, for any $V \in \Delta_p$,
    \[
        X^\flat(V) \coloneqq g(X,V), \quad Y^\flat(V) \coloneqq g(Y,V).
    \]
    One immediately sees that $X^\flat = -dx_{\vert \Delta}$ and $Y^\flat = dy_{\vert \Delta}$. In particular, it holds that
    \[
        X^\flat \wedge Y^\flat = (-dx \wedge dy)_{\vert \Delta}.
    \]

    Since $F$ is an isometry, the map $d_pF : \Delta_p \to \Delta_{F(p)}$ is a linear isometry between two-dimensional Lorentzian vector spaces. With respect to the bases $\{X_p,Y_p\}$ and $\{X_{F(p)},Y_{F(p)}\}$, we denote its matrix by $\Lambda_p$ so that
    \begin{equation}\label{eq:pullback_preserves_volume}
        F^*(X_{F(p)}^\flat \wedge Y^\flat_{F(p)}) = \det(\Lambda_p) X^\flat_p \wedge Y^\flat_p = \det(\Lambda_p)(-dx \wedge dy)_{\vert \Delta}.
    \end{equation}
    Since $X,Y$ are pseudo-orthonormal and $F$ is an isometry, we have $\Lambda_p \in O(1,1)$. In particular, $\epsilon := \det(\Lambda_p)\in\{-1,1\}$, and, by the continuity of $p \mapsto \det(\Lambda_p)$ and the connectedness of $\heis$, this sign is constant.

    We now consider the contact form $\alpha = dz -\frac{1}{2}(xdy-ydx)$ on $\heis$. Since $\Delta = \ker(\alpha)$ and $F$ preserves the horizontal distribution, we must have
    \[
        F^*\alpha = \lambda \alpha
    \]
    for some $\lambda \in C^{\infty}(\heis)$. Taking the exterior derivative of both sides, we obtain
    \[
        F^*d\alpha = d\lambda \wedge \alpha + \lambda d\alpha, \quad \text{ and thus } \quad F^*(d\alpha_{\vert \Delta}) = \lambda d\alpha_{\vert \Delta}.
    \]
    As $d\alpha = -dx \wedge dy$, from \eqref{eq:pullback_preserves_volume} we obtain
    \[
        \varepsilon (X^\flat \wedge Y^\flat)
        = \lambda d\alpha_{\vert \Delta}
        = \lambda(-dx \wedge dy)_{\vert \Delta}
        = \lambda(X^\flat \wedge Y^\flat),
    \]
    which yields $\lambda \equiv \varepsilon$. In particular, $d\lambda = 0$, and therefore the identity $F^*d\alpha = \varepsilon d\alpha$ holds everywhere. If we decompose $dF[Z] = aX+bY+cZ$ for some functions $a,b,c \in C^\infty(\heis)$, then, for any horizontal vector field $U$, we have
    \[
        \begin{split}
            0
             & = -\varepsilon (dx \wedge dy)(Z,U)
            = \varepsilon d\alpha(Z,U)
            = F^*(d\alpha)(Z,U)  = d\alpha(aX+bY+cZ, dF[U]) \\
             &
            = d\alpha(aX+bY, dF[U]) = (X^{\flat}\wedge Y^\flat)(aX+bY, dF[U]).
        \end{split}
    \]
    Since $dF[U]$ ranges over the horizontal distribution and $X^\flat \wedge Y^\flat$ is non-degenerate on $\Delta$, this forces $aX+bY=0$. The identites $\alpha(Z) = 1$ and $F^*\alpha = \varepsilon \alpha$ imply
    \[
        c = \alpha(cZ) = (F^*\alpha)(Z) = \varepsilon \alpha(Z) = \varepsilon,
    \]
    and therefore $dF[Z] = \varepsilon Z$. Writing $F=(F_1,F_2,F_3)$, this means  that $\partial_z F_1 = \partial_z F_2 = 0$.

    Define a map on the $(x,y)$-plane by $f(x,y) \coloneqq \pi(F(x,y,0))$, where $\pi(x,y,z)=(x,y)$. Thus, $f$ is the map given by the first two components of $F$, restricted to the plane $z=0$. By what we have just proved, $\pi(F(x,y,z))=\pi(F(x,y,0))$ for all $(x,y,z) \in \heis$. For any $a,b \in \R$, setting $c \coloneqq \frac{ay-bx}{2}$ we have, at the point $(x,y,0)$,
    \begin{align*}
        d\pi \circ dF(a\partial_x+b\partial_y)
         & =
        d\pi \circ dF(aX+bY+cZ) =
        d\pi\bigl(\Lambda_{(x,y,0)}(aX+bY)+\varepsilon cZ\bigr) \\
         & =
        d\pi\bigl(\Lambda_{(x,y,0)}(aX+bY)\bigr) =
        \Lambda_{(x,y,0)}(a\partial_x+b\partial_y),
    \end{align*}
    where we used $dF[Z] = \varepsilon Z$ and the fact that $dF$ acts as an isometry on the horizontal distribution. Hence $df_{(x,y)}=\Lambda_{(x,y,0)} \in O(1, 1)$. At this stage, $\Lambda_{(x,y,0)}$ could, in principle, depend on $(x,y)$. However, this cannot actually happen, because a smooth map of the Minkowski plane whose differential lies everywhere in $O(1,1)$ must be affine. Therefore $\Lambda \coloneqq \Lambda_{(x,y,0)}$ is constant. Since $F(0)=0$, we have $f(0,0)=\pi(F(0,0,0))=(0,0)$, and $f(x,y) = \pi(F(x, y, z))=\Lambda(x, y)$.

    From the previous steps, the first two components of $F$ are given by $F_1=\Lambda_{11}x+\Lambda_{12}y$ and $F_2=\Lambda_{21}x+\Lambda_{22}y$, while $\partial_z F_3=\varepsilon$, where $\varepsilon=\det(\Lambda)$. Writing explicitly $F^*\alpha=\varepsilon\alpha$ and rearranging we get
    \begin{align*}
        dF_3
         & =
        \frac{1}{2}(F_1\,dF_2-F_2\,dF_1)
        + \varepsilon\left(dz-\frac{1}{2}(x\,dy-y\,dx)\right) \\
         & =
        \frac{\det(\Lambda)}{2}(x\,dy-y\,dx)
        + \varepsilon\left(dz-\frac{1}{2}(x\,dy-y\,dx)\right) =
        \varepsilon dz.
    \end{align*}
    Hence $\partial_xF_3=\partial_yF_3=0$ and $\partial_zF_3=\varepsilon$. Since $F(0)=0$, we conclude that $F_3(x,y,z)=\varepsilon z=\det(\Lambda)z$, which completes the proof.
\end{proof}

In view of the previous proposition, an isometry is completely characterised by a point $p \in \heis$, which we call the {\em translation part} of the isometry, and an element $\Lambda \in O(1,1)$, which we call the {\em Lorentz part} of the isometry.

\begin{definition}
    \label{def:HeisenbergIsometries}
    An isometry of the sub-Lorentzian Heisenberg group is called a
    \begin{enumerate}[label=\normalfont(\roman*)]
        \item {\em time-preserving boost} if $\Lambda \in O^+(1,1)$, that is, $\det(\Lambda) =1$ and $\Lambda_{11} > 0$;
        \item {\em time-inverting boost} if $\det(\Lambda) = 1$ and $\Lambda_{11}<0$;
        \item {\em time-preserving rotation} if $\det(\Lambda) = -1$ and $\Lambda_{11} > 0$;
        \item {\em time-inverting rotation} if $\det(\Lambda) = -1$ and $\Lambda_{11}<0$.
    \end{enumerate}
    In the case of a time-preserving boost with $p=0$, we simply call it a {\em boost}. In the case of a time-preserving rotation with $p=0$, we call it a {\em timed rotation}.
\end{definition}

We use the following general component decomposition of \(O(1,1)\) repeatedly. Let $A(x,y,z)=(-x,y,-z)$ and $T(x,y,z)=(x,-y,-z)$, and note that their Lorentz parts on the \((x,y)\)-plane are
\[
    A_0=
    \begin{pmatrix}
        -1 & 0 \\
        0  & 1
    \end{pmatrix},
    \qquad
    T_0=
    \begin{pmatrix}
        1 & 0  \\
        0 & -1
    \end{pmatrix}.
\]
In general, a time-preserving boost has Lorentz part
\[
    B_u=
    \begin{pmatrix}
        \cosh u & \sinh u \\
        \sinh u & \cosh u
    \end{pmatrix},
\]
for some $u \in \mathbb R$. Then, the four components of \(O(1,1)\) are parametrized by
\[
    B_u,\qquad
    A_0B_u,\qquad
    T_0B_u,\qquad
    A_0T_0B_u,
\]
as $u \in \mathbb{R}$. In particular, \(B_u\) is a time-preserving boost, \(T_0B_u\) a time-preserving rotation, \(A_0B_u\) a time-inverting rotation, and \(A_0T_0B_u\) a time-inverting boost.
\subsection{Horizontal geometry of hypersurfaces}

We review the geometry of hypersurfaces in the sub-Lorentzian Heisenberg group, including the notion of sub-Lorentzian mean curvature.

\begin{definition}[Horizontal tangent bundle]
    Let $\Sigma \subset \heis$ be an embedded $C^{1}$-surface. At each point $p \in \Sigma$, we define $\mathrm{H}_p(\Sigma) \coloneqq \Delta_p \cap \T_p(\Sigma)$. The {\em characteristic set} of $\Sigma$ is defined as
    \[
        \mathrm S(\Sigma) \coloneqq \{p \in \Sigma \ \colon \ \mathrm{H}_p(\Sigma) = \T_p(\Sigma)\}.
    \]
    This set is closed in $\Sigma$ and, if $\Sigma$ is $C^{1, 1}$, is negligible by \cite[Proposition 13]{Prandi2019}. On the manifold $\Sigma \setminus \mathrm S(\Sigma)$, we define the vector bundle
    \[
        \mathrm H(\Sigma) \coloneqq \bigsqcup_{p \in \Sigma \setminus \mathrm S(\Sigma)} \mathrm H_p (\Sigma),
    \]
    called the {\em horizontal tangent bundle} of $\Sigma$. It naturally carries the structure of a rank-one vector bundle over $\Sigma \setminus \mathrm S(\Sigma)$, since, for every $p \in \Sigma \setminus \mathrm S(\Sigma)$, we have
    \[
        \dim(\mathrm H_p (\Sigma)) = \dim(\Delta_p) + \dim(\T_p(\Sigma)) - \dim(\Delta_p \oplus \T_p(\Sigma)) = 2 +2 -3=1.
    \]
\end{definition}
The sub-Lorentzian metric $g$ restricts naturally to a quadratic form on the horizontal tangent bundle $\mathrm H(\Sigma)$ over $\Sigma \setminus \mathrm S(\Sigma)$.
\begin{definition}
    We say that an embedded $C^{1}$-surface $\Sigma$ is
    \begin{enumerate}[label=\normalfont(\roman*)]
        \item {\em spacelike} if the restriction of $g$ to $\mathrm H(\Sigma)$ is positive definite;
        \item {\em timelike} if the restriction of $g$ to $\mathrm H(\Sigma)$ is non-degenerate with a negative eigenvalue;
        \item {\em null} or {\em lightlike} if the restriction of $g$ to $\mathrm H(\Sigma)$ is degenerate.
    \end{enumerate}
\end{definition}

Notice that $\mathrm H_p(\Sigma)$ is a hyperplane in $\Delta_p$ for every
$p \in \Sigma \setminus \mathrm S(\Sigma)$. In the first two cases, taking its $g$-orthogonal complement
inside $\Delta_p$ gives a non-degenerate line. If $\Sigma$ is spacelike (resp. timelike), this orthogonal line is timelike (resp. spacelike), and the fixed time orientation (resp. a choice of orientation) selects its future-directed unit vector (resp. its unit vector).

\begin{definition}[Horizontal unit normal]
    The {\em horizontal unit normal} of $\Sigma$ is the vector field $\nu_\heis \in C^0\bigl(\Sigma \setminus \mathrm S(\Sigma), \Delta|_{\Sigma \setminus \mathrm S(\Sigma)}\bigr)$ uniquely determined by
    \[
        g_p(\nu_\heis(p),v)=0 \quad \text{for every } v \in \mathrm H_p(\Sigma),
        \qquad
        g_p(\nu_\heis(p),\nu_\heis(p))
        =
        \begin{cases}
            -1, & \text{if } \Sigma \text{ is spacelike}, \\
            1,  & \text{if } \Sigma \text{ is timelike},
        \end{cases}
    \]
    together with the sign convention that $\nu_\heis$ is future-directed when $\Sigma$ is spacelike and agrees with a chosen orientation of $\Sigma$ when it is timelike. If one wants to regard $\nu_\heis$ as defined on all of $\Sigma$, we set $\nu_\heis(p)=0$ for $p \in \mathrm S(\Sigma)$.
\end{definition}

A canonical horizontal connection determined by the left-invariant horizontal frame $X,Y$ can then be introduced.

\begin{definition}[Horizontal connection]
    Let $\Delta=\operatorname{span}\{X,Y\}$. For a horizontal vector field
    $V=aX+bY$ and any smooth vector field $U$, define
    \[
        \nabla_UV:=U(a)X+U(b)Y.
    \]
    In other words, $\nabla$ is the unique connection on $\Delta$ for which the left-invariant horizontal frame $X,Y$ is parallel. In particular, $\nabla$ is metric compatible:
    \[
        U(g(V,W))=g(\nabla_UV,W)+g(V,\nabla_UW)
    \]
    for all horizontal vector fields $V,W$.
\end{definition}

A horizontal vector field along $\Sigma$ is a smooth section of $\Delta|_\Sigma$. On $\Sigma \setminus \mathrm S(\Sigma)$, since $\Delta|_{\Sigma \setminus \mathrm S(\Sigma)} = \mathrm H(\Sigma) \oplus \operatorname{span}\{\nu_\heis\}$, every horizontal vector field $V$ along $\Sigma$ decomposes uniquely into its horizontal tangential and horizontal normal components, which we denote by $V^\parallel$ and $V^\perp$, respectively.

\begin{definition}[Horizontal tangent connection and second fundamental form]
    Let $\Sigma \subset \heis$ be an embedded $C^1$ spacelike or timelike surface, and $\nu_\heis$ denote its horizontal unit normal. We define the {\em horizontal tangent connection} by
    \begin{equation}
        \label{eq:HorTConnection}
        \nabla^\Sigma_U V
        \coloneqq
        \bigl(\nabla_U V\bigr)^\parallel
    \end{equation}
    for all horizontal vector field $U$ and $V$ along $\Sigma \setminus \mathrm S(\Sigma)$. The {\em scalar horizontal second fundamental form} is the bilinear form
    \[
        h_\Sigma(U,V)
        \coloneqq
        g(\nabla_U V,\nu_\heis).
    \]
\end{definition}

\begin{remark}
    If $V\in\Gamma(\mathrm H(\Sigma))$ is regarded as a vector field only along $\Sigma \setminus \mathrm S(\Sigma)$, then $\nabla_UV$ in \cref{eq:HorTConnection} is computed by choosing any local horizontal extension of $V$. The value along $\Sigma \setminus \mathrm S(\Sigma)$ is independent of the chosen extension, because $U\in\Gamma(\mathrm H(\Sigma))$ is tangent to $\Sigma$.
\end{remark}

Since $\mathrm H(\Sigma)$ has rank one, we may locally choose a generator $T\in\Gamma_{\mathrm{loc}}(\mathrm H(\Sigma))$ normalized with respect to $g$, so that $g(T,T)=1$ in the spacelike case and $g(T,T)=-1$ in the timelike case. If $T$ be a local non-vanishing generator of $\mathrm H(\Sigma)$, then any $U,V\in\Gamma(\mathrm H(\Sigma))$ can be written as $U=uT$ and $V=vT$. First notice that $h_\Sigma$ is $C^\infty(\Sigma)$-bilinear. For instance,
\[
    h_\Sigma(U,fV)
    =
    g(\nabla_U(fV),\nu_\heis)
    =
    U(f)g(V,\nu_\heis)+f h_\Sigma(U,V)
    =
    f h_\Sigma(U,V),
\]
since $V\in \Gamma(\mathrm H(\Sigma))$ and $g(V,\nu_\heis)=0$. Linearity in the first slot follows from $\nabla_{fU}V=f\nabla_UV$. Hence
\[
    h_\Sigma(U,V)
    =
    h_\Sigma(uT,vT)
    =
    uv\,h_\Sigma(T,T)
    =
    h_\Sigma(vT,uT)
    =
    h_\Sigma(V,U).
\]
Thus the horizontal second fundamental form $h_\Sigma$ is symmetric.

\begin{definition}
    Let $\Sigma \subset \heis$ be an embedded $C^1$ spacelike or timelike surface, and $\nu_\heis$ denote its horizontal unit normal. The trace of scalar horizontal second fundamental form with respect to the restriction of $g$ to $\mathrm H(\Sigma)$ is the {\em horizontal mean curvature}
    \[
        H
        \coloneqq
        \operatorname{tr}_{g|_{\mathrm H(\Sigma)}} h_\Sigma.
    \]
    Equivalently, if $\Sigma$ is $C^2$ and $T$ is a local $g$-unit generator of $\mathrm H(\Sigma)$, then
    \[
        H=g(T,T)\,g(\nabla_TT,\nu_\heis).
    \]
\end{definition}

The following proposition computes $H$ in terms of the integral curves of $T$.

\begin{proposition}\label{prop:mean_curvature_via_minkowski}
    Let $\Sigma \subset \heis$ be an embedded $C^2$ spacelike or timelike surface, and work on $\Sigma \setminus \mathrm S(\Sigma)$. Choose a local generator $T$ of $\mathrm H(\Sigma)$ normalized with respect to $g$. For $p\in\Sigma \setminus \mathrm S(\Sigma)$, denote by $\gamma$ the integral curve of $T$ with $\gamma(0)=p$, and set $\sigma\coloneqq\pi\circ\gamma$, where $\pi(x,y,z)=(x,y)$. Then $H(p)$ is the signed curvature of $\sigma$ in the Minkowski plane $(\R^2,-dx^2+dy^2)$, namely
    \[
        H(p)=g(T,T)
        \left\langle \ddot\sigma(0),n(0)\right\rangle_{\mathbb M},
    \]
    where $\nu_\heis=n_1X+n_2Y$ and $n=(n_1,n_2)$ is its projection to the Minkowski plane.
\end{proposition}

\begin{proof}
    Write $T=fX+hY$. Since $X$ and $Y$ project respectively to $\partial_x$ and $\partial_y$, and since $\dot\gamma=T_{\gamma}$, we have
    \[
        \dot\sigma=(f\circ\gamma,h\circ\gamma),
        \qquad
        \ddot\sigma(0)=(T(f)_p,T(h)_p).
    \]
    The definition of the horizontal connection gives
    \[
        \nabla_TT=T(f)X+T(h)Y.
    \]
    For some $\eta\in\{-1,1\}$ fixed by the chosen convention, the horizontal unit normal is
    \[
        \nu_\heis=\eta(hX+fY),
    \]
    and its projection to the Minkowski plane is $n=\eta(h,f)$, which coincides with the normal to \(\sigma\) in Minkowski plane. Therefore
    \[
        \left\langle \ddot\sigma(0),n(0)\right\rangle_{\mathbb M}
        =
        \eta\bigl(-T(f)_p h(p)+T(h)_p f(p)\bigr)
        =
        g(\nabla_TT,\nu_\heis)_p.
    \]
    Since $H=g(T,T)\,g(\nabla_TT,\nu_\heis)$, the claim follows.
\end{proof}

\subsection{The area functional and first variation formula}

The next step is to introduce a surface measure. To do so, a reference volume measure on the sub-Lorentzian Heisenberg group must first be fixed. There are several natural choices: the three-dimensional Lebesgue measure $\mathcal L^3$, which is a Haar measure; the Hausdorff measure induced by the sub-Riemannian distance; and the Lorentzian Hausdorff measure introduced for Lorentzian length spaces in \cite{McCann2022}. In a previous work, it was shown that these measures coincide in the sub-Lorentzian Heisenberg group; see \cite[Theorem~1.2]{hausdorffSLHeis}. We therefore take as reference volume form the standard Haar volume $\Omega \coloneqq dx\wedge dy\wedge dz$.

\begin{definition}[Induced area]
    \label{def:area}
    Let $\Sigma \subset \heis$ be an embedded $C^1$ spacelike or timelike surface, equipped with a choice of horizontal unit normal $\nu_\heis$ on $\Sigma \setminus \mathrm S(\Sigma)$. We define the induced horizontal area form on $\Sigma \setminus \mathrm S(\Sigma)$ by
    \[
        d\sigma \coloneqq \iota_{\nu_\heis}(\Omega)_{\vert \Sigma}.
    \]
    If $\Sigma$ is such that $\mathrm S(\Sigma)$ is negligible, for instance if $\Sigma$ is $C^{1,1}$, we define its horizontal area by
    \[
        A(\Sigma) \coloneqq \int_{\Sigma\setminus \mathrm S(\Sigma)} d\sigma.
    \]
\end{definition}

The \emph{divergence} associated with the volume form $\Omega$ is the function $\operatorname{div}_{\Omega}(W)$ defined by
\[
    \mathcal L_W \Omega = \operatorname{div}_{\Omega}(W)\,\Omega
\]
for every smooth vector field $W$. If $W=aX+bY+cZ$, then, since $X,Y,Z$ are divergence-free with respect to $\Omega$, we have $\operatorname{div}_{\Omega}(W)=X(a)+Y(b)+Z(c)$. In particular, for a horizontal vector field $W=aX+bY$, we have $\operatorname{div}_{\Omega}(W)=X(a)+Y(b)$. Equivalently, if $\{E_0,E_1\}$ is a local horizontal pseudo-orthonormal frame with $g(E_0,E_0)=-1$ and $g(E_1,E_1)=1$, then for every horizontal vector field $W$,
\[
    \operatorname{div}_{\Omega}(W)
    =
    -g(\nabla_{E_0}W,E_0)+g(\nabla_{E_1}W,E_1).
\]

\begin{lemma}\label{lemma:horizontal_normal_divergence_well_defined}
    Let $\Sigma \subset \heis$ be an embedded $C^2$ spacelike or timelike surface, and let $p\in\Sigma\setminus \mathrm S(\Sigma)$. If $\widetilde\nu$ is a local horizontal extension of $\nu_\heis$ near $p$ such that $g(\widetilde\nu,\widetilde\nu)$ is constant and equal to $g(\nu_\heis(p),\nu_\heis(p))$, then $\operatorname{div}_{\Omega}(\widetilde\nu)(p)$ is independent of the choice of such an extension.
\end{lemma}
\begin{proof}
    Let $\widetilde\nu_1$ and $\widetilde\nu_2$ be two such extensions and set $W\coloneqq \widetilde\nu_1-\widetilde\nu_2$, so that $W=0$ along $\Sigma$. Write $W=aX+bY$ and let $T$ be a local $g$-unit generator of $\mathrm H(\Sigma)$. Since $a$ and $b$ vanish on $\Sigma$, their derivatives in the tangent direction $T$ and $\nabla_TW=T(a)X+T(b)Y=0$ at $p$.

    The extensions also have the same constant $g$-norm, equal to $g(\nu_\heis(p),\nu_\heis(p))$. Thus differentiating the constant function $g(\widetilde\nu_i,\widetilde\nu_i)$ in the direction $\widetilde\nu_i$ gives $g(\nabla_{\widetilde\nu_i}\widetilde\nu_i,\widetilde\nu_i)=0$. Evaluating at $p$ and using $\widetilde\nu_i(p)=\nu_\heis(p)$, we get $g(\nabla_{\nu_\heis}\widetilde\nu_i,\nu_\heis)(p)=0$ for $i=1,2$, and therefore $g(\nabla_{\nu_\heis}W,\nu_\heis)(p)=0$.

    The horizontal frame $\{\nu_\heis,T\}$ is pseudo-orthonormal up to order. The two trace terms of $\operatorname{div}_{\Omega}(W)(p)$ in this frame therefore vanish by the previous two observations, so $\operatorname{div}_{\Omega}(W)(p)=0$, and thus $\operatorname{div}_{\Omega}(\widetilde\nu_1)(p)=\operatorname{div}_{\Omega}(\widetilde\nu_2)(p)$.
\end{proof}

Thus the restriction $\operatorname{div}_{\Omega}(\widetilde\nu)|_{\Sigma\setminus \mathrm S(\Sigma)}$ is intrinsically determined by $\Sigma$ and its chosen horizontal unit normal; we denote it simply by $\operatorname{div}_{\Omega}(\nu_\heis)$.

\begin{proposition}[Divergence formula for the horizontal mean curvature]
    \label{prop:mean_curvature_divergence}
    Let $\Sigma \subset \heis$ be an embedded $C^2$ spacelike or timelike surface, and work on $\Sigma\setminus \mathrm S(\Sigma)$. Let $\nu_\heis$ be its chosen horizontal unit normal and let $T$ be a local $g$-unit generator of $\mathrm H(\Sigma)$. Then
    \[
        H(p)=-\operatorname{div}_{\Omega}(\nu_\heis)(p),
        \qquad p\in \Sigma\setminus \mathrm S(\Sigma).
    \]
\end{proposition}

\begin{proof}
    Choose a local horizontal extension $\widetilde\nu$ of $\nu_\heis$ with constant $g$-norm equal to $g(\nu_\heis(p),\nu_\heis(p))$. By \Cref{lemma:horizontal_normal_divergence_well_defined}, $\operatorname{div}_{\Omega}(\nu_\heis)$ is the restriction of $\operatorname{div}_{\Omega}(\widetilde\nu)$ to $\Sigma\setminus \mathrm S(\Sigma)$. Extend $T$ locally to a horizontal vector field $\widetilde T$ such that $g(\widetilde T,\widetilde T)=g(T,T)$, $g(\widetilde T,\widetilde\nu)=0$, and $\widetilde T=T$ along $\Sigma$. Using the frame $\{\widetilde T,\widetilde\nu\}$ in the divergence formula and then restricting to $\Sigma$ gives
    \[
        \operatorname{div}_{\Omega}(\nu_\heis)
        =
        g(T,T)g(\nabla_T\widetilde\nu,T)
        + g(\nu_\heis,\nu_\heis)g(\nabla_{\widetilde\nu}\widetilde\nu,\widetilde\nu)
        \qquad \text{on } \Sigma\setminus \mathrm S(\Sigma).
    \]
    The second term is zero because $\widetilde\nu$ has constant $g$-norm. Moreover, differentiating the identity $g(\widetilde T,\widetilde\nu)=0$ in the direction $\widetilde T$ and then restricting to $\Sigma$ gives $g(\nabla_TT,\nu_\heis)+g(T,\nabla_T\widetilde\nu)=0$. Hence
    \[
        \operatorname{div}_{\Omega}(\nu_\heis)
        =
        -g(T,T)g(\nabla_TT,\nu_\heis)
        =
        -H,
    \]
    because $H=g(T,T)g(\nabla_TT,\nu_\heis)$.
\end{proof}

Let $\Sigma \subset \heis$ be an embedded $C^2$ spacelike surface. A variation of $\Sigma$ is a smooth map
\[
    \Gamma:(-\varepsilon,\varepsilon)\times \Sigma\to \heis
\]
such that $\Gamma_0=\operatorname{id}_\Sigma$ and, for each $s$ sufficiently small, $\Gamma_s\coloneqq \Gamma(s,\cdot)$ is an embedding. We write $\Sigma_s\coloneqq \Gamma_s(\Sigma)$, denote by $d\sigma_s$ the horizontal area form on $\Sigma_s$, and denote by $\nu_s$ the horizontal unit normal along $\Sigma_s$. The initial velocity of the variation is the vector field along $\Sigma$ defined by
\[
    V(x)\coloneqq\left.\partial_s\Gamma(s,x)\right|_{s=0}.
\]
We say that the variation has \emph{horizontal initial velocity} if its initial velocity satisfies $V\in\Gamma(\Delta|_\Sigma)$. The variation is \emph{compactly supported away from the characteristic set} if the support of $V$ is compact and contained in $\Sigma\setminus \mathrm S(\Sigma)$.

\begin{proposition}[First variation of horizontal area]
    \label{prop:first_variation_formula}
    Let $\Sigma \subset \heis$ be an embedded $C^2$ spacelike surface, and let $\Gamma:(-\varepsilon,\varepsilon)\times \Sigma\to \heis$ be a variation of $\Sigma$ with horizontal initial velocity $V$. Assume that $V$ is compactly supported in $\Sigma\setminus \mathrm S(\Sigma)$. Then
    \[
        \left.\frac{d}{ds}\right|_{s=0}A(\Sigma_s)
        =
        \int_{\Sigma\setminus \mathrm S(\Sigma)}
        H\,g(\nu_\heis,V)\,d\sigma .
    \]
\end{proposition}

\begin{proof}
    Whenever necessary, we choose local extensions of the objects defined along $\Sigma$: the initial velocity $V$ of the variation is extended to a vector field near $\Sigma$, $\nu_\heis$ is extended to a horizontal vector field with constant $g$-norm, and the local generator $T$ of $\mathrm H(\Sigma)$ is extended to a horizontal vector field. For simplicity, we keep the same symbols $V$, $\nu_\heis$, and $T$ for these extensions.

    By \cref{def:area}, $A(\Sigma_s)=\int_{\Sigma_s}d\sigma_s$, where the area form on $\Sigma_s$ is $\sigma_s=\iota_{\nu_s}\Omega|_{\Sigma_s}$. Therefore,
    \[
        \left.\frac{d}{ds}\right|_{s=0}
        A(\Sigma_s)
        =
        \left.\frac{d}{ds}\right|_{s=0}
        \int_{\Sigma_s}\iota_{\nu_s}\Omega.
    \]
    By the Reynolds Transport Theorem, we have that
    \[
        \left.\frac{d}{ds}\right|_{s=0}
        \int_{\Sigma_s} \iota_{\nu_s}\Omega
        =
        \int_{\Sigma\setminus \mathrm S(\Sigma)}
        \left(
        \left.\partial_s\right|_{s=0}\iota_{\nu_s}\Omega
        +
        \mathcal L_{V}(\iota_{\nu_\heis}\Omega)
        \right).
    \]
    Along the variation, we set
    \[
        \nu_s(\Gamma_s(x))
        =
        \alpha_s(x)X_{\Gamma_s(x)}
        +
        \beta_s(x)Y_{\Gamma_s(x)}, \qquad \dot\nu(x)
        \coloneq
        \left.\partial_s\alpha_s(x)\right|_{s=0}X_x
        +
        \left.\partial_s\beta_s(x)\right|_{s=0}Y_x
    \]
    Then, using that $X$, $Y$, and $\Omega$ are fixed,
    \begin{align*}
        \left.\partial_s\right|_{s=0}\iota_{\nu_s}\Omega
         & =
        \left.\partial_s\right|_{s=0}
        \bigl(\alpha_s\,\iota_X\Omega+\beta_s\,\iota_Y\Omega\bigr) =
        \left.\partial_s\alpha_s\right|_{s=0}\iota_X\Omega
        +
        \left.\partial_s\beta_s\right|_{s=0}\iota_Y\Omega =
        \iota_{\dot\nu}\Omega .
    \end{align*}
    Differentiating the identity $g(\nu_s,\nu_s)=-1$ at $s=0$ gives $g(\dot\nu,\nu_\heis)=0$. Since $\dot\nu$ is horizontal and $T \in \mathrm H(\Sigma)$, we have $\Omega(\dot\nu,T,W)=0$ for every vector field $W$ tangent to $\Sigma$, and thus $\iota_{\dot\nu}\Omega|_\Sigma=0$.

    We now simplify the Lie derivative term using the following two identities. First, for every smooth function $f$, we have
    \begin{align*}
        \mathcal L_{f\nu_\heis}(\iota_{\nu_\heis}\Omega)
         & =
        \iota_{\nu_\heis}\mathcal L_{f\nu_\heis}\Omega
        +
        \iota_{[f\nu_\heis,\nu_\heis]}\Omega =
        f\iota_{\nu_\heis}( \mathcal{L}_{\nu_\heis} \Omega) + \iota_{\nu_\heis}(df \wedge \iota_{\nu_\heis} \Omega) +f \iota_{[\nu_\heis,\nu_\heis]} \Omega -\nu_\heis[f] \iota_{\nu_\heis} \Omega
        \\& =
        f\operatorname{div}_{\Omega}(\nu_\heis)\iota_{\nu_\heis}\Omega + \nu_\heis[f] \iota_{\nu_\heis}\Omega - \nu_\heis[f] \iota_{\nu_\heis}\Omega =  f\operatorname{div}_{\Omega}(\nu_\heis)\iota_{\nu_\heis}\Omega.
    \end{align*}
    Because $fT$ is tangent to $\Sigma$, we must have that $(\mathcal L_{fT}(d \sigma))|_\Sigma
        =
        \mathcal L_{fT}^{\Sigma}(d\sigma)$, where $\mathcal L^{\Sigma}$ the intrinsic Lie derivative of $\Sigma$. Then, Cartan's formula (on $\Sigma$), gives
    \[
        \mathcal L_{fT}^{\Sigma}d\sigma
        =
        d_\Sigma\bigl(\iota_{fT}d\sigma\bigr)
        +
        \iota_{fT}d_\Sigma(d\sigma)
        =
        d_\Sigma\bigl(\iota_{fT}d\sigma\bigr),
    \]
    where $d_\Sigma$ denotes the intrinsic exterior derivative on $\Sigma$, and where we used that on the surface $\Sigma\setminus \mathrm S(\Sigma)$, $d\sigma$ is a top-degree $2$-form and therefore its exterior derivative is zero, i.e. $d_\Sigma(d\sigma)=0$.

    Since $\Sigma$ is spacelike, we can write $V=g(V,T)T-g(V,\nu_\heis)\nu_\heis$. Therefore,
    \begin{align*}
        \mathcal L_V(\iota_{\nu_\heis}\Omega)\big|_\Sigma
         & =
        \mathcal L_{g(V,T)T-g(V,\nu_\heis)\nu_\heis}d\sigma =
        \mathcal L_{g(V,T)T}d\sigma
        -
        \mathcal L_{g(V,\nu_\heis)\nu_\heis}d\sigma                   \\
         & =
        d\bigl(\iota_{g(V,T)T}d\sigma\bigr)
        -
        g(V,\nu_\heis)\operatorname{div}_{\Omega}(\nu_\heis)\,d\sigma \\
         & =
        d\bigl(\iota_{g(V,T)T}d\sigma\bigr)
        +
        H\,g(\nu_\heis,V)\,d\sigma .
    \end{align*}

    Since $\iota_{g(V,T)T}d\sigma$ is compactly supported in $\Sigma\setminus \mathrm S(\Sigma)$, Stokes' theorem gives
    \[
        \int_{\Sigma\setminus \mathrm S(\Sigma)}
        d\bigl(\iota_{g(V,T)T}d\sigma\bigr)
        =
        0,
    \]
    which concludes the proof.
\end{proof}

\subsection{Lorentzian isoperimetric problem and volume-preserving variations}\label{sec:volume_preserving_variations}

In a metric measure space $(X,\mathsf d,\mathfrak m)$, the \emph{isoperimetric problem}
consists in determining, for each prescribed volume $v>0$, the sets with minimal
surface area among all sets of measure $v$. More precisely, one studies minimisers of
\begin{equation}
    \label{eq:isopmms}
    \inf\{A(E) \mid E \subseteq X \text{ measurable, } \mathfrak m(E)=v\},
\end{equation}
where $A(E)$ denotes an appropriate notion of surface area, such as the
De Giorgi perimeter in the abstract sense of \cite{AmbrosioPer}. In the Lorentzian setting, the isoperimetric problem is reformulated by replacing the infimum in \cref{eq:isopmms} with a supremum and the perimeter with a suitable timelike Lorentzian perimeter, and by restricting the set of competitors to achronal hypersurfaces whose cone from a fixed point, the \emph{pole}, has prescribed volume. Here, a subset \(A\) of a Lorentzian space
is called \emph{achronal} if no two points of \(A\) are chronologically related;
that is, for every \(p,q\in A\), one has \(p\not\ll q\).

In the sub-Lorentzian Heisenberg group, the formulation is as follows.
For \(p\in\heis\) and \(q\in I^+(p)\), let \(\gamma_{p,q}\) denote the
future-directed maximizing geodesic from \(p\) to \(q\), parametrized on
\([0,1]\) with constant speed. Given a set \(\Sigma\subset I^+(p)\), we define the
\emph{cone over \(\Sigma\) with vertex \(p\)} by
\[
    C(\Sigma,p) \coloneqq \bigcup_{q\in \Sigma}\gamma_{p,q}([0,1]).
\]
When \(p=e\), we simply write \(C(\Sigma)\) instead of \(C(\Sigma,e)\).
By left-invariance, the pole can always be fixed at the identity \(e\), and one seeks
to find the maximisers of
\begin{equation}
    \label{eq:LoreIsop}
    \max \left\{
    A(\Sigma)
    \mid
    \text{\(C^2\) spacelike compact hypersurface } \Sigma \subset I^+(e)\subset \heis,
    \mathcal{L}^3(C(\Sigma))=v
    \right\}.
\end{equation}

A fully general formulation would not require the surface \(\Sigma\) to be regular, and would instead use a non-smooth notion of surface area, such as the \emph{timelike Minkowski content} studied in \cite{MondinoIsopLor}. Furthermore, it is not difficult to see that \cref{eq:LoreIsop} is infinite unless one imposes an additional constraint on the class of competitors. At present, two such constraints are studied in the literature. In \cite{BahnEhrlich}, the competitor is required to have the ``solid angle'', or aperture, of its cone bounded by a fixed constant. In other words, the cone may have prescribed volume, but it cannot spread over arbitrarily many future directions. In \cite{MondinoIsopLor}, the authors instead imposes a lower bound \(\inf_{q\in\Sigma}\uptau(p,q)\geq \uptau_0>0\) on the time separation from the pole \(p\). Thus, the cone may open in many directions, but the surface cannot dip arbitrarily close to the pole.

The main point is that, once the additional constraint has been fixed, smooth maximisers can be studied through admissible volume-preserving variations, which are a focus of this work. We start by defining variations that are adapted to the conical nature of the problem.

\begin{definition}[Radial variation]
    Let $p\in\heis$ and let $\Sigma\subset I^+(p)$ be a $C^1$ hypersurface, possibly with boundary. For every $x\in\Sigma$ and given $\lambda_x:=(\exp_p)^{-1}(x)$, the curve $\gamma_x(t):=\exp_p(t\lambda_x)$ is the unique maximising geodesic parametrised by constant speed on $[0, 1]$ that satisfies
    $\gamma_x(0)=p$ and $\gamma_x(1)=x$. A \emph{radial variation} of $\Sigma$ is a family of the form
    \[
        \Gamma_s(x):=\exp_p(r_s(x) \lambda_x),
    \]
    where $r: \ointerval{-\epsilon}{\epsilon} \times \Sigma\to(0,\infty)$ is smooth in $(s,x)$ and satisfies
    $r_0\equiv 1$. We write $\Sigma_s \coloneqq \Gamma_s(\Sigma)$ and $\dot r(x) \coloneqq\partial_s r_s(x)|_{s=0}$.
    The \emph{cone} over $\Sigma_s$ is
    \[
        C_s
        \coloneqq
        \{
        \exp_p(\rho\lambda_x)
        \mid
        x\in\Sigma,\ 0\leq \rho\leq r_s(x)
        \} = F_s(\Sigma\times[0,1]),
    \]
    where we set $F_s:\Sigma\times[0,1]\to\heis : (x,t) \mapsto \exp_p(t r_s(x)\lambda_x)$. The \emph{cone velocity} is the vector field defined by $X_0(\exp_p(t\lambda_x))
        :=
        \left.\partial_s\right|_{s=0}\exp_p(t r_s(x)\lambda_x)$ for every $\exp_p(t\lambda_x)\in
        C_0$.

    We say the radial variation is \emph{compactly supported} if $\operatorname{supp}\dot r$ is compactly contained in the interior of $\Sigma\setminus \mathrm S(\Sigma)$.
\end{definition}

The next lemma computes the infinitesimal velocity produced by a radial variation.

\begin{lemma}[Velocity of a radial cone variation]
    \label{lemma:velocity_radial_cone_variation}
    Let $\Gamma_s$ be a radial variation. At every point $\exp_p(t\lambda_x)\in
        C_0$ of the initial cone, the \emph{cone velocity} is given by
    \begin{equation}
        \label{eq:X0}
        X_0(\exp_p(t\lambda_x)) =
        -t\,\dot r(x)\,\tau(p,x)\,
        \bigl(\operatorname{grad}_{\heis}\tau(p,\cdot)\bigr)_{\exp_p(t\lambda_x)}
    \end{equation}
    In particular, the initial velocity of the variation $\Sigma_s$ is
    \begin{equation}
        \label{eq:V(x)}
        V(x)=X_0(x)=
        -\dot r(x)\,\tau(p,x)\,\bigl(\operatorname{grad}_{\heis}\tau(p,\cdot)\bigr)_x .
    \end{equation}
\end{lemma}

\begin{proof}
    Since $\lambda_x$ is fixed when differentiating with respect to $s$, the
    chain rule gives
    \[
        X_0(\exp_p(t\lambda_x))
        =
        t\,\dot r(x)\,\frac{d}{dt}\exp_p(t\lambda_x).
    \]
    Thus \cref{eq:X0} follows from \cref{eq:eikonaleq2}. Finally, since
    $F_0(x,1)=x$, the cone velocity restricts on the top surface to the initial
    velocity $V$ of the variation $\Gamma_s$. Hence $X_0(x)=V(x)$ for every
    $x\in\Sigma$, and \cref{eq:V(x)} follows by setting $t=1$ in \cref{eq:X0}.
\end{proof}

We next study how the volume of the cone changes under a compactly supported radial variation. The formula expresses the first variation entirely as a boundary integral over the moving top surface.

\begin{proposition}[First variation of the cone volume]
    \label{prop:first_variation_cone_volume}
    Let $\Gamma_s$ be a compactly supported radial variation of a $C^2$ spacelike hypersurface $\Sigma$. Then
    \[
        \left.\frac{d}{ds}\right|_{s=0}\operatorname{Vol}(C_s)
        =
        \int_{\Sigma\setminus \mathrm S(\Sigma)}
        \dot r(x)\tau(p,x)\,
        g(\operatorname{grad}_{\heis}\tau(p,\cdot),\nu_\heis)\,d\sigma(x).
    \]
\end{proposition}

\begin{proof}
    By definition, $\operatorname{Vol}(C_s)=\int_{C_s}\Omega$. Applying the
    Reynolds Transport Theorem gives
    \[
        \left.\frac{d}{ds}\right|_{s=0}\operatorname{Vol}(C_s)
        =
        \int_{C_0}\mathcal L_{X_0}\Omega =
        \int_{\partial C_0}\iota_{X_0}\Omega,
    \]
    where we use Cartan's formula on the last equality
    \[
        \mathcal L_{X_0}\Omega
        =
        d(\iota_{X_0}\Omega)+\iota_{X_0}(d\Omega)
        =
        d(\iota_{X_0}\Omega),
    \]
    together with Stokes theorem.
    Indeed, $\Omega$ is the ambient volume form on $\heis$, hence a top-degree
    form on a $3$-dimensional manifold. Therefore $d\Omega$ is a $4$-form, so
    $d\Omega=0$.

    The boundary of the initial cone decomposes as $\partial C_0=\Sigma\cup C(\partial\Sigma)$, where $\Sigma=F_0(\Sigma\times\{1\})$, $C(\partial\Sigma)=F_0(\partial\Sigma\times[0,1])$, and $\{p\}=F_0(\Sigma\times\{0\})$. Therefore
    \[
        \int_{\partial C_0}\iota_{X_0}\Omega
        =
        \int_{\Sigma}\iota_{X_0}\Omega
        +
        \int_{C(\partial\Sigma)}\iota_{X_0}\Omega.
    \]
    The compact support assumption gives $\dot r=0$ near $\partial\Sigma$ and thus, by \cref{eq:X0}, we have $X_0=0$ on $C(\partial\Sigma)$ and $\int_{C(\partial\Sigma)}\iota_{X_0}\Omega=0$.
    On $\Sigma=F_0(\Sigma\times\{1\})$, we have $X_0|_\Sigma=V$, where $V$ is the
    initial velocity of the variation $\Gamma_s$. Therefore
    \[
        \int_{\partial C_0}\iota_{X_0}\Omega
        =
        \int_{\Sigma\setminus \mathrm S(\Sigma)}\iota_V\Omega.
    \]
    Using $d\sigma=\iota_{\nu_\heis}\Omega|_\Sigma$ and writing $V=V^\parallel-g(V,\nu_\heis)\nu_\heis$ with $V^\parallel\in \mathrm H(\Sigma)$, we obtain
    \[
        \iota_V\Omega|_\Sigma
        =
        \iota_{-g(V,\nu_\heis)\nu_\heis}\Omega|_\Sigma
        =
        -g(V,\nu_\heis)\,\iota_{\nu_\heis}\Omega|_\Sigma
        =
        -g(V,\nu_\heis)d\sigma.
    \]

    We conclude by \cref{lemma:velocity_radial_cone_variation} and \cref{eq:V(x)}, noting that
    \[
        \int_{\Sigma\setminus \mathrm S(\Sigma)}\iota_V\Omega = -\int_{\Sigma\setminus \mathrm S(\Sigma)}g(V,\nu_\heis)d\sigma =
        \int_{\Sigma\setminus \mathrm S(\Sigma)}
        \dot r(x)\,\tau(p,x)\,
        g(\operatorname{grad}_{\heis}\tau(p,\cdot),\nu_\heis)\,d\sigma(x),
    \]
    and the proof is complete.
\end{proof}

Combining the first variation of horizontal area with the first variation of cone volume gives the expected Euler--Lagrange condition for radial volume-constrained critical points. Namely, such surfaces have constant horizontal mean curvature away from the characteristic set.

\begin{proposition}[Radial volume-preserving stationarity implies CMC]
    \label{prop:CMC}
    Let $\Sigma\subset I^+(p)$ be a connected embedded $C^2$ spacelike surface.
    Assume that, for every compactly supported radial variation $\Gamma_s$ of
    $\Sigma$ whose associated cones satisfy
    \[
        \operatorname{Vol}(C_s)=\operatorname{Vol}(C_0)
        \qquad\text{for all sufficiently small }s,
    \]
    one has
    \[
        \left.\frac{d}{ds}\right|_{s=0}A(\Sigma_s)=0.
    \]
    Then the horizontal mean curvature $H$ is constant on
    $\Sigma\setminus \mathrm S(\Sigma)$.
\end{proposition}

\begin{proof}
    Since $\Sigma$ is spacelike, $\nu_\heis$ is timelike. Moreover,
    $\operatorname{grad}_{\heis}\tau(p,\cdot)$ is timelike on $I^+(p)$, and therefore
    $g(\operatorname{grad}_{\heis}\tau(p,\cdot),\nu_\heis)$ does not vanish on
    $\Sigma\setminus \mathrm S(\Sigma)$. Let $\varphi\in C_c^\infty(\Sigma\setminus \mathrm S(\Sigma))$ and choose the radial variation defined by $r_s=1+s\varphi$, so that $\dot r=\varphi$. \cref{lemma:velocity_radial_cone_variation}
    and \cref{prop:first_variation_formula} give
    \begin{equation}
        \label{eq:areastationary}
        \left.\frac{d}{ds}\right|_{s=0}A(\Sigma_s)
        =
        -\int_{\Sigma\setminus \mathrm S(\Sigma)}
        H(x)\varphi(x)\,\tau(p,x)\,g(\operatorname{grad}_{\heis}\tau(p,\cdot),\nu_\heis)\,d\sigma(x).
    \end{equation}
    On the other hand, \cref{prop:first_variation_cone_volume} gives
    \begin{equation}
        \label{eq:volumestationary}
        \left.\frac{d}{ds}\right|_{s=0}\operatorname{Vol}(C_s)
        =
        \int_{\Sigma\setminus \mathrm S(\Sigma)}
        \varphi(x)\,\tau(p,x)\,g(\operatorname{grad}_{\heis}\tau(p,\cdot),\nu_\heis)\,d\sigma(x).
    \end{equation}
    We wish to argue that, if \cref{eq:volumestationary} is zero, then \cref{eq:areastationary} vanishes as well. This would imply, by the fundamental lemma of the calculus of variations, that $H$ must be constant on $\Sigma\setminus \mathrm S(\Sigma)$.

    Choose
    $\eta\in C_c^\infty(\Sigma\setminus \mathrm S(\Sigma))$ such that
    \begin{equation}
        \label{eq:etachosen}
        \int_{\Sigma\setminus \mathrm S(\Sigma)}
        \eta(x)\tau(p,x)g(\operatorname{grad}_{\heis}\tau(p,\cdot),\nu_\heis)\,d\sigma(x)\neq 0.
    \end{equation}
    This is always possible because    \(\tau(p,x)g(\operatorname{grad}_{\heis}\tau(p,\cdot),\nu_\heis)\)
    is continuous and nowhere vanishing on \(\Sigma\setminus \mathrm S(\Sigma)\).
    For small $(s,a)$, consider the radial variations with $r_{s,a}=1+s\varphi+a\eta$, and note that \(\partial_a
    \operatorname{Vol}(C_{s,a})|_{(0,0)}\neq 0\) because of \cref{eq:etachosen}.
    By the implicit function theorem, there is a smooth function
    $a=a(s)$, with $a(0)=0$, such that $\operatorname{Vol}(C_{s,a(s)})=
        \operatorname{Vol}(C_0)$ for all small $s$. If \cref{eq:volumestationary} is zero, then $a'(0)=0$ because, by the chain rule \(0
    =
    \left.\partial_s\right|_{(0,0)}\operatorname{Vol}(C_{s,a})
    +
    a'(0)
    \left.\partial_a\right|_{(0,0)}\operatorname{Vol}(C_{s,a})\) and therefore
    \[
        a'(0)
        =
        -
        \frac{
            \displaystyle
            \int_{\Sigma\setminus \mathrm S(\Sigma)}
            \varphi(x)\tau(p,x)
            g(\operatorname{grad}_{\heis}\tau(p,\cdot),\nu_\heis)\,d\sigma(x)
        }{
            \displaystyle
            \int_{\Sigma\setminus \mathrm S(\Sigma)}
            \eta(x)\tau(p,x)
            g(\operatorname{grad}_{\heis}\tau(p,\cdot),\nu_\heis)\,d\sigma(x)
        }.
    \]
    The denominator is not zero because of \cref{eq:etachosen} while the numerator is zero because we assume that \cref{eq:volumestationary} is zero. This implies that
    \[
        \left.\frac{d}{ds}\right|_{s=0} r_{s,a(s)}(x)
        =
        \left.\frac{d}{ds}\right|_{s=0}\bigl(1+s\varphi(x)+a(s)\eta(x)\bigr)
        =
        \varphi(x)+a'(0)\eta(x)
        =
        \varphi(x) = \left.\frac{d}{ds}\right|_{s=0} r_{s}(x)
    \]
    and therefore, by \cref{lemma:velocity_radial_cone_variation}
    and \cref{prop:first_variation_formula} again,
    \[
        \left.\frac{d}{ds}\right|_{s=0}A(\Sigma_s)
        =
        \left.\frac{d}{ds}\right|_{s=0}A(\Sigma_{s,a(s)}) = 0,
    \]
    concluding the proof.
\end{proof}

Following the terminology in the literature, see \cite{RitoreRosales}, we say that a surface $\Sigma$ satisfying the hypotheses of \cref{prop:CMC} is \emph{area-stationary under volume-preserving radial variations}. In particular, any smooth isoperimetric maximiser satisfying the aperture constraint of \cite{BahnEhrlich} or the time-separation constraint of \cite{MondinoIsopLor} must satisfy this stationarity condition for all admissible compactly supported radial variations and have constant mean curvature.

A striking result in the sub-Riemannian Heisenberg group is that metric balls are not isoperimetric minimisers \cite{BallNotIsop}, see also \cite[Example~4.13]{Capogna2007}. We prove the analogous statement in the sub-Lorentzian Heisenberg group by showing that sub-Lorentzian pseudo-spheres do not have constant horizontal mean curvature.

\begin{theorem}
    \label{thm:pseudo-sphere-not-CMC}
    Let
    \[
        \Sigma \coloneqq \{q\in I^+(0):\tau(0,q)=1\}
    \]
    be the unit sub-Lorentzian pseudo-sphere. Then no nonempty relatively open
    smooth patch of $\Sigma$ is area-stationary under volume-preserving radial
    variations. In particular, such a patch cannot be a smooth isoperimetric
    maximiser for either the aperture-constrained problem of \cite{BahnEhrlich}
    or the lower-time-separation constrained problem of \cite{MondinoIsopLor}.
\end{theorem}

\begin{proof}
    By the description of the exponential map in \cref{eq:exponentialmap}, the
    unit pseudo-sphere satisfies
    \[
        \Sigma
        =
        \left\{
        \exp_0(u,v,w)
        \;\middle|\;
        u^2-v^2=1,\ u>\abs{v},\ w\in\mathbb R
        \right\}.
    \]
    Since $\Sigma$ is the level set $\tau(0,\cdot)=1$, every horizontal tangent vector $V\in \mathrm H(\Sigma)$ satisfies
    \[
        g\bigl(\operatorname{grad}_{\heis}\tau(0,\cdot),V\bigr)
        =
        V\bigl(\tau(0,\cdot)\bigr)
        =
        0.
    \]
    Thus $\operatorname{grad}_{\heis}\tau(0,\cdot)$ is horizontally normal to $\Sigma$, and since it is unit, see \cref{eq:eikonaleq1}, we get $\nu_\heis=\operatorname{grad}_{\heis}\tau(0,\cdot)$, with the chosen orientation.

    Differentiating the geodesic
    \(\gamma(t)=\exp_0(tu,tv,tw)\) given by \cref{eq:exponentialmap}, we get
    \[
        \dot\gamma(1)
        =
        \bigl(u\cosh w+v\sinh w\bigr)X
        +
        \bigl(v\cosh w+u\sinh w\bigr)Y.
    \]
    Since \(u^2-v^2=1\) on \(\Sigma\), this curve has unit speed. Therefore,
    by the relation between unit speed maximising geodesics and the time
    separation gradient,
    \[
        \nu_\heis(\exp_0(u,v,w))
        =
        -\bigl(u\cosh w+v\sinh w\bigr)X
        -
        \bigl(v\cosh w+u\sinh w\bigr)Y.
    \]
    The divergence formula gives $\operatorname{div}_{\Omega}(\nu_\heis) = X(a)+Y(b)$, where the functions $a,b$ are defined by
    \[
        a\circ\exp_0(u,v,w)
        =
        -u\cosh w-v\sinh w,
    \]
    and
    \[
        b\circ\exp_0(u,v,w)
        =
        -v\cosh w-u\sinh w.
    \]
    On $\Sigma$, we have $u=\sqrt{1+v^2}$. The derivatives $X(a)$ and $Y(b)$ are computed by pulling $X$ and $Y$ back through the differential of $\exp_0$:
    \[
        (X(a))(\exp_0(u,v,w))
        =
        d(a\circ\exp_0)_{(u,v,w)}
        \left[
            (d\exp_0)_{(u,v,w)}^{-1}
            \bigl(X_{\exp_0(u,v,w)}\bigr)
            \right],
    \]
    and
    \[
        (Y(b))(\exp_0(u,v,w))
        =
        d(b\circ\exp_0)_{(u,v,w)}
        \left[
            (d\exp_0)_{(u,v,w)}^{-1}
            \bigl(Y_{\exp_0(u,v,w)}\bigr)
            \right].
    \]
    Using the explicit formula for $\exp_0$ from \cref{eq:exponentialmap} and
    then restricting to $u^2-v^2=1$, this gives
    \[
        H(v,w) = - \operatorname{div}_{\Omega}(\nu_\heis)
        =
        \frac{w^2\cosh w-w\sinh w}
        {w\sinh w-2\cosh w+2}.
    \]
    Here we have used \cref{prop:mean_curvature_divergence} for the first equality. This function is clearly not constant on any neighbourhood.
\end{proof}

\section{Characterisation of constant mean curvature symmetric surfaces}
\label{section:boost-symmetric-surfaces}

\subsection{Boost-symmetric constant mean curvature surfaces}

\label{subsec:boost-symmetric-surfaces}

The previous section showed that smooth isoperimetric candidates must have
constant horizontal mean curvature. Since such candidates are also expected to
possess some symmetry, we restrict our attention to \emph{boost-symmetric
    surfaces}, namely surfaces invariant under the action of \(O^+(1,1)\) on the
sub-Lorentzian Heisenberg group, as described in
\cref{thm:sub_lorentzian_isometries}. More concretely, for each \(u\in\mathbb R\), let
\[
    B_u(x,y,z)
    =
    (x\cosh u+y\sinh u,\; x\sinh u+y\cosh u,\; z).
\]
Then \(\Sigma\) is boost-symmetric if $B_u(\Sigma)=\Sigma$ for every $u\in\mathbb R$.

Locally around each non-characteristic point \(p\), choose a unit generator
\(T\) of \(\mathrm H(\Sigma)\). The corresponding (horizontal) integral curve
\(\gamma:\ointerval{-\epsilon}{\epsilon}\to\heis\) projects to the
\((x,y)\)-plane as a curve \(\sigma\) whose signed curvature at \(t=0\) equals
\(H(p)\), by \cref{prop:mean_curvature_via_minkowski}. When \(\Sigma\) is boost-symmetric, applying the boosts to this profile gives a
local parametrisation of \(\Sigma\) away from the \(z\)-axis, namely
\(S(t,u)=B_u(\gamma(t))\). In particular, \(H\) is constant whenever the
projected curve \(\sigma\) has constant signed curvature.

Conversely, if \(\sigma\) is a curve in the Minkowski plane with constant
signed curvature, then any horizontal lift \(\gamma\) of \(\sigma\) to \(\heis\)
can be boosted to produce a surface \(S(t,u)=B_u(\gamma(t))\). Wherever this
map is an immersed spacelike surface, its image is boost-symmetric and has
constant horizontal mean curvature, by
\cref{prop:mean_curvature_via_minkowski} and the fact that boosts are
isometries. Since we are interested in achronal spacelike surfaces, the
projected horizontal profiles relevant to the construction must be spacelike.

The classification of constant-curvature curves in the Minkowski plane is
classical. We recall the relevant statement next, and refer the reader to the
discussion preceding \cite[Theorem~2.10]{Lopez2014}.

\begin{proposition}\label{prop:planar_constant_curvature_curves}
    Up to the action of the full Lorentz isometry group
    \(O(1,1)\ltimes\mathbb R^2\) and reparametrisation, the spacelike curves in the
    Minkowski plane with constant curvature \(k\geq 0\) are precisely the straight
    lines when \(k=0\), and, when \(k>0\), the hyperbolas
    \[
        \gamma(t)=\frac{1}{k}(\cosh(kt),\sinh(kt)).
    \]
\end{proposition}

We now carry out the construction explicitly. Starting from the spacelike
constant-curvature curves in the Minkowski plane classified in
\cref{prop:planar_constant_curvature_curves}, we horizontally lift them to
\(\heis\) and then apply the boosts to the lifted curves to obtain surfaces.

We start with the nonzero curvature case \(k>0\). Choosing the arclength
parametrisation and allowing translations in the Minkowski plane, as well as
the two choices of branch, a representative can be written as
\begin{equation}
    \label{eq:genericMinkowkik>0}
    (x(t),y(t))
    =
    \left(
    \frac{\varepsilon}{k}\cosh(kt)+c,\;
    \frac{\varepsilon}{k}\sinh(kt)+d
    \right),
    \qquad \varepsilon\in\{-1,1\},\ c,d\in\mathbb R.
\end{equation}
Using \cref{eq:horizontallift}, the horizontal lift of
\cref{eq:genericMinkowkik>0} with \(z(t_0)=0\) is given by
\[
    z(t)
    =
    \frac{1}{2k}
    \left(
    t-t_0
    +c\varepsilon\sinh(kt)-c\varepsilon\sinh(kt_0)
    -d\varepsilon\cosh(kt)+d\varepsilon\cosh(kt_0)
    \right).
\]
We can then boost the lifted curve $\gamma(t) = (x(t),y(t),z(t))$ according to \(S(t,u)=B_u(\gamma(t))\). We now remove the inessential constants, namely those which can be absorbed by
ambient isometries, dilations, or reparametrisation.

The choice of \(t_0\) only fixes the additive constant in the horizontal lift.
Since vertical translations are sub-Lorentzian isometries and commute with
boosts, we may absorb this constant by a vertical translation. Thus, up to
isometry, the profile curve may be written as
\begin{equation}
    \label{eq:profilenot0}
    \left(
    \frac{\varepsilon}{k}\cosh(kt)+c,\;
    \frac{\varepsilon}{k}\sinh(kt)+d,\;
    \frac{1}{2k}
    \left(
    t+c\varepsilon\sinh(kt)-d\varepsilon\cosh(kt)+d\varepsilon
    \right)
    \right).
\end{equation}

The sign \(\varepsilon\) can also be removed: the isometry
\((x,y,z)\mapsto(-x,-y,z)\) is a time-inverting boost, which commutes with boosts, changes
\((\varepsilon,k,c,d)\) into \((-\varepsilon,k,-c,-d)\). Since \(c,d\) range
over all real values, the same family is obtained with \(\varepsilon=1\).

Finally, the parameter \(k>0\) can be absorbed by Heisenberg dilations $\delta_k(x,y,z)=(kx,ky,k^2z)$. Indeed, applying \(\delta_k\) to \cref{eq:profilenot0} with $\epsilon = 1$ gives
\[
    \left(
    \cosh(kt)+kc,\;
    \sinh(kt)+kd,\;
    \frac{1}{2}\bigl(kt+kc\sinh(kt)-kd\cosh(kt)+kd\bigr)
    \right).
\]
Since \(c,d\) range over all real values, we may rename \(kc\) and \(kd\) as
\(c\) and \(d\), and then replace \(kt\) by \(t\). Thus, up to Heisenberg
dilation and reparametrisation, it is enough to consider the case \(k=1\).
The dilation \(\delta_\lambda\) preserves boost symmetry and do not alter the causal character of \((c,d)\) in the Minkowski plane. Dilations do not however preserve the numerical value of the
mean curvature. It sends a surface with constant mean curvature \(H\)
to a surface with constant mean curvature \(\lambda^{-1}H\).

Thus, up to vertical translations, the isometry
\((x,y,z)\mapsto(-x,-y,z)\), Heisenberg dilations, and reparametrisation, the
nonzero-curvature case is represented by
\begin{equation}\label{eq:general_expression_CMC_surface}
    \begin{split}
        S_{(c,d)}(t,u)
         & =
        B_u\underbrace{\left(
               \cosh(t)+c,\;
               \sinh(t)+d,\;
               \frac{1}{2}\left(t+c\sinh(t)-d\cosh(t)+d\right)
               \right)}_{=: \gamma_{(c, d)}(t)} \\
         & =
        \begin{pmatrix}
            \cosh(t+u) +c\cosh(u) +d\sinh(u) \\
            \sinh(t+u) +c\sinh(u) +d\cosh(u) \\
            \dfrac{1}{2} \left( t +c\sinh(t) -d\cosh(t) +d \right)
        \end{pmatrix}.
    \end{split}
\end{equation}
We say that such a surface has \emph{timelike}, \emph{spacelike}, or
\emph{null} parameter according to the causal character of the vector
\((c,d)\) in the Minkowski plane. We shall classify these surfaces according
to this vector below. We often refer to the surface associated with the lift of
the \((c,d)\)-translated hyperbola as the \emph{\((c,d)\)-surface}.

We now turn to the case \(k=0\). Choosing an arclength parametrisation and
allowing translations in the Minkowski plane, a spacelike line can be written as
\[
    (x(t),y(t))
    =
    \bigl(\sinh(\beta)t+\alpha,\;
    \varepsilon\cosh(\beta)t+\alpha'\bigr),
    \qquad
    \alpha,\alpha',\beta\in\mathbb R,\quad \varepsilon\in\{-1,1\}.
\]
Translating the parameter \(t\) only changes the point of the line chosen as
origin. Hence, up to reparametrisation, we may assume \(\alpha'=0\). As above,
the additive constant in the horizontal lift only produces a vertical
translation of the boosted surface, so we choose the lift with \(z(0)=0\).

Using \cref{eq:horizontallift}, the horizontal lift is given by $z(t)=
    \frac{\alpha\varepsilon}{2}\cosh(\beta)t$.
Boosting this horizontal lift, we obtain
\begin{equation}\label{eq:maximal_surface_parametrization}
    \begin{split}
        S_{(\epsilon,\alpha, \beta)}(t,u)
         & =
        B_u\underbrace{\left(
               \sinh(\beta)t+\alpha,\;
               \varepsilon\cosh(\beta)t,\;
               \frac{\alpha\varepsilon}{2}\cosh(\beta)t
               \right)}_{=: \gamma_{(\epsilon,\alpha, \beta)}(t)}           \\
         & = \begin{pmatrix}
                 \varepsilon t \sinh(u+\varepsilon\beta) + \alpha\cosh(u)   \\
                 \varepsilon t \cosh(u +\varepsilon\beta) + \alpha \sinh(u) \\
                 \dfrac{\alpha \varepsilon}{2} \cosh(\beta) t
             \end{pmatrix}.
    \end{split}
\end{equation}
These are the boost-symmetric surfaces with zero horizontal mean curvature,
which we shall refer to as the \emph{maximal surfaces}.

Although the rest of this work is devoted to a case-by-case classification of
these surfaces up to the sub-Lorentzian isometries described in
\cref{thm:sub_lorentzian_isometries,def:HeisenbergIsometries}, the following
proposition records a basic symmetry of the parameter space that will be used
repeatedly later.

\begin{proposition}\label{prop:timed_rotation_behaviour}
    $S_{(c,d)}$ is isometric to $S_{(c,-d)}$ by any timed rotation.
\end{proposition}

\begin{proof}
    Any timed rotation is of the form \(A\circ B_v\), where \(B_v\) is a boost and
    \(A(x,y,z)=(x,-y,-z)\).  Since \(S_{(c,d)}\) is
    boost-symmetric, it is enough to apply \(A\). Moreover
    \(A\circ B_u=B_{-u}\circ A\), and a direct computation gives
    \(A\gamma_{(c,d)}(t)=\gamma_{(c,-d)}(-t)\). Hence
    \(A(S_{(c,d)}(t,u))=B_{-u}\gamma_{(c,-d)}(-t)\), which parametrises
    \(S_{(c,-d)}\) after the change of variables \((t,u)\mapsto(-t,-u)\).
\end{proof}

We conclude this section by determining when the parametrisations introduced in \cref{eq:general_expression_CMC_surface,eq:maximal_surface_parametrization} are
regular. In \cref{prop:real_singularity}, we shall see that in the non-regular nonzero-curvature case the corresponding
image is genuinely singular, and it will be studied separately.

\begin{proposition}\label{prop:non_degeneracy_surfaces}
    The regularity of the parametrisations \cref{eq:general_expression_CMC_surface,eq:maximal_surface_parametrization} is as follows:
    \begin{enumerate}[label=\normalfont(\roman*)]
        \item \(S_{(c,d)}(t,u)=B_u\gamma_{(c,d)}(t)\)  is regular if and only if
              \(c^2-d^2\neq 1\) or \(c\geq 0\);
        \item \(S_{(\varepsilon,\alpha,\beta)}(t,u)=B_u\gamma_{(\varepsilon,\alpha,\beta)}(t)\)
              is regular if and only if
              \(\alpha\neq0\). Moreover, when \(\alpha=0\), the the image of the parametrisation is the set $\{(x,y,0):\abs{y}>\abs{x}\}\cup\{(0,0,0)\}$.
    \end{enumerate}
\end{proposition}
\begin{proof}
    Since boosts are diffeomorphisms, it is enough to check regularity at
    \(u=0\). For the first family, \(\partial_tS_{(c,d)}(t,0)=\gamma'_{(c,d)}(t)\)
    is horizontal and nonzero. Also,
    \(\partial_uS_{(c,d)}(t,0)=(\sinh t+d,\cosh t+c,0)\). With respect to the
    basis \(\{X,Y,Z\}\), using \(Z^*=dz+\frac y2dx-\frac x2dy\), these two
    vectors have coordinate matrix
    \[
        \begin{pmatrix}
            \sinh t & \sinh t+d                                            \\
            \cosh t & \cosh t+c                                            \\
            0       & \dfrac12\left(-1+d^2-c^2+2(d\sinh t-c\cosh t)\right)
        \end{pmatrix}.
    \]
    Hence the parametrisation is singular at \((t,0)\) if and only if
    \begin{equation}\label{eq:singular_rank_drop_condition}
        c\sinh t-d\cosh t=0,
        \qquad
        -1+d^2-c^2+2(d\sinh t-c\cosh t)=0.
    \end{equation}
    If \(c^2-d^2\neq0\), this system is
    equivalent to
    \[
        \begin{pmatrix}
            \cosh t \\
            \sinh t
        \end{pmatrix}
        =
        \frac{-1-c^2+d^2}{2(c^2-d^2)}
        \begin{pmatrix}
            c \\
            d
        \end{pmatrix}.
    \]
    The identity \(\cosh^2t-\sinh^2t=1\) implies \(c^2-d^2=1\), and then the condition
    \(\cosh t>0\) gives \(c<0\). Conversely, if \(c^2-d^2=1\) and \(c<0\), the
    displayed formula gives \(\cosh t=-c>0\) and \(\sinh t=-d\), so the rank
    drops. Thus \(S_{(c,d)}\) is regular exactly when \(c^2-d^2\neq1\) or
    \(c\geq0\).

    For the maximal family, again at \(u=0\),
    \(\partial_tS_{(\varepsilon,\alpha,\beta)}
    =(\sinh\beta,\varepsilon\cosh\beta,\frac{\varepsilon\alpha}{2}\cosh\beta)\)
    and
    \(\partial_uS_{(\varepsilon,\alpha,\beta)}
    =(\varepsilon t\cosh\beta,t\sinh\beta+\alpha,0)\). If \(\alpha\neq0\),
    these vectors are clearly linearly independent, while if \(\alpha=0\) the vector
    \(\partial_uS_{(\varepsilon,\alpha,\beta)}\) vanishes at \(t=0\). Therefore
    the parametrisation is regular if and only if \(\alpha\neq0\). When \(\alpha=0\), the parametrisation reduces to
    \(S(t,u)=(\varepsilon t\sinh(u+\varepsilon\beta),
    \varepsilon t\cosh(u+\varepsilon\beta),0)\). Thus \(x^2-y^2=-t^2\), so for
    \(t\neq0\) the image is the region \(\abs{y}>\abs{x}\) in the plane \(z=0\),
    while \(t=0\) gives only the origin.
\end{proof}

\subsection{Surfaces of spacelike parameters}
In this section, we study the surfaces $S_{(c,d)}$ associated with spacelike vectors $(c,d)$ in the Minkowski plane, that is, vectors satisfying $c^2-d^2<0$. By \cref{prop:non_degeneracy_surfaces}, these surfaces are already known to be regular. We illustrate such a surface in \cref{fig:spacelikesurface}. Although no regularity issue arises in the classical sense, these surfaces still exhibit a form of \say{non-regularity} from the sub-Lorentzian point of view: their set of characteristic points is non-empty. This is described in the following proposition.

\begin{figure}[ht]
    \centering
    \includegraphics[scale = 0.6]{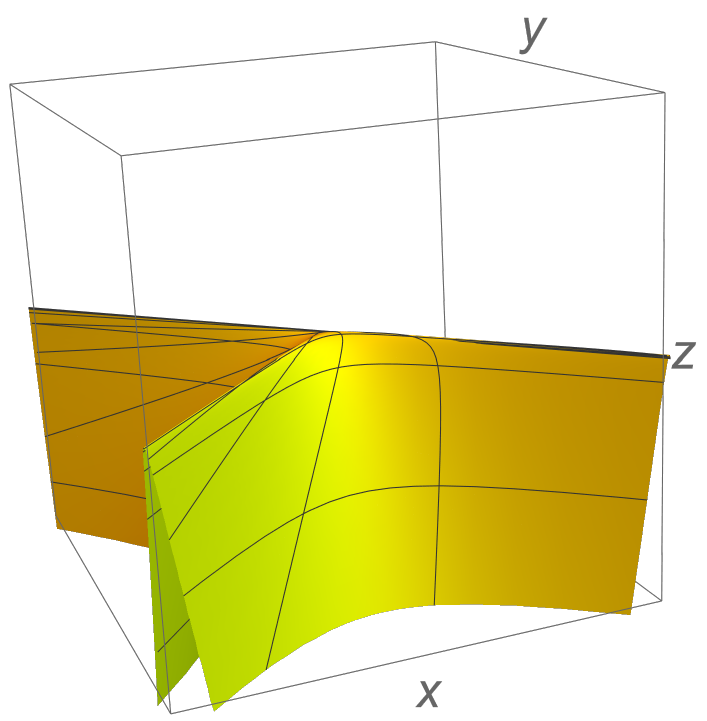}
    \caption{A surface \(S_{(c,d)}\) with spacelike parameter.}
    \label{fig:spacelikesurface}
\end{figure}

\begin{proposition}\label{prop:spacelike_characteristic_set}
    Let \(S_{(c,d)}\) be a \((c,d)\)-surface with \(c^2-d^2<0\), and set
    \[
        t^*=-\sign(d)\log(\abs d-c),
        \qquad
        \omega=\frac12\left(t^*+c\sinh(t^*)-d\cosh(t^*)+d\right).
    \]
    Then the characteristic set of \(S_{(c,d)}\) is the open half-line
    \[
        \{(s,\sign(d)s,\omega):s>0\}.
    \]
    In particular, it is a horizontal null curve, and therefore the surface is
    not acausal.
\end{proposition}
\begin{proof}
    By \cref{prop:non_degeneracy_surfaces}, the vectors
    \(\partial_tS_{(c,d)}\) and \(\partial_uS_{(c,d)}\) span the tangent plane
    everywhere. Since \(\partial_tS_{(c,d)}\) is horizontal by construction, a
    point is characteristic precisely when \(\partial_uS_{(c,d)}\) is also
    horizontal. Boosts preserve the horizontal distribution, so it is enough to
    work at \(u=0\).

    At \(u=0\), the boost direction is
    \(\partial_uS_{(c,d)}(t,0)=(\sinh t+d,\cosh t+c,0)\). Therefore
    \[
        Z^*(\partial_uS_{(c,d)}(t,0))
        =
        \frac12\left((\sinh t+d)^2-(\cosh t+c)^2\right),
    \]
    where, as before, \(Z^*\) is the dual one-form to \(Z\) with respect to the frame \(\{X,Y,Z\}\).
    Thus the characteristic condition is
    \begin{equation}\label{eq:general_characteristic_equation}
        \bigl(\sinh t+d\bigr)^2-\bigl(\cosh t+c\bigr)^2=0
        \quad\Longleftrightarrow\quad
        1+c^2-d^2+2(c\cosh t-d\sinh t)=0.
    \end{equation}

    Since \(c^2-d^2<0\), we have \(\abs d>\abs c\), and hence \(d\neq0\), so that the only possible sign in \cref{eq:general_characteristic_equation} is
    \begin{equation}
        \label{eq:definingeq}
        \sinh t+d=\sign(d)(\cosh t+c).
    \end{equation}
    If \(d>0\) this is the equation
    \(e^{-t}=d-c\), while if \(d<0\) it is the equation
    \(e^t=-d-c\). In both cases the unique solution is
    \[
        t^*=-\sign(d)\log(\abs d-c).
    \]

    At this value of \(t\), set \(\eta=\cosh(t^*)+c\). Equation \cref{eq:definingeq}
    gives \(\sinh(t^*)+d=\sign(d)\eta\). Moreover \(\eta>0\): if \(d>0\), then
    \(\eta=(1+d^2-c^2)/(2(d-c))\), while if \(d<0\), then
    \(\eta=(1+d^2-c^2)/(2(-d-c))\). Therefore, denoting
    the third component of \(\gamma_{(c,d)}(t^*)\) by \(\omega\), the
    characteristic points are
    \[
        B_u(\eta,\sign(d)\eta,\omega)
        =
        \bigl(\eta e^{\sign(d)u},\sign(d)\eta e^{\sign(d)u},\omega\bigr).
    \]
    Since \(s=\eta e^{\sign(d)u}\) ranges over \((0,+\infty)\), the
    characteristic set is exactly \(\{(s,\sign(d)s,\omega):s>0\}\).

    Finally, the curve \(s\mapsto(s,\sign(d)s,\omega)\) is horizontal and its horizontal velocity
    \((1,\sign(d))\) is null in the Minkowski plane. Hence the characteristic set
    is a horizontal null curve contained in the surface, so the surface is not
    acausal.
\end{proof}

\begin{corollary}\label{coro:no_time_inverting_isometry_spacelike_surfaces}
    No two surfaces of spacelike parameters are related by a time-inverting
    isometry.
\end{corollary}
\begin{proof}
    Suppose that a time-inverting isometry sends one such surface to another.
    Since isometries preserve the horizontal distribution, they send
    characteristic points to characteristic points. By
    \cref{prop:spacelike_characteristic_set}, the characteristic set of the
    source is \(\{(s,\sign(d)s,\omega):s>0\}\), and the characteristic set of the
    target is contained in \(\{x>0\}\).

    Let \(\Lambda\) be the Lorentz part of the isometry. Since \(\Lambda\) is time-inverting, any future-directed causal vector gets sent to a past-directed causal vector and in particular to a vector of negative \(x\)-coordinate.

    Applying the isometry to a point \((s,\sign(d)s,\omega)\) of the source
    characteristic set, its \(x\)-coordinate is therefore
    \(a+s(\Lambda_{11}+\sign(d)\Lambda_{12})\), where \(a\in\mathbb R\) comes
    from the translation part. This is negative for large \(s\), contradicting
    the fact that the target characteristic set lies in \(\{x>0\}\).
\end{proof}

We next classify these surfaces up to isometries. We first give a sufficient
condition for two spacelike-parameter surfaces to be isometric by a
time-preserving boost.

\begin{proposition}\label{prop:spacelike_isometric_surfaces}
    Let \(S_{(c,d)}\) and \(S_{(c',d')}\) be two surfaces of spacelike
    parameters. If \(c^2-d^2=c'^2-d'^2\) and \(dd'>0\), then they are isometric
    by a time-preserving boost.
\end{proposition}
\begin{proof}
    It is enough to show that \(S_{(c,d)}\) is isometric to \(S_{(0,D)}\), by a
    time-preserving boost, where
    \(D=\sign(d)\sqrt{d^2-c^2}\). Indeed, the hypotheses give
    \(d^2-c^2=(d')^2-(c')^2\) and \(\sign(d)=\sign(d')\), so the same value of \(D\)
    is obtained from both \((c,d)\) and \((c',d')\). The desired isometry is then
    obtained by composing one reduction with the inverse of the other.

    Choose \(\alpha\in\mathbb R\) such that \(\sinh\alpha=c/D\). Since \(D\) and
    \(d\) have the same sign and \(D^2=d^2-c^2\), we also have
    \(\cosh\alpha=d/D\). Starting from \(S_{(0,D)}(t,u)=B_u\gamma_{(0,D)}(t)\),
    make the change of parameters \(s=t+\alpha\), \(v=u-\alpha\). Then
    \(t=s-\alpha\), \(u=v+\alpha\), and the addition formulae give
    \[
        S_{(0,D)}(s-\alpha,v+\alpha)
        =
        B_{v+\alpha}\gamma_{(0,D)}(s-\alpha)
        =
        B_v\bigl(B_\alpha\gamma_{(0,D)}(s-\alpha)\bigr)
    \]
    The first two components agree with those of \(S_{(c,d)}(s,v)\), while the
    third component differs from that of \(S_{(c,d)}(s,v)\) by the constant
    \(\frac12(\alpha+d-D)\), that is to say
    \[
        S_{(c,d)}(s,v)
        = \left(0,0,\frac12(\alpha+d-D)\right) \ast
        S_{(0,D)}(s-\alpha,v+\alpha)
        ,
    \]
    proving the claim.
\end{proof}

We now show that the conditions of the previous proposition are also necessary
for two such surfaces to be isometric by a time-preserving boost.

\begin{proposition}\label{coro:spacelike_non_isometric_surfaces}
    If two surfaces of spacelike parameters \(S_{(c,d)}\) and \(S_{(c',d')}\)
    are isometric by a time-preserving boost, then
    \(c^2-d^2=c'^2-d'^2\) and \(dd'>0\).
\end{proposition}
\begin{proof}
    By the proof of \cref{prop:spacelike_isometric_surfaces}, after composing
    with time-preserving boosts, any
    spacelike-parameter surface can be put in the form \(S_{(0,d)}\). Since
    \(S_{(0,d)}\) is boost-invariant, it is enough to prove that two surfaces \(S_{(0,d)}\) and \(S_{(0,d')}\) cannot be related by a vertical
    translation unless \(d=d'\). Because \((0,d)\) and \((0,d')\) are
    spacelike, \(d\) and \(d'\) must be nonzero.

    A vertical translation must send the characteristic set of \(S_{(0,d)}\) to
    that of \(S_{(0,d')}\) since it preserves the horizontal distribution. By \cref{prop:spacelike_characteristic_set}, these
    characteristic sets are \(\{(s,\sign(d)s,\omega):s>0\}\) and
    \(\{(s,\sign(d')s,\omega'):s>0\}\). A vertical translation does not
    change \(x\) or \(y\), which already forces \(\sign(d)=\sign(d')\).

    Write \(S_{(0,d)}(t,u)=(x(t,u),y(t,u),z(t))\), where
    \(z(t)=\frac12(t-d\cosh t+d)\). A critical point of \(z\) satisfies
    \[
        1-d\sinh t=0,
        \qquad\text{hence}\qquad
        \sinh t=\frac1d.
    \]
    Since \(x^2-y^2\) is invariant under boosts, we may set \(u=0\). Thus
    \[
        x(t,u)^2-y(t,u)^2
        =
        x(t,0)^2-y(t,0)^2
        =
        \cosh^2t-(\sinh t+d)^2.
    \]
    Similarly, if \(S_{(0,d')}(t,u)=(x'(t,u),y'(t,u),z'(t))\), then at the
    critical points of \(z'\) one has
    \(x'(t,u)^2-y'(t,u)^2=-d'^2-1\). Since a vertical translation preserves the
    \(x,y\)-coordinates and sends critical points of \(z\) to critical points of
    \(z'\), we must have
    \[
        -d^2-1=-d'^2-1.
    \]
    Hence \(d^2=d'^2\), and together with \(\sign(d)=\sign(d')\), this gives
    \(d=d'\). Returning to the original parameters, this is exactly
    \(c^2-d^2=c'^2-d'^2\) and \(dd'>0\).
\end{proof}

We now investigate when two surfaces with spacelike parameters are isometric via a time-preserving
rotation.
\begin{proposition}\label{prop:spacelike_rotation_isometric_surfaces}
    Two surfaces of spacelike parameters \(S_{(c,d)}\) and \(S_{(c',d')}\) are
    isometric by a time-preserving rotation if and only if
    \(c^2-d^2=c'^2-d'^2\) and \(dd'<0\).
\end{proposition}
\begin{proof}
    Assume first that \(c^2-d^2=c'^2-d'^2\) and \(dd'<0\). Then \(d\) and
    \(-d'\) have the same sign, and
    \(c^2-d^2=c'^2-(-d')^2\). By
    \cref{prop:spacelike_isometric_surfaces}, \(S_{(c,d)}\) is isometric to
    \(S_{(c',-d')}\) by a time-preserving boost. By
    \cref{prop:timed_rotation_behaviour}, \(S_{(c',-d')}\) is isometric to
    \(S_{(c',d')}\) by a timed rotation. The composition has time-preserving
    rotation Lorentz part, proving one implication.

    Conversely, suppose that \(S_{(c,d)}\) is isometric to \(S_{(c',d')}\) by a
    time-preserving rotation. Composing with the timed rotation
    \((x,y,z)\mapsto(x,-y,-z)\), which sends \(S_{(c',d')}\) to
    \(S_{(c',-d')}\), gives an isometry from \(S_{(c,d)}\) to \(S_{(c',-d')}\)
    whose Lorentz part is a time-preserving boost. Hence
    \cref{coro:spacelike_non_isometric_surfaces} gives
    \[
        c^2-d^2=c'^2-(-d')^2
        \qquad\text{and}\qquad
        d(-d')>0,
    \]
    which is equivalent to \(c^2-d^2=c'^2-d'^2\) and \(dd'<0\).
\end{proof}

Having completed the isometry classification for spacelike parameters, we turn our attention to their causal behaviour.

\begin{proposition}
    Every surface of spacelike parameter is not achronal.
\end{proposition}
\begin{proof}
    By the classification results above, it suffices to prove non-achronality for \(S_{(0,d)}\) with \(d\neq0\).
    We construct two points on the
    surface which are chronologically related. Setting \(u=-t\) in
    \(S_{(0,d)}(t,u)\) gives the curve
    \[
        \left(
        1-d\sinh(t),\,
        d\cosh(t),\,
        \dfrac{1}{2} \left( t - d\cosh(t) +d \right)
        \right).
    \]
    Evaluating this curve at \(t=-\sign(d)s\), we define
    \[
        \gamma(s)\coloneqq
        \left(
        1+\abs{d}\sinh(s),\,
        d\cosh(s),\,
        \dfrac{1}{2} \left( -\sign(d)s - d\cosh(s) +d \right)
        \right).
    \]
    By left-invariance, \(\gamma(-s)\ll\gamma(s)\) is equivalent to
    \(0\ll\gamma(-s)^{-1}\ast\gamma(s)\). A direct computation gives
    \begin{align*}
        \gamma(-s)^{-1}\ast\gamma(s)
         & =
        \left(
        2\abs{d}\sinh(s),\,
        0,\,
        \frac{1}{2} (-2\sign(d)s +2d\abs{d}\sinh(s)\cosh(s))
        \right) \\
         & =
        \left(
        2\abs{d}\sinh(s),\,
        0,\,
        -\frac{\sign(d)}{2} (2s -d^2\sinh(2s))
        \right).
    \end{align*}
    For \(s>0\), the \(x\)-component is positive. The condition
    \(0\ll\gamma(-s)^{-1}\ast\gamma(s)\) is therefore
    \[
        -4d^2\sinh^2(s)
        +
        2\abs{2s-d^2\sinh(2s)}
        <0.
    \]
    For \(s\) sufficiently large, \(2s-d^2\sinh(2s)<0\), and the left hand side
    becomes
    \begin{equation}\label{eq:spacelike_parameter_chronological_quantity}
        -4d^2\sinh^2(s)+2d^2\sinh(2s)-4s
        =
        -2d^2\bigl(2\sinh^2(s)-\sinh(2s)\bigr)-4s.
    \end{equation}
    The term in parentheses is bounded as \(s\to+\infty\), whereas
    \(-4s\to-\infty\). Hence \cref{eq:spacelike_parameter_chronological_quantity} is negative for all
    sufficiently large \(s\). Thus \(\gamma(-s)\ll\gamma(s)\) for large \(s\),
    so the surface contains chronologically related points and thus is not achronal.
\end{proof}

\subsection{Surfaces of timelike parameters}\label{sec:timelike_parameters}

In this section, we consider the family of surfaces \(S_{(c,d)}\) with timelike parameters, that is, the surfaces defined in \cref{eq:general_expression_CMC_surface} for which \(c^2-d^2>0\). An illustration of these surfaces can be found in \cref{fig:timelikesurface}.

\begin{proposition}\label{prop:timelike_isometric_surfaces}
    Let \(S_{(c,d)}\) and \(S_{(c',d')}\) be two surfaces of timelike
    parameters. If \(c^2-d^2=c'^2-d'^2\) and \(cc'>0\), then they are isometric
    by a time-preserving boost.
\end{proposition}
\begin{proof}
    This follows closely the proof given in \cref{prop:spacelike_isometric_surfaces}.
    It is enough to show that \(S_{(c,d)}\) is isometric to \(S_{(C,0)}\), by a
    time-preserving boost, where
    \(C=\sign(c)\sqrt{c^2-d^2}\). Indeed, the hypotheses give
    \(c^2-d^2=(c')^2-(d')^2\) and \(\sign(c)=\sign(c')\), so the same value of
    \(C\) is obtained from both \((c,d)\) and \((c',d')\). The desired isometry is
    then obtained by composing one reduction with the inverse of the other.

    Choose \(\alpha\in\mathbb R\) such that \(\sinh\alpha=d/C\). Since \(C\) and
    \(c\) have the same sign and \(C^2=c^2-d^2\), we also have
    \(\cosh\alpha=c/C\). Starting from \(S_{(C,0)}(t,u)=B_u\gamma_{(C,0)}(t)\),
    make the change of parameters \(s=t+\alpha\), \(v=u-\alpha\). Then
    \(t=s-\alpha\), \(u=v+\alpha\), and the addition formulae give
    \[
        S_{(C,0)}(s-\alpha,v+\alpha)
        =
        B_{v+\alpha}\gamma_{(C,0)}(s-\alpha)
        =
        B_v\bigl(B_\alpha\gamma_{(C,0)}(s-\alpha)\bigr).
    \]
    The first two components agree with those of \(S_{(c,d)}(s,v)\), while the
    third component of \(S_{(c,d)}(s,v)\) is obtained by adding the constant
    \(\frac12(\alpha+d)\). Thus
    \[
        S_{(c,d)}(s,v)
        =
        \left(0,0,\frac12(\alpha+d)\right)\ast
        S_{(C,0)}(s-\alpha,v+\alpha),
    \]
    proving the claim.
\end{proof}

The next result is parallel to \cref{prop:spacelike_rotation_isometric_surfaces}.

\begin{proposition}\label{prop:timelike_positive_limit_null_surface}
    Two surfaces of timelike parameters are isometric via a time-preserving rotation if and only if they are isometric via a time-preserving boost.
\end{proposition}
\begin{proof}
    By \cref{prop:timed_rotation_behaviour}, recall that the timed rotation
    \(A(x,y,z)=(x,-y,-z)\) sends a \((c,d)\)-surface to the
    \((c,-d)\)-surface.

    Suppose first that a \((c,d)\)-surface
    is isometric to a \((c',d')\)-surface by a time-preserving
    rotation. Composing this isometry with \(A\) gives a time-preserving boost
    from the \((c,d)\)-surface to the \((c',-d')\)-surface. Moreover,
    \cref{prop:timelike_isometric_surfaces} gives a time-preserving boost
    between the \((c',-d')\)-surface and the \((c',d')\)-surface, since changing \(d'\) to
    \(-d'\) preserves \((c')^2-(d')^2\) and the sign of \(c'\). The \((c,d)\)-surface is thus isometric to the \((c',d')\)-surface by a
    time-preserving boost.

    Conversely, suppose that the two surfaces are isometric by a
    time-preserving boost. Composing with \(A\) gives a time-preserving
    rotation from the \((c,d)\)-surface to the \((c',-d')\)-surface. Again by
    \cref{prop:timelike_isometric_surfaces}, the \((c',-d')\)-surface and the
    \((c',d')\)-surface are isometric by a time-preserving boost. Composing
    these two isometries gives a time-preserving rotation from the
    \((c,d)\)-surface to the \((c',d')\)-surface.
\end{proof}
We next describe the characteristic points of these surfaces; this will allow us
to prove that the sufficient condition in the previous proposition is in fact
necessary.
\begin{proposition}\label{prop:timelike_surfaces_characteristic_set}
    The
    characteristic set of a \((c,d)\)-surface with timelike parameters is non-empty if and only if \(c<0\) and
    \(c^2-d^2\neq1\). In this case, for \(i=0,1\), set
    \[
        r=\sqrt{c^2-d^2},
        \qquad
        \alpha=\arcsinh(d/r),
        \qquad
        t_i=-\alpha+(-1)^i\log r,
        \qquad
        \omega_i=\sinh(t_i)+d.
    \]
    Then the characteristic set is the disconnected union of the two half-lines
    \[
        \left\{
        \left(
        \sign(1-c^2+d^2)s,\,
        \sign(\omega_i)s,\,
        \frac{1}{2}\left(t_i+c\sinh(t_i)-d\cosh(t_i)+d\right)
        \right):s>0
        \right\},
        \quad i=0,1.
    \]
\end{proposition}
\begin{proof}
    As in the proof of \cref{prop:spacelike_characteristic_set}, it is enough to
    work at \(u=0\). At points where the tangent plane is defined,
    \(\partial_tS_{(c,d)}\) is horizontal by construction. Hence the
    characteristic points are exactly those for which \(\partial_uS_{(c,d)}\) is
    horizontal, namely those satisfying \cref{eq:general_characteristic_equation}.

    Put \(r=\sqrt{c^2-d^2}\). If \(c>0\), choose \(\alpha\) such that
    \(\sinh\alpha=d/r\) and \(\cosh\alpha=c/r\). The previous equation becomes
    \(1+r^2+2r\cosh(t-\alpha)=0\), which has no solution. Hence no
    characteristic points occur when \(c>0\).

    Suppose now that \(c<0\). Choose \(\alpha\) such that
    \(\sinh\alpha=d/r\) and \(\cosh\alpha=-c/r\). The characteristic equation is
    then
    \[
        \cosh(t+\alpha)=\frac{1+r^2}{2r}
        =\cosh(\log r).
    \]
    The equation has two distinct solutions if and only if \(r\neq1\), namely
    \(t_i=-\alpha+(-1)^i\log r\), for \(i=0,1\). If \(r=1\), the only
    solution is \(t=-\alpha\), which is precisely the singular value by
    \cref{prop:non_degeneracy_surfaces}. Therefore it does not contribute to the
    characteristic set. For either solution, a direct computation gives
    \[
        \cosh(t_i)+c=\sign(1-c^2+d^2)\abs{\sinh(t_i)+d}.
    \]
    Denote \(\omega_i=\sinh(t_i)+d\). Since \(r\neq1\), we have
    \(\omega_i\neq0\). Substituting \(t=t_i\) in \(S_{(c,d)}(t,u)\) gives
    \[
        \begin{pmatrix}
            \sign(1-c^2+d^2)\abs{\omega_i}\cosh u + \omega_i \sinh u  \\
            \omega_i \cosh u + \sign(1-c^2+d^2)\abs{\omega_i} \sinh u \\
            \frac{1}{2} \left(t_i +c\sinh(t_i) -d\cosh(t_i) +d\right)
        \end{pmatrix}
        =
        \begin{pmatrix}
            \sign(1-c^2+d^2) \abs{\omega_i} e^{\sign(1-c^2+d^2)\sign(\omega_i)u} \\
            \sign(\omega_i) \abs{\omega_i} e^{\sign(1-c^2+d^2)\sign(\omega_i)u}  \\
            \frac{1}{2} \left(t_i +c\sinh(t_i) -d\cosh(t_i) +d\right)
        \end{pmatrix}
    \]
    Setting \(s=\abs{\omega_i}e^{\sign(1-c^2+d^2)\sign(\omega_i)u}\), the previous
    expression becomes
    \[
        \begin{pmatrix}
            \sign(1-c^2+d^2)s \\
            \sign(\omega_i) s \\
            \frac{1}{2} \left(t_i +c\sinh(t_i) -d\cosh(t_i) +d\right)
        \end{pmatrix}
    \]
    with \(s\in(0,+\infty)\). Hence the characteristic set consists of two null
    half-lines starting on the \(z\)-axis.
    To see that
    they are disjoint, let
    \[
        z_i=\frac{1}{2}\left(t_i+c\sinh(t_i)-d\cosh(t_i)+d\right), \quad \text{ so that } \quad
        z_0-z_1
        =
        \log r-\frac{r^2-1}{2}.
    \]
    The function \(r\mapsto \log r-\frac{r^2-1}{2}\) vanishes only at \(r=1\);
    indeed its derivative is \((1-r^2)/r\), so it has a strict maximum \(0\) at
    \(r=1\). Since \(r\neq1\), we have \(z_0\neq z_1\). Therefore the two
    half-lines lie in distinct horizontal \(z\)-planes, and are consequently
    disjoint. Their union is thus disconnected.
\end{proof}

By \cref{prop:non_degeneracy_surfaces}, the nonzero-curvature
parametrisations fail to be regular precisely when
\[
    c^2-d^2=1,
    \qquad c<0.
\]
These parameters are timelike, since \(c^2-d^2>0\). By
\cref{prop:non_degeneracy_surfaces}, the regular timelike cases are precisely
the complementary cases. Here we study the exceptional case
\(c^2-d^2=1\), \(c<0\), where the parametrisation loses rank. By \cref{prop:timelike_isometric_surfaces}, all such singular
timelike-parameter surfaces are isometric to one another by time-preserving
boosts, and it is therefore enough to study $S_{(-1, 0)}$. In this case the parametrisation is
\begin{equation}\label{eq:singular_surface_expression}
    S_{(-1,0)}(t,u)
    =
    B_u\gamma_{(-1,0)}(t)
    =
    \left(
    \cosh(t+u)-\cosh u,\;
    \sinh(t+u)-\sinh u,\;
    \frac12(t-\sinh t)
    \right).
\end{equation}
This expression has a useful interpretation: it is the expression of the
Heisenberg group's sub-Lorentzian exponential map at the origin, as written in
\cref{eq:exponentialmap}, evaluated at
\((t\sinh u,t\cosh u,t)\). Equivalently, it is the curve
\(t\mapsto\exp_e(t\sinh u,t\cosh u,t)\). This initial datum does not lie in the
standard causal domain of the exponential map, because the horizontal vector
\((\sinh u,\cosh u)\) is spacelike rather than causal.

To locate the singular set, we return to the proof of
\cref{prop:non_degeneracy_surfaces}. For \(S_{(-1,0)}\), the parametrisation is
singular at parameter values \((t,u)\) where, after reducing to \(u=0\),
\cref{eq:singular_rank_drop_condition} holds. With \(c=-1\) and \(d=0\), this
becomes
\[
    \sinh t=0,
    \qquad
    \cosh t=1,
\]
which is equivalent to \(t=0\). Substituting \(t=0\) into
\cref{eq:singular_surface_expression} gives only the origin. Therefore the
singular set in the parameter domain is $\{(0,u):u\in\mathbb R\}$.
All these parameter values have the same image
\[
    S_{(-1,0)}(0,u)
    =
    B_u\gamma_{(-1,0)}(0)
    =
    B_u(0,0,0)
    =
    (0,0,0).
\]
At this stage,
however, this only shows that the chosen parametrisation is regular everywhere except at the
origin. It does not yet rule out the possibility that the image admits another
smooth parametrisation there. We next show that the origin is a genuine
singular point of the image. We first record the projection of the surface and
a useful symmetry.

\begin{lemma}\label{lem:singular_surface_projection}
    The projection of \(S_{(-1,0)}\) to the \((x,y)\)-plane is
    \[
        \{y^2-x^2>0\}\cup\{(0,0)\}.
    \]
    Moreover, the projection to the \((x,y)\)-plane is
    injective on \(S_{(-1,0)}\).
\end{lemma}
\begin{proof}
    Set \(v=u+t/2\). Then \cref{eq:singular_surface_expression} becomes
    \begin{equation}\label{eq:singular_surface_tv_coordinates}
        S_{(-1,0)}(t,v)
        =
        \left(
        2\sinh v\sinh(t/2),\;
        2\cosh v\sinh(t/2),\;
        \frac12(t-\sinh t)
        \right).
    \end{equation}
    Hence the projected coordinates satisfy
    \[
        y^2-x^2=4\sinh^2(t/2)\geq0.
    \]
    Equality holds if and only if \(t=0\), and then
    \(S_{(-1,0)}(0,u)=(0,0,0)\). Thus the only projected point with
    \(y^2-x^2=0\) is \((0,0)\).

    Conversely, let \((x,y)\) satisfy \(y^2-x^2>0\). Then \(y\neq0\) and
    \((t,v)\) is uniquely determined by
    \[
        2\sinh(t/2)=\sign(y)\sqrt{y^2-x^2}, \qquad \tanh v=\frac{x}{y},
    \]
    proving the claim.
\end{proof}

\begin{proposition}\label{prop:singular_surface_symmetry}
    The singular surface \(S_{(-1,0)}\) is invariant under the reflection
    \((x,y,z)\mapsto(-x,y,z)\), but it is not invariant under the map
    \((x,y,z)\mapsto(-x,y,-z)\).
\end{proposition}
\begin{proof}
    Replacing
    \((t,u)\) by \((t,-t-u)\) in \cref{eq:singular_surface_expression} shows that
    \begin{align*}
        S_{(-1,0)}(t,-t-u)
         & =
        \left(
        \cosh u-\cosh(t+u),\;
        \sinh(t+u)-\sinh u,\;
        \frac12(t-\sinh t)
        \right)                           \\
         & =R\bigl(S_{(-1,0)}(t,u)\bigr),
    \end{align*}
    where \(R(x,y,z)=(-x,y,z)\).

    Choose a point \((x,y,z)\) on the surface with \(z\neq0\). By the first
    part, \((-x,y,z)\) is also on the surface. If the surface were invariant
    under \((x,y,z)\mapsto(-x,y,-z)\), then \((-x,y,-z)\) would also be on the
    surface. This would give two distinct points of the surface with the same
    \((x,y)\)-projection, contradicting the injectivity proved in
    \cref{lem:singular_surface_projection}.
\end{proof}

We show that the rank
loss at the origin reflects a genuine singularity of the image, rather than
being caused only by the choice of parametrisation.

\begin{proposition}\label{prop:real_singularity}
    The only point of the image \(S_{(-1,0)}\) projected to \((0,0)\) is the
    origin. In particular, the image of \(S_{(-1,0)}\) is singular at the
    origin.
\end{proposition}
\begin{proof}
    Consider the surface \(S_{(-1,0)} \setminus \{(0,0,0)\}\). By \Cref{lem:singular_surface_projection} we know that its projection coincides with the set \(y^2 -x^2>0\). In particular, this set is disconnected; as the projection is a continuous map we get that the surface \(S_{(-1,0)} \setminus \{(0,0,0)\}\) is disconnected as well. Since \(S_{(-1,0)}\) is connected we can conclude that it cannot be locally Euclidean with dimension 2 at the origin, as a two dimensional locally Euclidean surface cannot become disconnected after removing one point.
\end{proof}

We start by showing that there are no time-inverting isometries between singular timelike-parameters surfaces.

\begin{corollary}\label{coro:no_time_inverting_singular_timelike}
    No two singular timelike-parameter surfaces are related by a
    time-inverting isometry.
\end{corollary}
\begin{proof}
    By \cref{prop:timelike_isometric_surfaces}, all singular
    timelike-parameter surfaces are isometric to \(S_{(-1,0)}\) by
    time-preserving boosts. It is therefore enough to rule out time-inverting
    self-isometries of \(S_{(-1,0)}\).

    By \cref{prop:real_singularity}, the
    origin is the unique singular point of \(S_{(-1,0)}\), so any time-inverting self-isometry fixes the
    origin and has no translation part. Up to composing with
    time-preserving symmetries of \(S_{(-1,0)}\), namely boosts and the timed
    rotation from \cref{prop:timed_rotation_behaviour}, the Lorentz part of this isometry can be reduced to the map \(A(x,y,z)=(-x,y,-z)\). This would
    force \(A(S_{(-1,0)})=S_{(-1,0)}\), contradicting
    \cref{prop:singular_surface_symmetry}.
\end{proof}

We also note that this singular surface has no characteristic points.

\begin{proposition}\label{prop:singular_surface_no_characteristic_points}
    The singular surface \(S_{(-1,0)}\) has no characteristic points.
\end{proposition}
\begin{proof}
    As in the proofs of
    \cref{prop:spacelike_characteristic_set,prop:timelike_surfaces_characteristic_set},
    it is enough to work at \(u=0\). At regular points,
    \(\partial_tS_{(-1,0)}\) is horizontal by construction, so a point is
    characteristic precisely when \(\partial_uS_{(-1,0)}\) is horizontal.
    Setting \(c=-1\) and \(d=0\) in
    \cref{eq:general_characteristic_equation}, this condition becomes
    \[
        2-2\cosh t=0,
    \]
    and thus \(t=0\), which is not a regular
    point of the surface by \cref{prop:real_singularity}.
\end{proof}

As a consequence, the sign of \(c\) is an isometry invariant for timelike-parameter surfaces.

\begin{corollary}
    Two surfaces with timelike parameters \((c,d)\) and \((c',d')\) are not
    isometric if \(c\) and \(c'\) have different signs.
\end{corollary}
\begin{proof}
    Suppose that \(c<0\) and \(c'>0\). If \(c^2-d^2=1\), then by
    \Cref{prop:non_degeneracy_surfaces} the parametrisation has a singularity.
    In \Cref{prop:real_singularity} we have shown that this surface is genuinely
    singular, as it is not locally Euclidean at that point; on the other hand,
    the \((c',d')\)-surface is smooth by \Cref{prop:non_degeneracy_surfaces}. Thus
    the two surfaces are not even homeomorphic. We may therefore assume
    \(c^2-d^2\neq1\).

    Any isometry between the two surfaces must map the characteristic set of one
    surface to the characteristic set of the other. However, by
    \Cref{prop:timelike_surfaces_characteristic_set}, the \((c,d)\)-surface has
    two half-lines as its characteristic set, whereas the \((c',d')\)-surface has
    empty characteristic set. Hence no isometry can exist.
\end{proof}

We extend \cref{coro:no_time_inverting_singular_timelike} to any surfaces of timelike parameters.

\begin{corollary}\label{coro:no_time_inverting_isometry_timelike_case}
    There is no time-inverting isometry between any pair of surfaces with timelike parameters.
\end{corollary}
\begin{proof}
    We prove the statement for non-singular surfaces, as the singular case is
    treated in \cref{coro:no_time_inverting_singular_timelike}. By
    \cref{prop:timelike_isometric_surfaces} and the previous corollary, it is
    enough to consider the \((c,0)\)- and \((c',0)\)-surfaces with \(cc'>0\).

    Assume first that \(c,c'>0\). Let
    \begin{equation*}\label{eq:projection1}
        P=\{x^2-y^2\geq(c+1)^2\}\cap\{x\geq c+1\},
        \qquad
        P'=\{x^2-y^2\geq(c'+1)^2\}\cap\{x\geq c'+1\}
    \end{equation*}
    be the projections of \(S_{(c,0)}\) and \(S_{(c',0)}\) onto the Minkowski
    plane, respectively.
    The sets \(P\) and \(P'\) are future-closed but nowhere past-closed: if \(p\)
    belongs to \(P\) (resp. \(P'\)), then \(J^+(p)\) is contained in
    \(P\) (res. \(P'\)), while \(J^-(p)\) is not.
    Since a time-inverting Lorentz map reverses the causal order, it sends
    future-closed sets to past-closed sets. Therefore \(P\) and \(P'\) cannot be
    mapped onto one another.

    It remains to consider \(c,c'<0\). It remains to consider \(c,c'<0\). Up to composing with time-preserving isometries of \(S_{(c,0)}\) and
    \(S_{(c',0)}\), a time-inverting isometry has Lorentz part $A(x, y, z) = (-x,y,-z)$,
    followed by a left translation. Indeed, any time-inverting boost is the
    composition of a time-preserving boost with \(A\), and the surfaces
    \(S_{(c,0)}\) and \(S_{(c',0)}\) are invariant under time-preserving boosts.
    Since \(A\) fixes the \(z\)-axis and the
    characteristic half-lines of both surfaces start on this axis, the
    translation must be vertical. The midpoint of the two starting points is the
    origin for both surfaces, so this vertical translation is trivial.

    Thus only \(A\) remains. The function \(f : S_{(c,0)}\to\mathbb R : (x, y, z) \mapsto x\) has a unique critical point on
    the \((c,0)\)-surface, namely \((c+1,0,0)\). Since \(f\circ A^{-1}=-x\), its
    image must be the unique critical point of \(-x\) on the \((c',0)\)-surface,
    namely \((c'+1,0,0)\). Therefore \(-c-1=c'+1\), hence \(c+c'=-2\).
    Applying the same argument to \(g(x, y, z)=z\), and using
    \(g\circ A^{-1}=-z\), shows that \(c\geq-1\) and \(c'\geq-1\),
    since \(z\)
    has critical points on \(S_{(c,0)}\) and \(S_{(c',0)}\) only in these cases. Together with
    \(c+c'=-2\), this forces \(c=c'=-1\). This corresponds exactly to the singular case, but we had assumed that the surfaces were not singular.
\end{proof}

The next lemma shows that a surface with timelike parameters and \(c>0\) eventually escapes the causal future of the origin.

\begin{lemma}\label{lemma:timelike_surfaces_not_in_future_of_origin}
    Consider a surface with timelike parameters and \(c>0\). Then, for \(z\) large
    enough, no point \((x,y,z)\) on the surface belongs to \(J^+(0)\). In
    particular, none of these surfaces is entirely contained in the future of the
    origin.
\end{lemma}
\begin{proof}
    With respect to the standard parametrisation
    \cref{eq:general_expression_CMC_surface}, the \(z\)-component depends only
    on the variable \(t\). Denote
    \(\alpha=\arcsinh(d/\sqrt{c^2-d^2})\). Since \(c>0\), we can write it as
    \[
        z(t) = \frac{1}{2} \left( t + \sqrt{c^2-d^2}\sinh(t -\alpha) +d\right).
    \]
    The function \(z(t)\) is strictly increasing, hence the parameter \(t\) is
    determined by the \(z\)-coordinate. For a point on the surface we compute the quantity \(-x^2 +y^2 +4\abs{z}\):
    \[
        -1 -c^2 + d^2 - 2\sqrt{c^2-d^2}\cosh(t - \alpha)+ 2\abs{t +\sqrt{c^2-d^2}\sinh(t-\alpha) +d}.
    \]
    The function inside the absolute value tends to \(+\infty\) as
    \(t\to+\infty\). Hence there exists \(M'>0\) such that, for \(t>M'\), the
    absolute value can be removed and the previous expression becomes
    \[
        -1 -c^2 + d^2 + 2t -2\sqrt{c^2-d^2}e^{-t + \alpha} + 2d.
    \]
    This tends to \(+\infty\) as \(t\to+\infty\). Thus there exists
    \(M''>M'\) such that, for \(t>M''\), the quantity
    \(-x^2+y^2+4\abs{z}\) is positive, independently of \(u\). By
    \cref{eq:future_origin_heisenberg}, such points do not belong to \(J^+(0)\). Since
    \(z(t)\) is strictly increasing, the condition \(t>M''\) is equivalent to
    \(z>M\) for some \(M>0\).
\end{proof}

The following proposition is the missing necessity statement for the timelike classification in \cref{prop:timelike_isometric_surfaces}.

\begin{proposition}
    If two surfaces with timelike parameters \((c,d)\) and \((c',d')\) are such
    that \(c\) and \(c'\) have the same sign but
    \(c^2-d^2\neq(c')^2-(d')^2\), then they are not isometric.
\end{proposition}
\begin{proof}
    By \Cref{prop:timelike_isometric_surfaces}, it is enough to prove the
    statement for \(S_{(c,0)}\) and \(S_{(c',0)}\), with \(cc'>0\) and
    \(c \neq c'\). By
    \Cref{coro:no_time_inverting_isometry_timelike_case}, any isometry between
    them is time-preserving. Since \(S_{(c,0)}\) and \(S_{(c',0)}\) are
    invariant under boosts and, by \Cref{prop:timed_rotation_behaviour}, under
    timed rotations, we may assume that this isometry is a left translation.

    Assume first that \(c,c'<0\). If exactly one of the two surfaces is singular,
    then no isometry exists by \Cref{prop:real_singularity}. If both are
    singular, then \(c=c'=-1\), contrary to \(c^2\neq(c')^2\). Thus we may
    assume that both are non-singular.

    By \Cref{prop:timelike_surfaces_characteristic_set}, each characteristic set
    consists of two null half-lines starting on the \(z\)-axis. A left
    translation preserving these characteristic sets must be vertical because the starting points of the characteristic half-lines lie on the \(z\)-axis, and a left translation moves the \(z\)-axis to a vertical line. Moreover,
    the \(x\)-coordinates of the characteristic half-lines have sign
    \(\sign(1-c^2)\) for \(S_{(c,0)}\), and \(\sign(1-(c')^2)\) for
    \(S_{(c',0)}\). Since vertical translations preserve \((x,y)\), we get
    \(\sign(1-c^2)=\sign(1-(c')^2)\), and therefore
    \(\sign(1+c)=\sign(1+c')\) as \(c,c'<0\).

    Writing \(S_{(c,0)}(t,u)=(x(t,u),y(t,u),z(t))\), we have
    \[
        x(t,u)^2-y(t,u)^2=1+c^2+2c\cosh t.
    \]
    Since \(c<0\), the maximum value of $x^2 - y^2$ is \((c+1)^2\), attained at \(t=0\). Similarly, on
    \(S_{(c',0)}\) the maximum is \((c'+1)^2\). A vertical translation preserves \(x^2-y^2\), so
    \((c+1)^2=(c'+1)^2\).
    Together with \(\sign(1+c)=\sign(1+c')\), this gives \(c=c'\).

    It remains to consider \(c,c'>0\). Up to exchanging the two surfaces, assume
    \(c'\geq c\). The \(x\)-coordinate has a unique minimum at
    \((c+1,0,0)\) on \(S_{(c,0)}\), and at \((c'+1,0,0)\) on
    \(S_{(c',0)}\). Hence the left translation must be by \((c'-c,0,0)\). The
    \(t=0\) curve on \(S_{(c,0)}\) is
    \[
        B_u\gamma_{(c,0)}(0)=\bigl((c+1)\cosh u,(c+1)\sinh u,0\bigr).
    \]
    Its image under the left translation by \((c'-c,0,0)\) is
    \[
        \biggl((c+1)\cosh u+c'-c,\,
        (c+1)\sinh u,\,
        \frac{1}{2}(c'-c)(c+1)\sinh u\biggr),
    \]
    and this curve lies on \(S_{(c',0)}\). Along it,
    \[
        \begin{split}
            -x^2+y^2+4\abs{z}
             & =-(c+1)^2-(c'-c)^2
            -2(c'-c)(c+1)\cosh u +2(c'-c)(c+1)\abs{\sinh u}   \\
             & =-(c+1)^2-(c'-c)^2-2(c'-c)(c+1)e^{-\abs{u}}<0.
        \end{split}
    \]
    Since also \(x>0\), \cref{eq:future_origin_heisenberg} implies that the
    translated curve is contained in \(J^+(0)\). If \(c'>c\), its
    \(z\)-coordinate tends to \(+\infty\) as \(u\to+\infty\), contradicting
    \Cref{lemma:timelike_surfaces_not_in_future_of_origin} applied to
    \(S_{(c',0)}\).

    In both cases, an isometry forces \(c=c'\), contradicting
    \(c^2\neq(c')^2\). Therefore the two surfaces are not isometric.
\end{proof}

The next lemma identifies two full null lines in the \((c,0)\)-surface with \(c<0\)
which are obtained as limits of points on the surface.

\begin{lemma}\label{lemma:timelike_surfaces_lines_closure}
    Consider the \((c,0)\)-surface with \(c<0\) and \(c\neq-1\). Let
    \(t_i=(-1)^i\log(-c)\), for \(i=0,1\), as in
    \Cref{prop:timelike_surfaces_characteristic_set}. Then the lines
    \[
        \left\{
        \left(
        s,\,
        (-1)^i s,\,
        \frac{1}{2}\left(t_i+c\sinh(t_i)\right)
        \right):s\in\mathbb R
        \right\},
        \quad i=0,1,
    \]
    are contained in the closure of \(S_{(c,0)}\).
\end{lemma}
\begin{proof}
    We use the parametrisation in \cref{eq:general_expression_CMC_surface}. Fix
    \(s\in\mathbb R\). We first approach the point
    \[
        \left(s,s,\frac{1}{2}\left(t_0+c\sinh(t_0)\right)\right).
    \]
    For \(u\to+\infty\), set
    \[
        t(u)=-u+\arccosh(s-c\cosh u),
    \]
    which is well defined for all sufficiently large \(u\), since \(c<0\). Then
    the point \(S_{(c,0)}(t(u),u)\) has \(x\)-coordinate equal to \(s\). Moreover,
    using \(\arccosh \rho=\log(\rho+\sqrt{\rho^2-1})\), one obtains
    \(t(u)\to\log(-c)=t_0\). Its \(y\)-coordinate is
    \[
        \sqrt{(s-c\cosh u)^2-1}+c\sinh u
        =
        s-c\cosh u+c\sinh u+o(1),
    \]
    and therefore tends to \(s\). The \(z\)-coordinate tends to
    \(\frac12(t_0+c\sinh t_0)\).

    Similarly, as \(u\to-\infty\), set
    \[
        t(u)=-u-\arccosh(s-c\cosh u).
    \]
    Then \(S_{(c,0)}(t(u),u)\) again has \(x\)-coordinate \(s\), while
    \(t(u)\to-\log(-c)=t_1\). Its \(y\)-coordinate is
    \[
        -\sqrt{(s-c\cosh u)^2-1}+c\sinh u
        =
        -s+c\cosh u+c\sinh u+o(1),
    \]
    and therefore tends to \(-s\). The \(z\)-coordinate tends to
    \(\frac12(t_1+c\sinh t_1)\). Since \(s\in\mathbb R\) was arbitrary, both
    lines are contained in the closure of \(S_{(c,0)}\).
\end{proof}

The previous lemma is the key ingredient in proving the non-achronality of
non-singular timelike-parameter surfaces with \(c<0\).

\begin{proposition}
    The non-singular surfaces with timelike parameters and \(c<0\) are not
    achronal.
\end{proposition}
\begin{proof}
    Since achronality is invariant under isometries, \Cref{prop:timelike_isometric_surfaces}
    allows us to prove the claim for \(S_{(c,0)}\). Since the surface is
    non-singular, \(c\neq-1\). Let
    \(p=S_{(c,0)}(0,0)=(c+1,0,0)\), and put
    \(t_i=(-1)^i\log(-c)\), as in
    \Cref{prop:timelike_surfaces_characteristic_set,lemma:timelike_surfaces_lines_closure}.

    Assume first that \(c<-1\). Then \(t_0=\log(-c)\) and
    \(c\sinh(t_0)=-(c^2-1)/2\). By
    \Cref{lemma:timelike_surfaces_lines_closure}, the line
    \[
        q(s)=\left(s,s,\frac12\left(t_0-\frac12(c^2-1)\right)\right),
        \qquad s\in\mathbb R,
    \]
    is contained in the closure of \(S_{(c,0)}\). We compute
    \[
        p^{-1}\ast q(s)=
        \left(
        s-c-1,\,
        s,\,
        \frac12\left(t_0-\frac12(c^2-1)-(c+1)s\right)
        \right),
    \]
    whose \(x\)-coordinate is positive for all large \(s\). Since \(c+1<0\), the \(z\)-coordinate of \(p^{-1}\ast q(s)\) is positive for all large \(s\).
    Thus, for large \(s\),
    \[
        -x^2+y^2+4\abs{z}
        =
        -2c^2-2c+2t_0
        =
        -2c^2-2c+2\log(-c).
    \]
    Since \(\log x\leq x-1\) for \(x>0\), this is at most
    \(-2(c+1)^2<0\). By left invariance and
    \cref{eq:future_origin_heisenberg}, \(q(s)\in I^+(p)\) for large \(s\). As
    \(I^+(p)\) is open and the line is contained in the closure of
    \(S_{(c,0)}\), the surface itself intersects \(I^+(p)\).

    Now assume that \(-1<c<0\). Then \(t_1=-\log(-c)\) and
    \(c\sinh(t_1)=(c^2-1)/2\). This time
    \Cref{lemma:timelike_surfaces_lines_closure} gives the line
    \[
        q(s)=\left(s,-s,\frac12\left(t_1+\frac12(c^2-1)\right)\right),
        \qquad s\in\mathbb R,
    \]
    in the closure of \(S_{(c,0)}\). We compute
    \[
        q(s)^{-1}\ast p=
        \left(
        c+1-s,\,
        s,\,
        -\frac12\left(t_1+\frac12(c^2-1)\right)-\frac12(c+1)s
        \right).
    \]
    Its \(x\)-coordinate is positive for all sufficiently negative \(s\). Since
    \(c+1>0\), the term inside the absolute value is also positive for all
    sufficiently negative \(s\), and therefore
    \[
        -x^2+y^2+4\abs{z}
        =
        -2c^2-2c-2t_1
        =
        -2c^2-2c+2\log(-c)
        \leq -2(c+1)^2<0.
    \]
    Hence, by left invariance and \cref{eq:future_origin_heisenberg},
    \(q(s)\in I^-(p)\) for all sufficiently negative \(s\). Again, since
    \(I^-(p)\) is open and the line lies in the closure of the surface,
    \(S_{(c,0)}\) contains a point in \(I^-(p)\).

    In either case, \(S_{(c,0)}\) contains two chronologically related points,
    so it is not achronal.
\end{proof}

In contrast with the case \(c<0\), timelike-parameter surfaces with \(c>0\) are acausal.

\begin{proposition}
    The surfaces with timelike parameters and \(c>0\) are acausal. In
    particular, they are achronal.
\end{proposition}
\begin{proof}
    By \Cref{prop:timelike_isometric_surfaces}, it is enough to prove the
    result for \(S_{(c,0)}\). Put \(r=x^2-y^2\) and define
    \[
        \Phi_c(r)=\frac12\left(
        \arccosh\left(\frac{r-c^2-1}{2c}\right)
        +\sqrt{\left(\frac{r-c^2-1}{2}\right)^2-c^2}
        \right),
    \]
    on the set \(r\geq(c+1)^2\). By
    \cref{eq:general_expression_CMC_surface}, the surface \(S_{(c,0)}\) is
    contained in the zero level set of
    \[
        f(x,y,z)=\abs{z}-\Phi_c(x^2-y^2).
    \]
    Indeed, on \(S_{(c,0)}\) we have
    \(x^2-y^2=1+c^2+2c\cosh t\) and
    \(z=\frac12(t+c\sinh t)\), while \(z\) has the same sign as \(t\).

    On the open set \(x^2-y^2>(c+1)^2\), set
    \(f_\pm(x,y,z)=\pm z-\Phi_c(x^2-y^2)\). A direct computation gives
    \[
        \begin{split}
            \nabla^\heis f_\pm
                                ={} & -\frac12\left(
            \pm y+x
            \frac{\frac{x^2-y^2-c^2+1}{2}}
            {\sqrt{\left(\frac{x^2-y^2-c^2-1}{2}\right)^2-c^2}}
            \right)X +\frac12\left(
            \pm x+y
            \frac{\frac{x^2-y^2-c^2+1}{2}}
            {\sqrt{\left(\frac{x^2-y^2-c^2-1}{2}\right)^2-c^2}}
            \right)Y,
        \end{split}
    \]
    and thus
    \[
        g(\nabla^\heis f_\pm,\nabla^\heis f_\pm)
        =
        -\frac{x^2-y^2}
        {4\left(\left(\frac{x^2-y^2-c^2-1}{2}\right)^2-c^2\right)}
        <0.
    \]
    Therefore \(\nabla^\heis f_\pm\) is timelike and has fixed time orientation on \(x>|y|\).

    We first rule out causal relations inside
    \(S_{(c,0)}\cap\{z\geq0\}\) and \(S_{(c,0)}\cap\{z\leq0\}\). Suppose that
    \(p,q\in S_{(c,0)}\cap\{z\geq0\}\) and \(p\leq q\). Let \(\gamma\) be a
    future-directed causal curve from \(p\) to \(q\). The projection of
    \(\gamma\) to the Minkowski plane is
    future-directed causal. The projected region
    \[
        P_c\coloneqq\{x^2-y^2\geq(c+1)^2\}\cap\{x>0\}
    \]
    is future-closed in the Minkowski plane. Indeed, along future directed causal curves in the region \(x > \abs{y}\) we have
    \(\frac{d}{dt}(x^2-y^2)=2(x\dot x-y\dot y)>0\) by the reverse Cauchy-Schwarz inequality in Minkowski plane. Since \(S_{(c,0)}\) projects to
    \(P_c\), the projected curve stays in the domain of \(\Phi_c\). Therefore
    \(f_+\circ\gamma\) is defined, with
    \(f_+(p)=f_+(q)=0\). Away from the boundary
    \(x^2-y^2=(c+1)^2\), we have
    \[
        \frac{d}{dt}(f_+\circ\gamma)(t)
        =
        g\bigl(\nabla^\heis f_+,\dot\gamma(t)\bigr).
    \]
    Since \(\dot\gamma\) is future-directed causal and
    \(\nabla^\heis f_+\) is timelike with fixed time orientation, this
    derivative has a fixed non-zero sign. Notice that the curve can be on the boundary set \(x^2-y^2 = (c+1)^2\) at most at its starting point. Therefore \(f_+\circ\gamma\) is strictly monotone,
    contradicting \(f_+(p)=f_+(q)=0\). The same argument applies to
    \(S_{(c,0)}\cap\{z\leq0\}\), using \(f_-\).

    It remains to exclude the case where a causal curve joins points on
    different parts of the decomposition
    \(S_{(c,0)}=(S_{(c,0)}\cap\{z\geq0\})\cup(S_{(c,0)}\cap\{z\leq0\})\). Such
    a curve meets the plane \(\{z=0\}\) at some point \(r\). By the
    future-closed property of the projected region, the projection of \(r\)
    satisfies \(x^2-y^2\geq(c+1)^2\) and \(x>0\). After applying a
    boost, we may assume \(r=(\rho,0,0)\) with \(\rho\geq c+1\). Since
    \(p_0=(c+1,0,0)\) belongs to \(S_{(c,0)}\) and the \(x\)-axis is
    future-directed timelike, we have \(p_0\leq r\). If the original causal
    curve is parametrised from \(p\) to \(q\), then \(r\leq q\), hence
    \(p_0\leq q\). The points \(p_0\) and \(q\) have \(z\)-coordinates of the
    same sign, because \(z(p_0)=0\). But this contradicts the previous argument.
\end{proof}

The next result shows that all singular surfaces with nonzero mean
curvature are acausal.

\begin{proposition}\label{prop:singular_surface_acausal}
    Every singular surface \(S_{(c,d)}\), i.e. with \(c^2-d^2=1\) and \(c<0\), is acausal.
\end{proposition}
\begin{proof}
    By \cref{prop:timelike_isometric_surfaces}, it is enough to prove the result for
    \(S_{(-1,0)}\). Let \(p_i=S_{(-1,0)}(t_i,u_i)\), \(i=1,2\), and suppose that \(p_1\leq p_2\). We will show that \(p_1=p_2\). Set
    \[
        a_i=\sinh(t_i/2),
        \qquad
        F_i=\frac12(\sinh t_i-t_i),
        \qquad i=1,2.
    \]
    If one of the \(t_i\) is zero, then the causal condition
    \cref{eq:future_origin_heisenberg}, together with
    \cref{lem:singular_surface_projection}, forces both points to be the
    origin. We may therefore assume that \(t_1,t_2\neq0\).

    Using the change of variable \(v=u+t/2\), set
    \[
        \widetilde S(t,v)
        =
        S_{(-1,0)}\left(t,v-\frac{t}{2}\right).
    \]
    In other words, \(\widetilde S\) is the parametrisation given in
    \cref{eq:singular_surface_tv_coordinates}. If \(p_i=\widetilde S(t_i,v_i)\), then applying the boost \(B_{-v_1}\) to both
    points gives
    \[
        B_{-v_1}p_1=\widetilde S(t_1,0),
        \qquad
        B_{-v_1}p_2=\widetilde S(t_2,v_2-v_1).
    \]
    Since boosts preserve the causal order, we may replace \(p_1,p_2\) by these
    boosted points. Thus we assume \(v_1=0\) and write \(v=v_2-v_1\). Then
    \[
        p_1=(0,2a_1,-F_1),
        \qquad
        p_2=(2a_2\sinh v,2a_2\cosh v,-F_2).
    \]
    A direct computation gives
    \begin{equation}\label{eq:singular_acausal_difference}
        p_1^{-1}\ast p_2
        =
        \left(
        2a_2\sinh v,\;
        2(a_2\cosh v-a_1),\;
        F_1-F_2+2a_1a_2\sinh v
        \right) =: (x, y, z).
    \end{equation}
    By \cref{eq:future_origin_heisenberg}, the inequality \(p_1\leq p_2\)
    implies
    \begin{equation}\label{eq:singular_acausal_necessary_inequality}
        a_1^2+a_2^2-2a_1a_2\cosh v\leq0.
    \end{equation}
    Hence \(a_1a_2>0\), so \(t_1\) and \(t_2\) have the same sign. The timed
    rotation \((x,y,z)\mapsto(x,-y,-z)\) preserves the causal order and sends
    \(t\) to \(-t\), so we may assume \(t_1,t_2>0\). Then \(a_1,a_2>0\), and
    the first coordinate of \(p_1^{-1}\ast p_2\) gives \(v\geq0\). We now use \cref{eq:singular_acausal_necessary_inequality} once more. Since
    \(a_1,a_2>0\), it gives
    \[
        \cosh v\geq\frac{a_1^2+a_2^2}{2a_1a_2}
        =
        \frac12\left(
        \frac{\max\{a_1,a_2\}}{\min\{a_1,a_2\}}
        +
        \frac{\min\{a_1,a_2\}}{\max\{a_1,a_2\}}
        \right) \geq 1.
    \]
    As \(v\geq0\) and since $s \mapsto s - \sqrt{s^2 - 1}$ is decreasing on $\rinterval{1}{+\infty}$, this implies
    \begin{equation}\label{eq:singular_acausal_exp_bound}
        \begin{aligned}
            e^{-v}
             & =\cosh v-\sqrt{\cosh^2 v-1}               \\
             & \leq
            \frac12\left(
            \frac{\max\{a_1,a_2\}}{\min\{a_1,a_2\}}
            +
            \frac{\min\{a_1,a_2\}}{\max\{a_1,a_2\}}
            \right)
            -\sqrt{
                 \frac14\left(
                 \frac{\max\{a_1,a_2\}}{\min\{a_1,a_2\}}
                 +
                 \frac{\min\{a_1,a_2\}}{\max\{a_1,a_2\}}
                 \right)^2-1}  \\
             & =\frac{\min\{a_1,a_2\}}{\max\{a_1,a_2\}}.
        \end{aligned}
    \end{equation}
    Moreover,
    \begin{equation}\label{eq:singular_acausal_F_bound}
        \begin{aligned}
            F_2-F_1
             & =
            \frac12(\sinh t_2-t_2-\sinh t_1+t_1) =a_2^2-a_1^2
            -\frac12\bigl(t_2+e^{-t_2}-(t_1+e^{-t_1})\bigr) \\
             & \leq a_2^2-a_1^2
            \quad\text{whenever } t_1\leq t_2,
        \end{aligned}
    \end{equation}
    because
    \(t\mapsto t+e^{-t}\) is increasing on \([0,+\infty)\). By
    \cref{eq:singular_acausal_difference,eq:future_origin_heisenberg}, we find
    that
    \begin{equation}\label{eq:singular_acausal_main_chain}
        \begin{aligned}
            0  \geq \frac14(-x^2+y^2+4\abs z) & = a_1^2+a_2^2-2a_1a_2\cosh v
            +
            \abs{F_1-F_2+2a_1a_2\sinh v}                                              \\
                                              & \geq a_1^2+a_2^2-2a_1a_2\cosh v
            +F_1-F_2+2a_1a_2\sinh v                                                   \\
                                              & =a_1^2+a_2^2-2a_1a_2(\cosh v-\sinh v)
            +F_1-F_2                                                                  \\
                                              & =a_1^2+a_2^2-2a_1a_2e^{-v}+F_1-F_2.
        \end{aligned}
    \end{equation}

    Using \cref{eq:singular_acausal_exp_bound,eq:singular_acausal_F_bound}, if
    \(t_1\leq t_2\), the last term is at least
    \(a_2^2-a_1^2+F_1-F_2\geq0\). If \(t_2\leq t_1\), it is at least
    \(a_1^2-a_2^2+F_1-F_2\geq0\).
    Therefore, we have \( -x^2+y^2+4\abs z=0\). Every inequality in
    \cref{eq:singular_acausal_main_chain} is actually an equality, and therefore
    equality in the estimates above gives \(a_1=a_2\) and \(v=0\). Hence
    \(t_1=t_2\) and \(p_1=p_2\), so no two distinct points of \(S_{(-1,0)}\) are
    causally related.
\end{proof}

All the surfaces considered here are non-compact and so unbounded. For compact
sets, one can always find a point whose future contains the whole set. It is
therefore natural to ask whether an entire timelike-parameter surface with
\(c>0\) can be contained in the future of a single point, as happens for
hyperboloids in Minkowski space. The answer is negative.

\begin{proposition}
    Let \(S_{(c,d)}\) be a surface with timelike parameters and \(c>0\). Then
    there is no point \(p\in\heis\) such that \(S_{(c,d)}\subset J^+(p)\).
\end{proposition}
\begin{proof}
    Since isometries preserve causal futures, \Cref{prop:timelike_isometric_surfaces}
    allows us to prove the statement for \(S_{(c,0)}\). We first show that
    \(S_{(c,0)}\) is not contained in \(J^+(p(a))\), where \(p(a)=(a,0,0)\) and
    \(a<0\). We consider the curve \(q(t)=S_{(c,0)}(t,0)\), and compute
    \[
        p(a)^{-1}\ast q(t)
        =
        \left(\cosh t+c-a,\sinh t,\frac12(t+(c-a)\sinh t)\right).
    \]
    Since \(c-a>0\), and we find that
    \[
        \begin{split}
            -x^2+y^2+4\abs{z}
             & =-(\cosh t+c-a)^2+\sinh^2t
            +2\abs{t+(c-a)\sinh t}          \\
             & =-(c-a)^2-1-2(c-a)e^{-t}+2t.
        \end{split}
    \]
    This tends to \(+\infty\) as \(t\to+\infty\). By
    \cref{eq:future_origin_heisenberg}, \(q(t)\notin J^+(p(a))\) for all large
    \(t\), and \(S_{(c,0)}\not\subset J^+(p(a))\).

    Now suppose, for contradiction, that \(S_{(c,0)}\subset J^+(p')\) for some
    \(p'=(p_1,p_2,p_3)\). For \(a\to-\infty\), we have
    \[
        p(a)^{-1}\ast p'
        =
        \left(p_1-a,p_2,p_3-\frac12ap_2\right),
    \]
    and
    \[
        -(p_1-a)^2+p_2^2+2\abs{2p_3-ap_2}\to-\infty.
    \]
    Also \(p_1-a>0\) for \(a\ll0\), so \cref{eq:future_origin_heisenberg}
    gives \(p'\in I^+(p(a))\) for \(a\ll0\). Therefore
    \(J^+(p')\subset J^+(p(a))\), and hence
    \(S_{(c,0)}\subset J^+(p(a))\), contradicting the first part of the proof.
\end{proof}

We end this subsection by recording the limiting behaviour as \(k \to +\infty\) of
\(S_{(c,0)}\) with \(c>0\), when the curvature parameter \(k\) is restored.

\begin{proposition}
    After the change of variable \(s=\sinh(kt)/k\), the surfaces \(S_{(c,0)}\),
    with \(c>0\), are parametrized by
    \begin{equation}\label{eq:timelike_k_rescaled_surfaces}
        B_u\left(
        \frac{\sqrt{1+k^2s^2}}{k}+c,\,
        s,\,
        \frac{1}{2k^2}\arcsinh(ks)+\frac{cs}{2}
        \right).
    \end{equation}
    As \(k\to+\infty\), the parametrisations in
    \cref{eq:timelike_k_rescaled_surfaces} converge to
    \begin{equation}\label{eq:timelike_k_limit_surface}
        B_u\left(
        \abs{s}+c,\,
        s,\,
        \frac{cs}{2}
        \right).
    \end{equation}
    Moreover, as \(c\to0^+\), the hypersurfaces in
    \cref{eq:timelike_k_limit_surface} converge to
    \(J^+(0)\setminus I^+(0)\).
\end{proposition}
\begin{proof}
    With the curvature parameter \(k\), the \((c,0)\)-surface is parametrized by
    \[
        B_u\left(
        \frac{\cosh(kt)}{k}+c,\,
        \frac{\sinh(kt)}{k},\,
        \frac{1}{2k}\left(t+c\sinh(kt)\right)
        \right).
    \]
    Setting \(s=\sinh(kt)/k\), we have
    \(t=k^{-1}\arcsinh(ks)\) and
    \(\cosh(kt)/k=\sqrt{1+k^2s^2}/k\), which gives
    \cref{eq:timelike_k_rescaled_surfaces}. Since
    \(\sqrt{1+k^2s^2}/k\to\abs{s}\) and
    \(\arcsinh(ks)/(2k^2)\to0\) \textcolor{blue}{as \(k \to +\infty\)}, the parametrisations converge to
    \cref{eq:timelike_k_limit_surface}.
    For every \(u\), the curve
    \(s\mapsto B_u\left(\abs{s}+c,s,cs/2\right)\) is horizontal and null on
    each of the intervals \(s>0\) and \(s<0\), and it is not smooth at \(s=0\).
    Thus the limit in \cref{eq:timelike_k_limit_surface} is a null
    hypersurface.

    For fixed \(c>0\), the profile curve of
    \cref{eq:timelike_k_limit_surface} in the plane \(y=0\) is obtained from
    \(s\cosh u+(\abs{s}+c)\sinh u=0\). Hence
    \(\tanh u=-s/(\abs{s}+c)\) and with \(r=2cs\), the corresponding profile is
    \[
        \left(\sqrt{c^2+2c\abs{s}},0,\frac{cs}{2}\right)
        =
        \left(\sqrt{c^2+\abs{r}},0,\frac{r}{4}\right)
        \longrightarrow
        \left(\sqrt{\abs{r}},0,\frac{r}{4}\right)
    \]
    as \(c\to0^+\).
    Boosting this limiting profile gives
    \[
        B_u\left(\sqrt{\abs{r}},0,\frac{r}{4}\right)
        =
        \left(
        \sqrt{\abs{r}}\cosh u,\,
        \sqrt{\abs{r}}\sinh u,\,
        \frac{r}{4}
        \right).
    \]
    This parametrizes the set \(\{-x^2+y^2+4\abs{z}=0,\ x\geq0\}\), which is
    \(J^+(0)\setminus I^+(0)\) by \cref{eq:future_origin_heisenberg}.
\end{proof}

\subsection{Surfaces of null parameters}

We next consider the family of surfaces \(S_{(c,d)}\) with null parameters, that
is, those for which \(c^2-d^2=0\). Equivalently, apart from the degenerate case
\((c,d)=(0,0)\), we may write \(c=\omega d\), where
\(\omega\in\{-1,1\}\). In this case the parametrisation
\cref{eq:general_expression_CMC_surface} reduces to
\begin{equation}\label{eq:null_parameter_surface_parametrization}
    S_{(c,\omega c)}(t,u)
    =
    B_u\left(
    \cosh(t)+c,\;
    \sinh(t)+\omega c,\;
    \frac{1}{2}\left(t-\omega c e^{-\omega t}+\omega c\right)
    \right).
\end{equation}
We have illustrated a surface of null parameter in \cref{fig:nullsurface}.

\begin{figure}[ht]
    \centering
    \includegraphics[scale = 0.6]{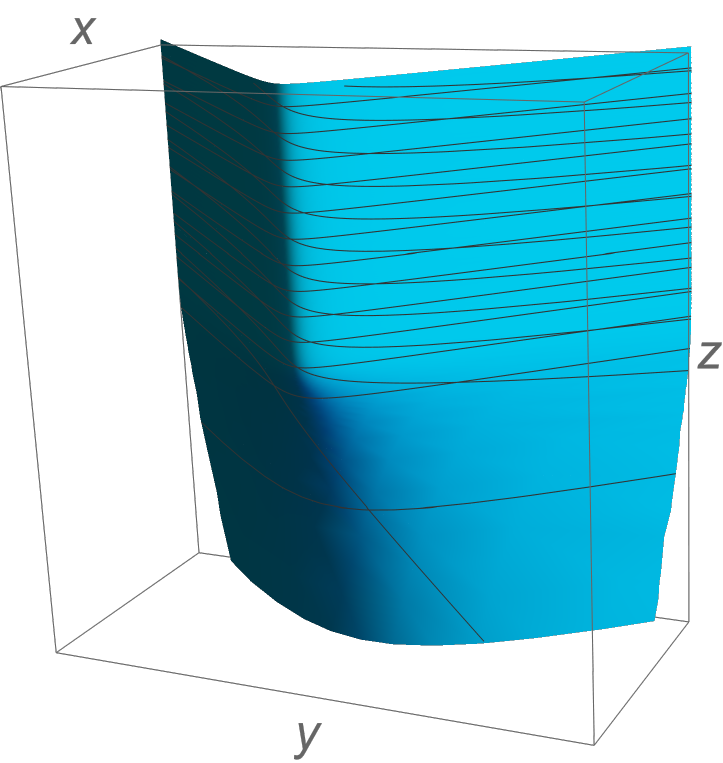}
    \caption{A surface \(S_{(c,d)}\) with null parameter.}
    \label{fig:nullsurface}
\end{figure}

As in \cref{prop:spacelike_isometric_surfaces,prop:timelike_isometric_surfaces},
we first give sufficient conditions for two null-parameter surfaces to be
isometric by either a time-preserving boost or a time-preserving rotation.

\begin{proposition}\label{prop:null_surfaces_isometric}
    Let \(S_{(c,\omega c)}\) and \(S_{(c',\omega' c')}\) be two surfaces with
    nonzero null parameters. Then the following hold:
    \begin{enumerate}[label=\normalfont(\roman*)]
        \item If \(cc'>0\) and \(\omega=\omega'\), then the two surfaces are
              isometric by a time-preserving boost.
        \item\label{prop:null_surfaces_timed_rotations_isometric} If \(cc'>0\)
              and \(\omega=-\omega'\), then the two surfaces are isometric by a
              time-preserving rotation.
    \end{enumerate}
\end{proposition}
\begin{proof}
    We first prove \normalfont{(i)}.
    Since \(cc'>0\), choose
    \(\alpha\in\mathbb R\) such that \(c e^{\omega\alpha}=c'\), that is,
    \(\alpha=\omega\log(c'/c)\). Starting from
    \(S_{(c,\omega c)}(t,u)=B_u\gamma_{(c,\omega c)}(t)\), make the change of
    parameters \(s=t+\alpha\), \(v=u-\alpha\). Then \(t=s-\alpha\),
    \(u=v+\alpha\), and
    \begin{align*}
        S_{(c,\omega c)} & (s-\alpha,v+\alpha)
        =
        B_{v+\alpha}\gamma_{(c,\omega c)}(s-\alpha) \\
                         & =
        \left(
        \cosh(s+v)+c'e^{\omega v},\;
        \sinh(s+v)+\omega c'e^{\omega v},\;
        \frac{1}{2}\left(s-\alpha-\omega c'e^{-\omega s}+\omega c\right)
        \right).
    \end{align*}
    Therefore, we have that
    \[
        S_{(c',\omega c')}(s,v)
        =
        \left(0,0,\frac12(\alpha+\omega c'-\omega c)\right)
        \ast S_{(c,\omega c)}(s-\alpha,v+\alpha),
    \]
    proving \normalfont{(i)}.

    We now prove \normalfont{(ii)}.
    By \cref{prop:timed_rotation_behaviour}, a timed rotation sends
    \(S_{(c,\omega c)}\) to \(S_{(c,-\omega c)}\). Since
    \(-\omega=\omega'\) and \(cc'>0\), part \normalfont{(i)}
    gives a time-preserving boost from \(S_{(c,-\omega c)}\) to
    \(S_{(c',\omega' c')}\). The composition of these two isometries has
    time-preserving rotation Lorentz part.
\end{proof}

\begin{lemma}\label{lemma:null_surface_vertical_line_intersection}
    Every surface \(S_{(c,\omega c)}\) with \(c\neq0\) is intersected at most
    once by any vertical line.
\end{lemma}
\begin{proof}
    The first two components of the
    parametrisation of  \(S_{(c,\omega c)}(t,u)\) are
    \[
        x=\cosh(t+u)+c e^{\omega u},
        \qquad
        y=\sinh(t+u)+\omega c e^{\omega u}.
    \]
    Therefore, we get that
    \[
        x-\omega y
        =
        \cosh(t+u)-\omega\sinh(t+u)
        =
        e^{-\omega(t+u)}.
    \]
    Thus the sum \(t+u\) is uniquely determined by the \((x,y)\)-coordinates.
    Once \(t+u\) is known, the identity
    \[
        c e^{\omega u}=x-\cosh(t+u)
    \]
    determines \(u\), because \(c\neq0\), and \(t\) is also determined.
    The parameters \((t,u)\), and hence the \(z\)-coordinate of the point on the
    surface, are uniquely determined by \((x,y)\). Thus no vertical line can
    meet the surface more than once.
\end{proof}

The presence of vertical lines gives an isometry invariant that distinguishes
the degenerate null surface from all other null-parameter surfaces.

\begin{corollary}\label{coro:surface_null_with_vanishing_parameters_no_isometry}
    The surface \(S_{(0,0)}\) is not isometric to any surface
    \(S_{(c,\omega c)}\) with \(c\neq0\).
\end{corollary}
\begin{proof}
    When \(c=0\), the parametrisation reduces to
    \[
        S_{(0,0)}(t,u)
        =
        \left(
        \cosh(t+u),\;
        \sinh(t+u),\;
        \frac{t}{2}
        \right).
    \]
    Hence \(S_{(0,0)}\) is the vertical cylinder over the hyperbola
    \(x^2-y^2=1\), \(x>0\). In particular, it contains, for every \(a\in\mathbb R\), the vertical line
    \[
        \{(\cosh a,\sinh a,z):z\in\mathbb R\}.
    \]

    By \cref{thm:sub_lorentzian_isometries}, the first two components of every
    isometry of \(\heis\) are independent of the \(z\)-coordinate. Thus
    vertically aligned points are sent to vertically aligned points. If
    \(S_{(0,0)}\) were isometric to a surface \(S_{(c,\omega c)}\) with
    \(c\neq0\), the image surface would contain a vertical line. This
    contradicts \cref{lemma:null_surface_vertical_line_intersection}.
\end{proof}

We now describe the characteristic set of the null-parameter surfaces, in
parallel with
\cref{prop:spacelike_characteristic_set,prop:timelike_surfaces_characteristic_set}.

\begin{proposition}\label{prop:null_surfaces_characteristic_set}
    The characteristic set of a surface \(S_{(c,\omega c)}\) with null
    parameters is non-empty if and only if \(c<0\). In this case, set
    \[
        t^*=\omega\log(-2c),
        \qquad
        z^*=\frac{1}{2}\left(t^*-\omega c e^{-\omega t^*}+\omega c\right).
    \]
    Then the characteristic set is the open half-line
    \[
        \{(s,-\omega s,z^*):s>0\}.
    \]
    In particular, it is a horizontal null curve.
\end{proposition}
\begin{proof}
    As in the proofs of
    \cref{prop:spacelike_characteristic_set,prop:timelike_surfaces_characteristic_set},
    it is enough to work at \(u=0\). Since \(\partial_tS_{(c,\omega c)}\) is
    horizontal by construction, a point is characteristic precisely when
    \(\partial_uS_{(c,\omega c)}\) is also horizontal. Setting \(d=\omega c\)
    in \cref{eq:general_characteristic_equation}, this
    condition becomes
    \[
        1+2c\bigl(\cosh t-\omega\sinh t\bigr)=0,
        \qquad\text{that is,}\qquad
        1+2c e^{-\omega t}=0.
    \]
    This equation has a solution if and only if \(c<0\), and in that case the
    unique solution is \(t^*=\omega\log(-2c)\). Set \(\eta=\cosh(t^*)+c\). Then
    \[
        \eta=\frac{1}{-4c}>0,
        \qquad
        \sinh(t^*)+\omega c=-\omega\eta.
    \]
    Substituting \(t=t^*\) in the parametrisation gives
    \[
        S_{(c,\omega c)}(t^*,u)
        =
        B_u(\eta,-\omega\eta,z^*)
        =
        \bigl(\eta e^{-\omega u},-\omega\eta e^{-\omega u},z^*\bigr).
    \]
    Since \(s=\eta e^{-\omega u}\) ranges over \((0,+\infty)\), the
    characteristic set is exactly \(\{(s,-\omega s,z^*):s>0\}\). This curve is
    horizontal and its horizontal velocity \((1,-\omega)\) is null.
\end{proof}

Once again, the characteristic set provides a useful isometry invariant.

\begin{corollary}\label{coro:null_surfaces_opposite_c_non_isometric}
    Let \(S_{(c,\omega c)}\) and \(S_{(c',\omega' c')}\) be two surfaces with
    null parameters. If \(c<0<c'\), then the two surfaces are not
    isometric.
\end{corollary}
\begin{proof}
    By \cref{prop:null_surfaces_characteristic_set}, the surface
    \(S_{(c,\omega c)}\) has non-empty characteristic set, while
    \(S_{(c',\omega' c')}\) has empty characteristic set because \(c'>0\).
    Since isometries preserve the horizontal distribution, they preserve
    characteristic points. Hence no isometry between the two surfaces can exist.
\end{proof}

We now rule out time-inverting isometries between null-parameter surfaces.

\begin{proposition}\label{coro:no_time_inverting_isometry_null_case}
    There is no time-inverting isometry between any two surfaces
    \(S_{(c,\omega c)}\) and \(S_{(c',\omega' c')}\) with null parameters.
\end{proposition}
\begin{proof}
    By
    \cref{coro:surface_null_with_vanishing_parameters_no_isometry,coro:null_surfaces_opposite_c_non_isometric},
    it is enough to consider the cases where either \(c,c'<0\), or \(c,c'>0\),
    or \(c=c'=0\).

    When \(c,c'<0\), the argument is the same as in
    \cref{coro:no_time_inverting_isometry_spacelike_surfaces}: the characteristic
    set is, by
    \cref{prop:null_surfaces_characteristic_set}, a null half-line contained in \(\{x>0\}\), while a time-inverting
    isometry sends such a half-line eventually into \(\{x<0\}\).

    Suppose now that \(c,c'>0\). By \cref{prop:null_surfaces_isometric}, up to
    composing with time-preserving boosts, we may reduce to the case
    \(S_{(1,\omega)}\) and \(S_{(1,\omega')}\). Both surfaces project to the same region of the Minkowski plane,
    namely
    \[
        P=\{x^2-y^2>1\}\cap\{x>0\}.
    \]
    The region \(P\) is unbounded only in future timelike directions: along every
    sequence in \(P\) escaping to infinity, the \(x\)-coordinate tends to \(+\infty\).
    A time-inverting Lorentz map sends future timelike directions to past timelike
    directions, so its image of \(P\), even after a translation, is unbounded in the
    negative \(x\)-direction. This cannot equal \(P\), which is contained in
    \(\{x>0\}\).

    Finally, when \(c=c'=0\), the surface \(S_{(0,0)}\) is the vertical cylinder
    over the future branch \(\{x^2-y^2=1,\ x>0\}\). The same asymptotic-direction
    argument shows that this branch cannot be mapped onto itself by a
    time-inverting affine Lorentz transformation of the Minkowski plane. Thus no
    time-inverting isometry exists in this case either.
\end{proof}

We are able to complete the boost-isometry classification for surfaces with
null parameters.

\begin{proposition}\label{prop:null_surfaces_boost_classification}
    Two surfaces \(S_{(c,\omega c)}\) and \(S_{(c',\omega' c')}\) with null
    parameters are isometric by a time-preserving boost if and only if either
    \(cc'>0\) and \(\omega=\omega'\), or \(c=c'=0\).
\end{proposition}
\begin{proof}
    The implication from right to left follows from
    \cref{prop:null_surfaces_isometric} when \(cc'>0\), and is trivial when
    \(c=c'=0\), since then both parametrisations give the same surface
    \(S_{(0,0)}\).

    Conversely, suppose that the two surfaces are isometric by a time-preserving
    boost. By
    \cref{coro:surface_null_with_vanishing_parameters_no_isometry,coro:null_surfaces_opposite_c_non_isometric},
    and after possibly exchanging the two surfaces, either \(c=c'=0\), or
    \(cc'>0\). It remains to prove, in the second case, that
    \(\omega=\omega'\).

    Assume first that \(c,c'<0\). By
    \cref{prop:null_surfaces_characteristic_set}, the characteristic set of
    \(S_{(c,\omega c)}\) is a null half-line with direction \((1,-\omega)\) in
    the Minkowski plane, and similarly for \(S_{(c',\omega' c')}\). A
    time-preserving boost preserves each future null direction. Since isometries
    preserve characteristic points, the two directions must agree, and therefore
    \(\omega=\omega'\).

    Assume now that \(c,c'>0\). By \cref{prop:null_surfaces_isometric}, after
    composing with time-preserving boosts we may reduce to the case of
    \(S_{(1,\omega)}\) and \(S_{(1,\omega')}\). Since the surfaces are
    boost-invariant, any time-preserving boost between them can then be reduced
    to a left translation. Both surfaces project to
    \[
        P=\{x^2-y^2>1\}\cap\{x>0\}.
    \]
    The only translation of the Minkowski plane preserving \(P\) is the trivial
    one and therefore the left translation must be
    vertical.

    On \(S_{(1,\omega)}\), the quantity \(r=x^2-y^2\) satisfies
    \(r=1+2e^{-\omega t}\). The \(z\)-coordinate, as a function of
    \(r>1\), is
    \[
        z_\omega(r)
        =
        \frac{\omega}{2}
        \left(
        -\log\left(\frac{r-1}{2}\right)-\frac{r-1}{2}+1
        \right).
    \]
    In other words, \(S_{(1,\omega)}\) is the graph over \(P\) given by
    \(S_{(1,\omega)}=\{(x,y,z_\omega(x,y)):(x,y)\in P\}\), and similarly for
    \(S_{(1,\omega')}\). If a vertical translation \(T_\lambda(x,y,z)=(x,y,z+\lambda)\) sent
    \(S_{(1,\omega)}\) to \(S_{(1,\omega')}\), then
    \(z_\omega(x,y)+\lambda=z_{\omega'}(x,y)\) for every \((x,y)\in P\), so
    \(z_{\omega'}-z_\omega\) would be constant on \(P\). This is impossible when \(\omega'=-\omega\), since then
    \(z_{\omega'}=-z_\omega\) and \(z_\omega\) is not constant. Thus
    \(\omega=\omega'\).
\end{proof}

We also obtain the corresponding classification under time-preserving rotations.

\begin{proposition}\label{prop:null_surfaces_rotation_classification}
    Two surfaces \(S_{(c,\omega c)}\) and \(S_{(c',\omega' c')}\) with null
    parameters are isometric by a time-preserving rotation if and only if either
    \(cc'>0\) and \(\omega=-\omega'\), or \(c=c'=0\).
\end{proposition}
\begin{proof}
    Assume first that \(cc'>0\) and \(\omega=-\omega'\). By
    \cref{prop:timed_rotation_behaviour}, a timed rotation sends
    \(S_{(c,\omega c)}\) to \(S_{(c,-\omega c)}\). Since
    \(-\omega=\omega'\), \cref{prop:null_surfaces_boost_classification} gives a
    time-preserving boost from \(S_{(c,-\omega c)}\) to
    \(S_{(c',\omega' c')}\). The composition is a time-preserving rotation. If
    \(c=c'=0\), then both parametrisations give the same surface \(S_{(0,0)}\),
    and this surface is invariant under the timed rotation.

    Conversely, suppose that \(S_{(c,\omega c)}\) and
    \(S_{(c',\omega' c')}\) are isometric by a time-preserving rotation.
    Composing this isometry with the timed rotation
    \((x,y,z)\mapsto(x,-y,-z)\), which sends
    \(S_{(c',\omega' c')}\) to \(S_{(c',-\omega' c')}\), gives a
    time-preserving boost from \(S_{(c,\omega c)}\) to
    \(S_{(c',-\omega' c')}\). By
    \cref{prop:null_surfaces_boost_classification}, either \(c=c'=0\), or
    \(cc'>0\) and \(\omega=-\omega'\).
\end{proof}

We now discuss the causal behaviour of null-parameter surfaces, starting with the case \(c\geq0\).

\begin{proposition}\label{prop:null_surfaces_nonnegative_c_acausal}
    Every surface \(S_{(c,\omega c)}\) with \(c\geq0\) is acausal. In
    particular, it is achronal.
\end{proposition}
\begin{proof}
    We first consider \(c=0\). In this case
    \[
        S_{(0,0)}(t,u)
        =
        \left(
        \cosh(t+u),\;
        \sinh(t+u),\;
        \frac{t}{2}
        \right),
    \]
    so the surface is the vertical cylinder over the future branch
    \(\{x^2-y^2=1,\ x>0\}\). If two points of this surface were joined by a
    causal curve, then their projections would be joined by a causal curve in
    the Minkowski plane with endpoints on this hyperbola. Since the future branch is acausal in the Minkowski plane, the projected curve
    must be constant. But a future-directed causal curve with constant
    \(x\)-coordinate is constant, so the two points coincide.

    Now let \(c>0\). By \cref{prop:null_surfaces_isometric}, it is enough to
    prove the claim for \(S_{(1,\omega)}\). Put \(r=x^2-y^2\) and define, on
    \(P=\{r>1,\ x>0\}\),
    \[
        F_\omega(x,y,z)
        =
        z-\frac{\omega}{2}
        \left(
        -\log\left(\frac{r-1}{2}\right)-\frac{r-1}{2}+1
        \right).
    \]
    The surface \(S_{(1,\omega)}\) is contained in the zero level set of
    \(F_\omega\). A direct computation gives
    \[
        X(F_\omega)
        =
        -\frac{1}{2}
        \left(
        y-\omega x\frac{r+1}{r-1}
        \right),
        \qquad
        Y(F_\omega)
        =
        \frac{1}{2}
        \left(
        x-\omega y\frac{r+1}{r-1}
        \right),
    \]
    and therefore
    \[
        g\bigl(\operatorname{grad}_{\heis}F_\omega,
        \operatorname{grad}_{\heis}F_\omega\bigr)
        =
        -\frac{r^2}{(r-1)^2}<0.
    \]
    Thus \(\operatorname{grad}_{\heis}F_\omega\) is timelike and has fixed time
    orientation on \(P\).

    Suppose, for contradiction, that two points of \(S_{(1,\omega)}\) are
    joined by a non-constant future-directed causal curve \(\gamma\). Its
    projection to the Minkowski plane is future-directed causal. Since
    \(P=\{r>1,\ x>0\}\) is future-closed in the Minkowski plane, the projected
    curve stays in \(P\), so \(F_\omega\circ\gamma\) is defined. Moreover
    \[
        \frac{d}{dt}(F_\omega\circ\gamma)(t)
        =
        g\bigl(\operatorname{grad}_{\heis}F_\omega,\dot\gamma(t)\bigr)
    \]
    has a fixed non-zero sign wherever \(\dot\gamma\neq0\), because both
    \(\operatorname{grad}_{\heis}F_\omega\) and \(\dot\gamma\) are causal and
    \(\operatorname{grad}_{\heis}F_\omega\) is timelike. Hence
    \(F_\omega\circ\gamma\) is strictly monotone. This contradicts the fact that
    it vanishes at both endpoints. Therefore no two distinct points of
    \(S_{(1,\omega)}\) are causally related.
\end{proof}

In contrast, the null-parameter surfaces with \(c < 0\) contain
chronologically related points.

\begin{proposition}\label{prop:null_surfaces_negative_c_not_achronal}
    Every surface \(S_{(c,\omega c)}\) with \(c<0\) is not achronal.
\end{proposition}
\begin{proof}
    Since achronality is invariant under isometries,
    \cref{prop:null_surfaces_isometric} allows us to prove the claim for
    \(S_{(-1,-1)}\). By \cref{eq:null_parameter_surface_parametrization}, its
    parametrisation is
    \[
        \left(
        \cosh(t+u)-e^u,\;
        \sinh(t+u)-e^u,\;
        \frac{1}{2}\left(t+e^{-t}-1\right)
        \right).
    \]
    The curve obtained by taking \(t=0\) is
    \[
        \gamma(a)=(-\sinh a,-\cosh a,0).
    \]
    We also consider the curve on the surface with \(y=0\). Writing
    \(s=e^u\), it is parametrised by
    \[
        \sigma(s)
        =
        \left(
        A_s,\,
        0,\,
        \frac{1}{2}B_s
        \right),
        \qquad s>0,
    \]
    where
    \[
        A_s=\sqrt{1+s^2}-s,
        \qquad
        B_s=\log\left(1+\sqrt{1+s^{-2}}\right)
        +\frac{1}{1+\sqrt{1+s^{-2}}}-1.
    \]
    A direct computation gives
    \[
        \gamma(a)^{-1}\ast\sigma(s)
        =
        \left(
        \sinh a+A_s,\,
        \cosh a,\,
        \frac{1}{2}\left(B_s-A_s\cosh a\right)
        \right).
    \]
    For fixed \(s>0\) and \(a\) sufficiently large, the \(x\)-coordinate is
    positive and \(B_s-A_s\cosh a<0\). Therefore, using
    \cref{eq:future_origin_heisenberg},
    \[
        -x^2+y^2+4\abs{z}
        =
        1-A_s^2+2A_se^{-a}-2B_s
        \longrightarrow
        1-A_s^2-2B_s
        \qquad\text{as } a\to+\infty
    \]
    As \(s\to0^+\), we have \(A_s\to1\) and \(B_s\to+\infty\), so the last
    limit tends to \(-\infty\). Hence we may choose \(s>0\) sufficiently small,
    and then \(a>0\) sufficiently large, so that
    \[
        \gamma(a)^{-1}\ast\sigma(s)\in I^+(0).
    \]
    By left-invariance, this means \(\gamma(a)\ll\sigma(s)\). Thus the surface
    contains chronologically related points, and is not achronal.
\end{proof}

We conclude the classification of the surfaces with non-zero horizontal mean
curvature by showing that surfaces whose parameter vectors have different causal
characters are not isometric.

\begin{proposition}\label{prop:different_causal_character_not_isometric}
    Let \(S_{(c,d)}\) and \(S_{(c',d')}\) be two surfaces with non-zero
    horizontal mean curvature. If the parameter vectors \((c,d)\) and
    \((c',d')\) have different causal characters, then the two surfaces are not
    isometric.
\end{proposition}
\begin{proof}
    Isometries preserve regularity and characteristic sets. In particular, a
    singular surface of timelike parameters cannot be isometric to a regular surface of spacelike or null
    parameters, by \cref{prop:real_singularity}. We may therefore restrict to
    regular surfaces.

    First suppose that one surface has spacelike parameters. By
    \cref{prop:spacelike_characteristic_set}, its characteristic set is one
    horizontal null half-line. A regular surface of timelike parameters has characteristic set
    either empty or equal to two disjoint half-lines, by
    \cref{prop:timelike_surfaces_characteristic_set}. Hence a
    surface of spacelike parameters cannot be isometric to a surface of timelike ones.

    It remains, in the spacelike case, to compare with null parameters. If the
    surface of null parameters has empty characteristic set, then the characteristic set
    already distinguishes the two surfaces. Otherwise both characteristic sets
    are single half-lines starting from the \(z\)-axis. Then any isometry
    between them must have vertical translation part, and the Lorentz part
    preserves the function
    \[
        f(x,y,z)=x^2-y^2.
    \]
    Thus the image of each surface under \(f\) would have to be the same. This
    is impossible: after reducing the surface of spacelike parameters to \(S_{(0,d)}\), we
    have
    \[
        f(S_{(0,d)}(t,u))=1-d^2-2d\sinh t,
    \]
    which ranges over all of \(\mathbb R\), while for a nonzero surface of null parameters
    \(S_{(c,\omega c)}\),
    \[
        f(S_{(c,\omega c)}(t,u))=1+2c e^{-\omega t},
    \]
    whose image is bounded on one side by \(1\).

    Finally consider a regular surface of timelike parameters and a surface of null ones. If their
    characteristic sets are not both empty, they are distinguished by
    \cref{prop:timelike_surfaces_characteristic_set,prop:null_surfaces_characteristic_set}.
    Thus we only have to consider the case in which both are empty. By
    \cref{prop:timelike_isometric_surfaces,prop:null_surfaces_isometric}, we may
    reduce to a surface of timelike parameters \(S_{(c,0)}\) with \(c>0\) and to
    \(S_{(1,\omega)}\). Their projections to the \((x,y)\)-plane are,
    respectively,
    \[
        \{x^2-y^2\geq(c+1)^2\}\cap\{x>0\},
        \qquad
        \{x^2-y^2>1\}\cap\{x>0\}.
    \]
    A sub-Lorentzian isometry acts on the \((x,y)\)-plane by a Lorentz map
    followed by a translation. The Lorentz part preserves \(x^2-y^2\), and a
    translation is a homeomorphism. It cannot send a projected region with
    boundary to one without boundary. Hence the surface of timelike and null parameters are
    not isometric.
\end{proof}

\subsection{Maximal surfaces}

We now consider the zero-curvature case \(k=0\), which gives the
boost-invariant maximal surfaces. Recall that, by
\cref{prop:planar_constant_curvature_curves}, these surfaces are obtained by
horizontally lifting spacelike lines in the Minkowski plane and then applying
the boost action, and their parametrisation \(S_{(\varepsilon,\alpha,\beta)}\) was given in
\cref{eq:maximal_surface_parametrization}. A plot of a typical maximal surface is given in \cref{fig:maximal}.

\begin{figure}[h]
    \centering
    \includegraphics[scale = 0.6]{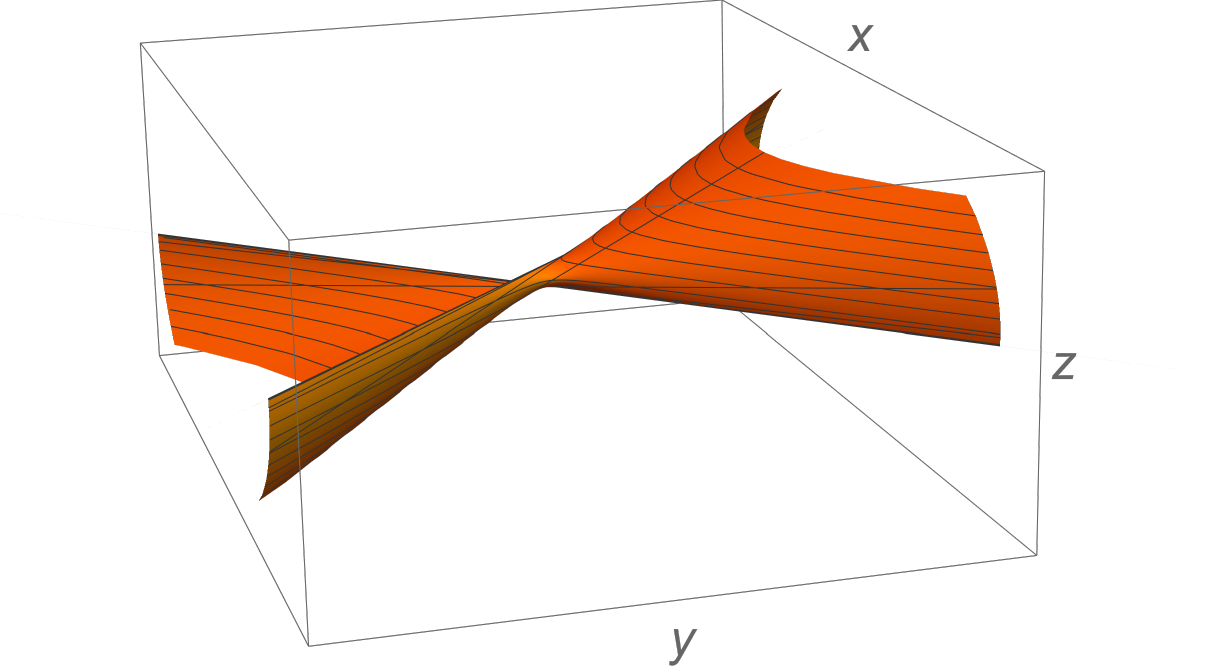}
    \caption{A maximal surface.}
    \label{fig:maximal}
\end{figure}

We first show that it is not restrictive to assume that \(\varepsilon=1\), since
changing the sign of \(\varepsilon\) produces an isometric surface via a timed
rotation.

\begin{proposition}\label{prop:maximal_surfaces_epsilon_sign_isometric}
    For every \(\alpha,\beta\in\mathbb R\), the maximal surfaces
    \(S_{(1,\alpha,\beta)}\) and \(S_{(-1,\alpha,\beta)}\) are isometric. More
    precisely, they are isometric both by a time-preserving rotation and by a
    vertical translation.
\end{proposition}
\begin{proof}
    By \cref{thm:sub_lorentzian_isometries}, the map \(A(x,y,z)=(x,-y,-z)\)
    is a time-preserving rotation. Since \(A\circ B_u=B_{-u}\circ A\), and
    \[
        A\bigl(\gamma_{(1,\alpha,\beta)}(t)\bigr)
        =
        \left(
        \sinh(\beta)t+\alpha,\;
        -\cosh(\beta)t,\;
        -\frac{\alpha}{2}\cosh(\beta)t
        \right)
        =
        \gamma_{(-1,\alpha,\beta)}(t),
    \]
    we get
    \[
        A\bigl(S_{(1,\alpha,\beta)}(t,u)\bigr)
        =
        B_{-u}\gamma_{(-1,\alpha,\beta)}(t)
        =
        S_{(-1,\alpha,\beta)}(t,-u).
    \]
    Thus \(A\bigl(S_{(1,\alpha,\beta)}\bigr)=S_{(-1,\alpha,\beta)}\).

    As for the vertical-translation statement, consider the vector
    \(p=(0,0,-\alpha^2\cosh(\beta)\sinh(\beta))\). Since \(p\) is vertical,
    left translation by \(p\) only changes the \(z\)-coordinate and commutes
    with boosts. Thus
    \begin{equation}\label{eq:translated_maximal_epsilon_positive}
        p\ast S_{(1,\alpha,\beta)}(t,u)
        =
        B_u\left(
        \sinh(\beta)t+\alpha,\;
        \cosh(\beta)t,\;
        \frac{\alpha}{2}\cosh(\beta)\bigl(t-2\alpha\sinh(\beta)\bigr)
        \right).
    \end{equation}
    Making the change of parameters
    \[
        s=-t+2\alpha\sinh(\beta),
        \qquad
        v=u+2\beta,
    \]
    so that \(t=-s+2\alpha\sinh(\beta)\), \(u=v-2\beta\),
    \cref{eq:translated_maximal_epsilon_positive} becomes
    \[
        B_{v-2\beta}\left(
        \alpha\cosh(2\beta)-s\sinh(\beta),\;
        \alpha\sinh(2\beta)-s\cosh(\beta),\;
        -\frac{\alpha}{2}\cosh(\beta)s
        \right).
    \]
    Since \(B_{v-2\beta}=B_vB_{-2\beta}\), the addition formulae give
    \[
        p\ast S_{(1,\alpha,\beta)}(t,u)
        =
        B_v\left(
        \alpha+s\sinh(\beta),\;
        -s\cosh(\beta),\;
        -\frac{\alpha}{2}\cosh(\beta)s
        \right)
        =
        S_{(-1,\alpha,\beta)}(s,v).
    \]
    Therefore \(p\ast S_{(1,\alpha,\beta)}=S_{(-1,\alpha,\beta)}\).
\end{proof}

In \Cref{coro:maximal_surface_boost_isometric_iff_rotation_isometric}, we will
see that changing the sign of \(\varepsilon\) does not affect whether an
isometry between two maximal surfaces can be chosen to be a boost or a
rotation. We therefore work from now on with \(\varepsilon=1\). We next describe
the characteristic sets of these surfaces.

\begin{proposition}\label{prop:maximal_surfaces_characteristic_set}
    If \(\alpha=0\), the regular part of \(S_{(1,\alpha,\beta)}\) has empty
    characteristic set. If \(\alpha\neq0\), then the characteristic set is the
    disconnected union of the two half-lines
    \begin{equation}\label{eq:maximal_surface_characteristic_half_lines}
        \left\{
        \left(
        \sign(\alpha)s,\,
        \omega\sign(\alpha)s,\,
        \frac{\omega\alpha^2}{4}\left(e^{2\omega\beta}+1\right)
        \right):s>0
        \right\},
        \qquad \omega\in\{-1,1\}.
    \end{equation}
    With the parameter \(s\) increasing, these half-lines are future-directed if
    \(\alpha>0\) and past-directed if \(\alpha<0\).
\end{proposition}
\begin{proof}
    As in
    \cref{prop:spacelike_characteristic_set,prop:timelike_surfaces_characteristic_set,prop:null_surfaces_characteristic_set},
    it is enough to work at \(u=0\). At points where the tangent plane is
    defined, \(\partial_tS_{(1,\alpha,\beta)}\) is horizontal by construction and thus characteristic points are exactly those for which
    \(\partial_uS_{(1,\alpha,\beta)}\) is horizontal. By
    \cref{eq:maximal_surface_parametrization}, we have that
    \[
        \partial_uS_{(1,\alpha,\beta)}(t,0)
        =
        \left(t\cosh\beta,\alpha+t\sinh\beta,0\right).
    \]
    Thus, as in \cref{eq:general_characteristic_equation}, the horizontality
    condition is
    \[
        t\cosh\beta=\omega\bigl(\alpha+t\sinh\beta\bigr),
        \qquad \omega\in\{-1,1\},
    \]
    and hence
    \[
        t_\omega=\omega\alpha e^{\omega\beta}.
    \]
    If \(\alpha=0\), the only solution is \(t=0\), and
    \cref{eq:maximal_surface_parametrization} gives
    \(S_{(1,0,\beta)}(0,u)=(0,0,0)\) for every \(u\). This is a singular point
    of the parametrisation, see the proof of
    \cref{prop:non_degeneracy_surfaces}, and there are no characteristic
    points.

    Suppose now that \(\alpha\neq0\). Substituting \(t=t_\omega\) into
    \cref{eq:maximal_surface_parametrization} gives
    \begin{equation}\label{eq:maximal_surface_characteristic_parametrized}
        \begin{aligned}
            S_{(1,\alpha,\beta)}(t_\omega,u)
             & =
            B_u\left(
            \frac{\alpha}{2}\left(e^{2\omega\beta}+1\right),\,
            \frac{\omega\alpha}{2}\left(e^{2\omega\beta}+1\right),\,
            \frac{\omega\alpha^2}{4}\left(e^{2\omega\beta}+1\right)
            \right) \\
             & =
            \left(
            \frac{\alpha}{2}\left(e^{2\omega\beta}+1\right)e^{\omega u},\,
            \frac{\omega\alpha}{2}\left(e^{2\omega\beta}+1\right)e^{\omega u},\,
            \frac{\omega\alpha^2}{4}\left(e^{2\omega\beta}+1\right)
            \right),
        \end{aligned}
    \end{equation}
    where in the second equality we used
    \(B_u(a,\omega a,z)=(ae^{\omega u},\omega ae^{\omega u},z)\).
    Since
    \[
        s=\frac{\abs{\alpha}}{2}\left(e^{2\omega\beta}+1\right)e^{\omega u}
    \]
    ranges over \((0,+\infty)\) as \(u\) ranges over \(\mathbb R\),
    \cref{eq:maximal_surface_characteristic_parametrized} is exactly the
    half-line in \cref{eq:maximal_surface_characteristic_half_lines}.
    Its horizontal velocity is \(\sign(\alpha)(1,\omega)\), which is
    future-directed when \(\alpha>0\) and past-directed when \(\alpha<0\).
\end{proof}

We now give the isometry classification for maximal surfaces.

\begin{proposition}\label{prop:maximal_surfaces_isometric}
    Two maximal surfaces \(S_{(1,\alpha,\beta)}\) and
    \(S_{(1,\alpha',\beta')}\) are isometric if and only if
    \[
        \abs{\alpha}\cosh\beta=\abs{\alpha'}\cosh\beta'.
    \]
    Moreover, if \(\alpha\alpha'\neq0\), then any isometry between them is
    time-preserving when \(\alpha\alpha'>0\), and time-inverting when
    \(\alpha\alpha'<0\).
\end{proposition}
\begin{proof}
    We first treat the singular case. If \(\alpha=0\), then the image of
    \(S_{(1,0,\beta)}\) is independent of \(\beta\) by
    \cref{prop:non_degeneracy_surfaces}. If \(\alpha'\neq0\), then
    \(S_{(1,\alpha',\beta')}\) is regular by the same proposition, and it
    cannot be isometric to the singular surface.

    Assume from now on that \(\alpha\alpha'\neq0\). Suppose first that the two
    surfaces are isometric. The isometry must send the characteristic set of the
    first surface to the characteristic set of the second one. By
    \cref{prop:maximal_surfaces_characteristic_set}, the characteristic
    half-lines are future-directed if \(\alpha>0\) and past-directed if
    \(\alpha<0\), and similarly for \(\alpha'\). Therefore the isometry is
    time-preserving when \(\alpha\alpha'>0\), and time-inverting when
    \(\alpha\alpha'<0\).

    The closures of the two characteristic half-lines of each surface start on
    the \(z\)-axis. Therefore the translation part of the isometry must be vertical.
    Since Lorentz maps and vertical translations preserve \(x^2-y^2\), the
    maximum value of \(x^2-y^2\) on the two surfaces must agree. On
    \(S_{(1,\alpha,\beta)}\), this quantity can be computed at \(u=0\):
    \[
        x^2-y^2
        =
        \bigl(\alpha+t\sinh\beta\bigr)^2-t^2\cosh^2\beta
        =
        -\bigl(t-\alpha\sinh\beta\bigr)^2+\alpha^2\cosh^2\beta.
    \]
    Its maximum is \(\alpha^2\cosh^2\beta\). The same argument for
    \(S_{(1,\alpha',\beta')}\) gives the maximum
    \((\alpha')^2\cosh^2\beta'\). Hence
    \[
        \abs{\alpha}\cosh\beta=\abs{\alpha'}\cosh\beta'.
    \]

    Conversely, assume
    \(\abs{\alpha}\cosh\beta=\abs{\alpha'}\cosh\beta'\). Since \((x,y,z)\mapsto(-x,-y,z)\) commutes with boosts and sends
    \(\gamma_{(1,\alpha',\beta')}(t)\) to
    \(\gamma_{(1,-\alpha',\beta')}(-t)\), it maps
    \(S_{(1,\alpha',\beta')}\) onto \(S_{(1,-\alpha',\beta')}\). Therefore, after possibly applying this
    isometry, it is enough to prove the claim when
    \[
        \alpha\cosh\beta=\alpha'\cosh\beta'.
    \]
    A vertical translation by $(0,0,-(\alpha^2\cosh\beta\sinh\beta)/2)$
    followed by the change of variables
    \[
        s=t-\alpha\sinh\beta,
        \qquad
        v=u+\beta,
    \]
    sends \(S_{(1,\alpha,\beta)}\) to
    \(S_{(1,\alpha\cosh\beta,0)}\):
    \[
        \left(0,0,-\frac{\alpha^2\cosh\beta\sinh\beta}{2}\right)
        \ast S_{(1,\alpha,\beta)}(t,u)
        =
        S_{(1,\alpha\cosh\beta,0)}(s,v),
    \]
    where the first two components are simplified using the addition formulae
    for \(\cosh(v-\beta)\) and \(\sinh(v-\beta)\).
    Applying the same reduction to \(S_{(1,\alpha',\beta')}\) gives the same
    surface, because \(\alpha'\cosh\beta'=\alpha\cosh\beta\).
\end{proof}

For maximal surfaces, fixing the time orientation also determines when boost
and rotation isometries can occur.

\begin{corollary}\label{coro:maximal_surface_boost_isometric_iff_rotation_isometric}
    Let \(S_{(\varepsilon,\alpha,\beta)}\) and
    \(S_{(\varepsilon',\alpha',\beta')}\) be two maximal surfaces. They are
    isometric by a time-preserving (respectively, time-inverting) boost if and
    only if they are isometric by a time-preserving (respectively,
    time-inverting) rotation.
\end{corollary}
\begin{proof}
    By \cref{prop:maximal_surfaces_epsilon_sign_isometric}, and possibly taking
    inverse isometries,
    \(S_{(\varepsilon,\alpha,\beta)}\) can be mapped onto
    \(S_{(-\varepsilon,\alpha,\beta)}\) both by a time-preserving rotation and
    by a vertical translation, the latter being a time-preserving boost.

    Suppose, for instance, that
    \(S_{(\varepsilon,\alpha,\beta)}\) is isometric to
    \(S_{(\varepsilon',\alpha',\beta')}\) by a time-preserving boost \(F\). Let
    \(R\) be a time-preserving rotation sending
    \(S_{(\varepsilon',\alpha',\beta')}\) to
    \(S_{(-\varepsilon',\alpha',\beta')}\), and let \(V\) be a vertical
    translation sending \(S_{(-\varepsilon',\alpha',\beta')}\) back to
    \(S_{(\varepsilon',\alpha',\beta')}\). Then \(V\circ R\circ F\) is a
    time-preserving rotation from \(S_{(\varepsilon,\alpha,\beta)}\) to
    \(S_{(\varepsilon',\alpha',\beta')}\).

    The converse is identical, with boost and rotation interchanged. The same
    argument applies to time-inverting isometries, since the auxiliary maps
    \(R\) and \(V\) are time-preserving and therefore do not change whether the
    resulting isometry preserves or reverses time orientation.
\end{proof}

In the following proposition, we show that all regular maximal surfaces are not achronal.

\begin{proposition}\label{prop:maximal_surfaces_non_achronal}
    For every \(\alpha\neq0\), the maximal surface \(S_{(\varepsilon,\alpha,\beta)}\)
    is not achronal.
\end{proposition}
\begin{proof}
    By
    \cref{prop:maximal_surfaces_epsilon_sign_isometric,prop:maximal_surfaces_isometric},
    and since achronality is invariant under isometries, it is enough to prove
    the claim for \(S_{(1,\alpha,0)}\) with \(\alpha>0\), using
    \cref{eq:maximal_surface_parametrization}.
    Setting \(p=S_{(1,\alpha,0)}(0,0)=(\alpha,0,0)\), it is enough to show that the surface intersects \(I^-(p)\), since then it
    contains a point chronologically preceding \(p\). Fix \(c<0\) and, for \(u\neq0\), choose
    \[
        t(u)=\frac{c-\alpha\cosh u}{\sinh u},
    \]
    so that the \(x\)-coordinate of \(S_{(1,\alpha,0)}(t(u),u)\) is \(c\). Then
    \[
        S_{(1,\alpha,0)}(t(u),u)
        =
        \left(
        c,\,
        \frac{c\cosh u-\alpha}{\sinh u},\,
        \frac{\alpha}{2}\frac{c-\alpha\cosh u}{\sinh u}
        \right).
    \]
    Letting \(u\to-\infty\), these points converge to
    \[
        q(c)=\left(c,-c,\frac{\alpha^2}{2}\right),
    \]
    so \(q(c)\) belongs to the closure of the surface.

    We now check that \(q(c)\in I^-(p)\) for \(|c|>0\) sufficiently large. By left-invariance, this is equivalent to
    \(q(c)^{-1}\ast p\in I^+(0)\). A direct computation gives
    \[
        q(c)^{-1}\ast p
        =
        \left(
        \alpha-c,\,
        c,\,
        -\frac{\alpha^2+\alpha c}{2}
        \right).
    \]
    Its \(x\)-coordinate is positive because \(c<0\). Moreover, for
    \(c<-\alpha\), we have \(\alpha^2+\alpha c<0\), and
    \[
        -x^2+y^2+4\abs{z}
        =
        -(\alpha-c)^2+c^2-2(\alpha^2+\alpha c)
        =
        -3\alpha^2<0.
    \]
    By \cref{eq:future_origin_heisenberg}, \(q(c)^{-1}\ast p\in I^+(0)\), and
    therefore \(q(c)\in I^-(p)\).

    Since \(I^-(p)\) is open and \(q(c)\) lies in the closure of the surface,
    the surface itself intersects \(I^-(p)\). Thus it contains a point
    chronologically preceding \(p\), and is not achronal.
\end{proof}

Having studied the singular surfaces with positive mean curvature, we now return to the singular maximal surface. By
\cref{prop:non_degeneracy_surfaces}, all maximal surfaces with parameter
\(\alpha=0\) have the same image,
\[
    \Sigma_0=\{(x,y,0):\abs y>\abs x\}\cup\{(0,0,0)\}.
\]
This image is singular at the origin, for the same topological reason as in
\cref{prop:real_singularity}: removing the origin separates the two components
\(\{y>\abs x\}\) and \(\{y<-\abs x\}\). One could ask whether the singularity can be removed by extending the image near
the origin. However, any smooth extension connecting the two components
\(\{y>\abs x\}\) and \(\{y<-\abs x\}\) would have to contain points in the
missing boundary directions \(\{\abs y=\abs x\}\). These directions are null in
the Minkowski \((x,y)\)-plane, since \(x^2-y^2=0\). Thus such an extension would
have null tangent directions and would no longer be spacelike.

Regular maximal surfaces are never achronal by
\cref{prop:maximal_surfaces_non_achronal}. In contrast, the singular maximal
surface is acausal.

\begin{proposition}\label{prop:singular_maximal_surface_acausal}
    The singular maximal surface \(\Sigma_0\) is acausal.
\end{proposition}
\begin{proof}
    Let \(p=(a,b,0)\) and \(q=(c,d,0)\) be two points of \(\Sigma_0\), and
    suppose that \(p\leq q\), i.e. \(p^{-1}\ast q\in J^+(e)\). Since
    \[
        p^{-1}\ast q
        =
        \left(c-a,\ d-b,\ \frac12(bc-ad)\right),
    \]
    \cref{eq:future_origin_heisenberg} gives
    \begin{equation}\label{eq:singular_maximal_future_inequality}
        -(c-a)^2+(d-b)^2+2\abs{ad-bc}\leq0.
    \end{equation}
    We prove that \(p=q\). If \(p=e\), then
    \cref{eq:future_origin_heisenberg} applied to \(q\) forces
    \(\abs c\geq\abs d\), so \(q=e\). The same argument holds if \(q=e\), so
    we may assume that both points lie in
    \(\{(x,y,0):\abs y>\abs x\}\).

    The two points cannot lie in different components of this set. Indeed,
    suppose without loss of generality that \(b>\abs a\) and \(d<-\abs c\). Then
    \[
        \abs{d-b}=b-d>\abs a+\abs c\geq\abs{c-a}, \qquad \text{ and } \qquad -(c-a)^2+(d-b)^2+2\abs{ad-bc}>0,
    \]
    contradicting \cref{eq:singular_maximal_future_inequality}.

    We therefore assume that \(p\) and \(q\) lie in the same connected
    component.
    Applying the timed rotation \((x,y,z)\mapsto(x,-y,-z)\) if
    necessary, we may assume \(b>\abs a\) and \(d>\abs c\). Set
    \[
        R=b+a,\qquad S=b-a,\qquad R'=d+c,\qquad S'=d-c.
    \]
    Then \(R,S,R',S'>0\), and a computation gives
    \begin{equation}\label{eq:singular_maximal_causal_quantity}
        \begin{aligned}
            -(c-a)^2+(d-b)^2+2\abs{ad-bc}
             & =
            (R'-R)(S'-S)+\abs{RS'-SR'} \\
             & =
            RS+R'S'-2\min\{RS',SR'\}.
        \end{aligned}
    \end{equation}
    Since
    \[
        \min\{RS',SR'\}\leq\sqrt{RSR'S'}\leq\frac12(RS+R'S'),
    \]
    the quantity in \cref{eq:singular_maximal_causal_quantity} is
    non-negative, and equality holds only when
    \(R=R'\) and \(S=S'\). Together with
    \cref{eq:singular_maximal_future_inequality}, this forces equality in
    \cref{eq:singular_maximal_causal_quantity}. Thus \(R=R'\) and \(S=S'\), and
    so \(a=c\) and \(b=d\).
\end{proof}

\printbibliography[heading=bibintoc]

\end{document}